\documentclass{article}
\usepackage[utf8]{inputenc}
\usepackage[english]{babel}
\usepackage{comment}

\usepackage[dvipsnames]{xcolor}

\definecolor{myorange}{RGB}{255, 176, 1}
\definecolor{mygreen}{RGB}{55, 184, 78}
\definecolor{mycyan}{RGB}{51, 255, 255}

\usepackage{amsfonts}
\usepackage{amsthm,amsmath,amssymb}

\usepackage{tabularx}

\usepackage{authblk}

\newcommand{\RR}{{\mathbb{R}}}
\newcommand{\CC}{{\mathbb C}}

\newcommand{\om}{\Omega}
\newcommand{\oms}{\Omega^{\sharp}}

\newcommand{\dive}{\text{\normalfont div}}

\usepackage{amsthm}

\newtheorem{theorem}{Theorem}[section]
\newtheorem{lemma}[theorem]{Lemma}
\newtheorem{proposition}[theorem]{Proposition}

\newtheorem{remark}[theorem]{Remark}
\newtheorem{definition}[theorem]{Definition}

\numberwithin{equation}{section}

\usepackage{graphicx}
\usepackage{caption}
\usepackage{subcaption}
\usepackage{enumitem}
\usepackage{hyperref} 

\begin{document}

\title{Lipschitz Stability for Polyhedral Elastic Inclusions \\
	{}from Partial Data}
\author[1]{Aspri Andrea} 
\author[2]{Beretta Elena}
\author[3]{Francini Elisa}
\author[4]{Morassi Antonino}
\author[5]{Rosset Edi}
\author[5]{     Sincich Eva}
\author[3]{Vessella Sergio}

\affil[1]{Department of Mathematics, Universit\`a degli Studi di Milano}
\affil[2]{Department of Mathematics, NYU Abu Dhabi}
\affil[3]{Department of Mathematics and Computer Science, Universit\`a degli Studi di Firenze}
\affil[4]{Dipartimento Politecnico di Ingegneria e Architettura, Universit\`a degli Studi di Udine}
\affil[5]{Dipartimento di Matematica, Informatica e Geoscienze, Universit\`a degli Studi di Trieste}
\date{August 8, 2025}
\maketitle

\begin{abstract}

The paper deals with the inverse problem of determining a polyhedral inclusion compactly contained
in an elastic body {}from boundary measurements of traction and displacement taken on an open portion of the boundary. Both the inclusion and the body are made of homogeneous isotropic material. Under suitable assumptions on the geometry of the unknown inclusion, we prove a constructive Lipschitz stability estimate {}from the local Dirichlet-to-Neumann map. 
\end{abstract}

\medskip

\medskip


\medskip

\medskip


\section{Introduction}
\label{Intro}

An important class of inverse boundary-value problems in linear elasticity aims to identify hidden or unknown inclusions inside an elastic body by using measurements taken at its boundary. Such problems arise in several contexts, including seismic exploration (where surface data reveal subsurface structures), see for example \cite{symes2009seismic}, \cite{brown2005variational} and \cite{shi2020numerical}, geophysics \cite{segall2010earthquake},
non-destructive material testing (for locating cracks, defects, or voids) see for example \cite{ammari2015mathematical}, and medical imaging techniques such
as elastography, see for example \cite{barbone2004elastic}. For
a topical review of inverse problems in elasticity and their applications, we refer to \cite{bonnet2005inverse}.

In this paper we address the inverse problem of determining a polyhedral elastic inclusion $D$ hidden in a three-dimensional body $\Omega$, $D \subset \subset \Omega$, by means of displacement and traction boundary measurements performed on an open portion $\Sigma$ of the boundary $\partial \Omega$. The linearly elastic material inside and outside the inclusion is assumed to be homogeneous and isotropic, with given elastic tensors $\CC^i$ and $\CC^e$ respectively, and volume forces are absent. Within the framework of linear elasticity, the displacement field $u \in H^1 (\Omega)$ for prescribed Dirichlet data $\psi \in H^{1/2}_{co}(\Sigma)$ satisfies the boundary value problem

\begin{equation}
	\label{eq:Introd-bvpbm}
	\begin{cases}
		\textup{div}
		(
		(\CC^e + (\CC^i - \CC^e) \chi_D	)
		\nabla u
		)=0, & \text{in}\ \Omega,
		\\ 
		u = \psi, & \text{on}\ \partial \Omega. 
	\end{cases}
\end{equation}
Our goal is to derive a stability estimate for the determination of the polyhedron $D$ {}from the local Dirichlet-Neumann map $\Lambda_D^\Sigma :  H^{1/2}_{co}(\Sigma) \rightarrow H^{-1/2}_{co}(\Sigma)$ defined by $\Lambda_D^\Sigma(\psi) = (  \CC^e \nabla u  )n$, where $n$ is the unit outer normal to $\Omega$.

In particular, assuming that $D$ belongs to an admissible class of polyhedra $\mathcal P$  and requiring monotonicity of the elastic tensor jump across $\partial D$ 
our main result is a constructive Lipschitz stability estimate; see Theorem \ref{th:main} for a precise statement.

Research on this class of geometric inverse problems began with Isakov \cite{v.isakov1988}, who showed how to identify a smooth inclusion inside an isotropic electrostatic conductor {}from full boundary data—namely, all possible pairs of the electric potential and current flux measured on the boundary. His uniqueness proof combined the Runge approximation theorem with special solutions that possess Green’s-function–type singularities.
The stability counterpart of Isakov’s result was later established by Alessandrini and Di Cristo in \cite{g.alessandrinim.dicristo2005}, who derived a conditional logarithmic stability estimate for inclusions $D$ of $C^{1,\alpha}$ class. Their argument still exploited singular solutions, but replaced the Runge‐approximation method with quantitative unique-continuation estimates. These estimates have since become a standard ingredient in stability analyses for more intricate equations and systems (see, for instance, \cite{FS2023}, \cite{ADCMR}, \cite{MR2016}).

In particular,  a log-type stability estimate based on a local
Dirichlet-to-Neumann map for general elastic inclusions $D$ of $C^{1,\alpha}$ class, was derived in \cite{ADCMR}. An extension to inhomogeneous background materials appeared in \cite{MR2016}.
Because of well-known results established in the context of conductivity (\cite{DiCristoRondi}, and more recently \cite{koch2025instability}), which involves a single elliptic equation rather than an elliptic system, the log-type estimate can be considered the optimal rate of convergence. 
Logarithmic stability estimates confirm that the inverse problem is extremely ill-conditioned: even modest data errors are strongly amplified, making practical reconstructions unreliable. 
When the physics of the problem permits, this difficulty can be mitigated by imposing geometric constraints on the admissible inclusions, reducing the inverse problem to the determination of finitely many unknown parameters. A pioneering step in this direction was the determination of piecewise-constant coefficients under a known finite partition of the domain, first for the conductivity case (\cite{druskin1982uniqueness} for the uniqueness, \cite{AlVes} for the Lipschitz stability) and later for the elastic case \cite{BFV-IPI-2014}; see also \cite{BerFraMorRosVes14} for extensions to non-flat interfaces. In the framework of linear elasticity, we want to highlight that other authors have derived Lipschitz stability  using localized potentials instead of singular solutions\cite{eberle2021lipschitz}.

Building on the seminal work of Bacchelli and Vessella \cite{BaV2006}, we note that quantitative Lipschitz stability for broad classes of nonlinear inverse problems rests the following conditions:
the unknown quantity can be described by a finite set of parameters; these parameters are confined to a compact admissible set $K$; the forward map $\mathcal{F}$ taking the unknown into the measurements is injective on this compact set and the inverse map of $\mathcal{F}_{|K}$ is uniformly continuous with a given modulus of continuity; finally, the forward map is continuously Fr\'echet  differentiable and the Fr\'echet derivative of the forward map is bounded {}from below.

Adopting the same overall strategy though with significant modifications, in the recent paper \cite{AspBerFraVes22}  a Lipschitz stability estimate in terms of the local Dirichlet-to-Neumann map, for the electrical analogue of problem \eqref{eq:Introd-bvpbm}, has been derived: recovering a hidden polyhedral inclusion in a conductor {}from boundary measurements of voltage and current.
We draw on their framework, but our Lipschitz stability result for linear elasticity is not at all a straightforward extension of the conductivity case. 

Elliptic systems such as the Lam\'e equations introduce new difficulties that call for additional, genuinely original tools. In particular, we develop techniques tailored to isotropic elasticity that have no counterpart in the scalar setting. Key innovations and their roles in the proof are outlined in the remarks that follow.

\begin{enumerate}
\item {\em Construction of singular solutions}\\
As already pointed out in \cite{BFV-IPI-2014} and \cite{BF}, a relevant difference with respect to the scalar
case treated in \cite{AspBerFraVes22} is the issue of the existence of singular solutions. Indeed, in the case of strongly elliptic systems with $L^{\infty}$ coefficients in dimension three, the existence of
the fundamental solution and Green's function cannot be inferred without
additional assumptions. In \cite{hofmann2007green} the authors prove their existence under the
additional information that weak solutions of the system satisfy De Giorgi-Nash
type local H\"older estimates. It is clear that, in the case of a polyhedral elastic inclusion,
solutions might not enjoy H\"older regularity at edges and vertices. On the other hand, to obtain our result, it is enough to construct singular solutions and analyze their
behavior at a given positive distance {}from edges. 
Therefore, following the ideas in \cite{BFV-IPI-2014}, we construct Green's functions having a pole that does not belong to the edges of the polyhedron.

\item{\em Derivation of a first logarithmic stability estimate}\\
A first application of the singular solutions described above is the rough stability estimate with logarithmic rate of convergence stated in Theorem \ref{Teo:stima_log}. 
This estimate, which is new for this class of inclusions that are not regular enough to apply the results in \cite{ADCMR},  ensures, as a byproduct, that polyhedra belonging to our admissible class $\mathcal{P}$, with sufficiently close Dirichlet-to-Neumann maps, have the same number of vertices, faces, and edges.

\item{\em Existence and regularity of Gateaux derivative of the Dirichlet-to-Neumann map}\\
The above property, in turn, allows us to define a suitable homotopy between any two polyhedra $D_0, D_1 \in \mathcal P$; see Proposition \ref{vectorfield} and Proposition \ref{prop:property_Phi} for precise statements, and Appendix \ref{Appendix} for an exhaustive and detailed construction of the homotopy.
Next, a crucial and challenging step is to prove the existence and regularity of the Gateaux derivative of the Dirichlet-to-Neumann map at $D_0 \in \mathcal{P}$. Precisely, we derive a distributed representation formula for the derivative by adapting tools and ideas taken {}from shape optimization \cite{henrot2018shape}, \cite{laurain2020distributed}, \cite{ammari2010reconstruction}, theory of composite materials and homogenization \cite{FraMur86}, and theory of small volume fraction asymptotic expansions \cite{beretta2012small}, \cite{beretta2006asymptotic}. In this respect, the Lipschitz continuity of the Gateaux derivative stated in Proposition \ref{prop:Continuità-derivata-di-F} improves the analogous estimate in \cite{AspBerFraVes22} and is based on more elementary arguments.

\item {\em Boundary formula of Gateaux derivative of the Dirichlet-to-Neumann map}\\
Another key point is the passage {}from the distributed derivative to the boundary derivative. This step is particularly challenging since it requires to express the integrand function appearing in the distributed derivative in terms of the divergence of a suitable vector field $b$, a property that was not at all predictable in the context of elasticity; see Lemma \ref{lem:LBD10.1}.
Starting {}from the field $b$, a further important step consists in writing the jump of $b\cdot n$ across the boundary of the polyhedron $D$ in terms of a boundary representation formula -- away {}from vertices and edges -- which involves only the value of the solutions outside $D$ and is expressed in terms of the so-called {\it elastic moment tensor} $\mathbb{M}$. This fourth-order tensor encodes the jump in strains across the interface and exhibits both minor and major symmetry, as well as positivity; see Lemma \ref{lemma:elast_mom_tens} and Proposition \ref{prop:LBD15.1}.

\item {\em Lower bound of the Gateaux derivative }\\
Finally, to derive the lower bound of the Dirichlet-to-Neumann map's derivative, we use, among other tools, a constructive estimate {}from below of the gradient of the biphase fundamental solution for the homogeneous Lam\'e system near the singularity, see Lemma \ref{lem:LBD77.1}. This estimate requires an accurate elaboration of its explicit expression, which was derived by Rongved in  \cite{rongved1955force}. 
\end{enumerate}

Since our elastic inclusion depends on finitely many parameters, it is natural to ask whether it can be recovered {}from finitely many measurements. In this context, we would like to mention that the results obtained recently in \cite{alberti2022inverse} in an abstract setting, and where Lipschitz continuity {}from finitely many measurements has been proved if the unknown belongs to a suitable finite-dimensional nonlinear manifold, seem not to include the case of general polygonal and polyhedral conductivity or elastic inclusions. 
To our knowledge, uniqueness and stability for general polyhedral elastic inclusions {}from finitely many measurements are an open issue. {\color{green}}We point out that a Lipschitz stability estimate {}from the Dirichlet-to-Neumann map could be a first step in this direction (see \cite{AS2021}).

We want to conclude by observing that quantitative Lipschitz stability estimates are crucial for reconstruction. Indeed, they allow to prove convergence of iterative methods for the effective reconstruction of the approximated solution to the inverse problem, \cite{deHoopetal}, \cite{deHQS}.
In \cite{BerMicPerSan18}, a reconstruction method based on a shape-optimization algorithm has been successfully implemented  enabling the reconstruction of piecewise-constant conductivity that jumps across a polygonal inclusion {}from current-to-voltage boundary measurements. In particular, by employing the shape derivative of a misfit cost functional and adjusting at each step the number of vertices of the iterates to impose some regularization, the reconstruction algorithm converges rapidly to the solution even when starting far {}from it. 
A similar approach has been used in \cite{dambrine2024robust} for smooth elastic inclusions.

The organization of the paper is as follows. In Section
\ref{Notation-a-priori} we introduce some notation and the a-priori
information needed for our stability result. The formulation of the inverse problem and the main result, Theorem \ref{th:main}, are stated in Section \ref{Main-result}. Some useful geometrical results are collected in Section \ref{geometry}. In Section \ref{sec:rough_log_stab_est} we construct ad hoc Green's functions and derive the rough logarithmic stability estimate, see \eqref{eq:stima-grezza-dH}--\eqref{eq:stima grezza 2 log}. Section \ref{sec:domain derivative} is devoted to the domain derivative of the Dirichlet-to-Neumann map. Precisely, in subsection \ref{omotopia} we introduce a homotopy map between the polyhedra $D_0$ and $D_1$; in subsection \ref{Gateaux} we compute both the distributed and boundary Gateaux derivative; the lower bound of the Dirichlet-to-Neumann map derivative is derived in subsection \ref{lower-bound}.  The proof of the main Theorem \ref{th:main} is presented in Section \ref{sec:proof_main_th}. Finally, in Section \ref{Appendix} we give a detailed construction of the homotopy map between $D_0$ and $D_1$.

\section{Notation and a priori information}
\label{Notation-a-priori}

\subsection{Notation}
We begin by introducing the notation and recalling some essential definitions that will be used throughout.

Let us denote by $\RR^3_+ = \{ x=(x_1,x_2,x_3) \in \RR^3 \ | \ x_3 >0  \}$, $\RR^3_- = \{ x=(x_1,x_2,x_3) \in \RR^3 \ | \ x_3 <0  \}$. 

Given $x \in \RR^3$, we shall denote $x=(x',x_3)$, where $x'=(x_1,x_2) \in \RR^2$, $x_3 \in \RR$.

For a point $P\in\mathbb{R}^3$ and $R>0$, we define the ball centered at $P$ with radius $R$ as
\begin{equation*}
B_{R}(P):=\{x\in\mathbb{R}^3\ |\ |x-P|<R\}
\end{equation*}
and, similarly, we denote by $B'_{R}(P)$ a disc centered at $P$ with radius $R$ lying in a plane that will be specified in the context. When the center $P$ is at the origin, we omit the reference to $P$ for simplicity.

For given $r_0 >0$, ${M_0 \geq 1}$, we denote 

\begin{equation*}
	R_{r_0, M_0} =\{x=(x',x_3)\ |\ x'\in[-r_0, r_0]^2, x_3\in [-M_0r_0, M_0r_0]\}.
\end{equation*}

\noindent When representing locally a boundary as a graph, we use
the following definition.

\begin{definition}
	\label{def:reg_bordo} (Lipschitz regularity)
	Let $U$ be a bounded domain in ${\RR}^{3}$. We say that a portion $S$ of
	$\partial U$ is of Lipschitz class with
		constants $r_{0}$, $M_{0}$, $M_0\geq 1$, if, for any $P \in S$, there
	exists a rigid transformation of coordinates under which we have
	$P=0$ and
	\begin{equation*}
		U \cap R_{r_0,2M_0}=\{x \in R_{r_0,2M_0} \ | \
		x_{3}>g(x_1,x_2)
		\},
	\end{equation*}
	where $g$ is a Lipschitz continuous function on
	$[-r_0,r_0]^2$
	satisfying
	\begin{equation*}
		g(0)=0,
	\end{equation*}
	\begin{equation*}
		\|g\|_{{C}^{0,1}([-r_0,r_0]^2)} \leq M_0r_0.
	\end{equation*}

\end{definition}

We use the convention to normalize all norms in such a way that
their terms are dimensionally homogeneous with the $L^\infty$-norm. For
instance, the norm appearing above is meant as follows
\begin{equation*}
	\|g\|_{{C}^{0,1}([-r_0,r_0]^2)} = \| g\|_{L^\infty([-r_0,r_0]^2)}+r_0
	\sup_ {\overset{\scriptstyle x', \, y'\in [-r_0,r_0]^2}{\scriptstyle
			x'\neq y'}}\frac{|g(x') -  g(y')|}{|x'-y'|}.
\end{equation*}
Similarly, for a vector function $w \in H^1(U)$, $w :U \subset \RR^3\rightarrow \RR^3$, we set
\begin{equation*}
	\|w\|_{H^1(U)}=r_0^{-3/2} \left ( \int_U |w|^2
	+r_0^2\int_U|\nabla w|^2 \right )^{1/2},
\end{equation*}
and so on for boundary and trace norms. 

We shall denote $w_{i,j}= \frac{\partial w_i}{\partial x_j}$ and similarly for higher derivatives.

The Hausdorff distance between two bounded and closed sets  $C$ and $D$ in $\mathbb{R}^3$ is defined as
\begin{equation*}
    d_H(C,D)=\max\left\{\max_{x\in C}\hbox{dist}(x,D),\max_{x\in D}\hbox{dist}(x,C)\right\}.
\end{equation*}
We denote the interior of a set $ C $ by $ \text{Int}(C) $.

Now, let $ F_1 $ and $ F_2 $ be two closed, simply connected, and bounded flat surfaces in $ \mathbb{R}^3 $. If $ F_1 \cap F_2 =: \sigma $, where $ \sigma $ is a non-empty segment, we denote the interior of $ F_1 $ relative to the plane containing it by $ \text{Int}_{\mathbb{R}^2}(F_1) $, and the interior of $ \sigma $ relative to the line containing it by $ \text{Int}_{\mathbb{R}}(\sigma) $.\ \\

We denote by $ \mathbb{M}^{m \times n} $ the space of $ m \times n $ real-valued matrices, and by $ \mathcal{L}(X, Y) $ the space of bounded linear operators between Banach spaces $ X $ and $ Y $. When $ m = n $, we will also denote $ \mathbb{M}^n = \mathbb{M}^{n \times n} $.
For any pair of real $ n $-vectors $ a $ and $ b $, we denote by $ a \otimes b $ the $ n \times n $ matrix with entries
\begin{equation*}
    (a \otimes b)_{ij} = a_i b_j, \quad i, j = 1, \dots, n.
\end{equation*}
We denote by  $a\cdot b=\sum_{i=1}^3 a_i b_i$ and by $a \times b = \sum_{i,j,k=1}^3 \epsilon_{ijk} a_ib_j e_k$ the scalar and vectorial product between two vectors $a=\sum_{k=1}^3a_k e_k, b=\sum_{k=1}^3b_k e_k \in \RR^3$, respectively.
We recall that the Levi-Civita symbol $\epsilon_{ijk}$ is defined as follows
\begin{equation*}
	\epsilon_{ijk} = \left\{ \begin{array}{ll}
		+1,
		& \hbox{if}\ 	(i,j,k) \hbox{ is an even permutation of  } (1,2,3),\\
		-1,
		& \hbox{if}\ 	(i,j,k) \hbox{ is an odd permutation of  } (1,2,3),\\
		0,
		& \hbox{if at least two indexes are equal}.
	\end{array}\right.
\end{equation*}
For any $ 3 \times 3 $ matrices $ A $ and $ B $, and for every $ \mathbb{C} \in \mathcal{L}(\mathbb{M}^3, \mathbb{M}^3) $, we adopt the following notation:
\begin{equation*}
    (\mathbb{C}A)_{ij} = \sum_{k,l=1}^{3} \mathbb{C}_{ijkl} A_{kl},
\end{equation*}
\begin{equation*}
    A \cdot B = \sum_{i,j=1}^{3} A_{ij} B_{ij},
\end{equation*}
\begin{equation*}
    |A| = \left( A \cdot A \right)^{1/2},
\end{equation*}
where $ \mathbb{C}_{ijkl} $, $ A_{ij} $, and $ B_{ij} $ are the entries of $ \mathbb{C} $, $ A $, and $ B $, respectively.

Finally, we denote $\widehat{A}= \dfrac{1}{2}(A + A^T)$ and 
 $A^{-T}=(A^{-1})^T$ for every $3 \times 3$ invertible matrix $A$.

\subsection{A priori information}
Let us begin by defining the polyhedron, introducing the notation for its faces and vertices, and outlining the necessary a priori assumptions required to derive our main result.
\begin{definition}\label{def1}
A closed subset $ D \subset \mathbb{R}^3 $ is called a polyhedron if:
\begin{equation}
    D \text{ is homeomorphic to a closed ball in } \mathbb{R}^3;
\end{equation}
the boundary $ \partial D $ is given by
\begin{equation}
\partial D = \bigcup_{j=1}^H F^D_j,
\end{equation}
where each $ F_j^D $ is a closed, simply connected polygonal region in a plane (referred to as a face of $ D $), and
\begin{equation}
    \text{Int}_{\mathbb{R}^2}(F_i^D) \cap \text{Int}_{\mathbb{R}^2}(F_j^D) = \emptyset \quad \text{for } i \neq j.
\end{equation}
For $ i \neq j $, the intersection $ \sigma^D_{ij} = F_i^D \cap F_j^D $ is called an edge of $ D $ if $ \text{Int}_{\mathbb{R}}(\sigma^D_{ij}) \neq \emptyset $.
The non-empty intersection of two edges is referred to as a vertex $ V^D $ of $ D $.
\end{definition}

We consider a class of non-degenerate polyhedra. Let
\begin{equation*}
r_0, \quad M_1, \quad \theta_0, \quad M_0
\end{equation*}
be given positive constants, where $ \theta_0 \in (0, \pi/2) $, $M_0\geq 1$ and $M_1 >1$.

\begin{enumerate}[label=\textbf{(\textit{H\arabic*})}, ref=(H-\arabic*)]
\item\label{ass:domain}\textit{Domain.} Let $ \Omega \subset \mathbb{R}^3 $ be a bounded domain such that
\begin{equation}
	\label{diam}
    \hbox{diam}(\Omega) \leq M_1 r_0,
\end{equation}
where $ \mathrm{diam}(\Omega) $ denotes the diameter of $ \Omega $. 
Moreover, let $\Sigma$ be an open portion of  $\partial\om$ with size at least $r_0$, i.e.  we assume there exists at least one point $P_{\Sigma}\in \Sigma$ such that
\begin{equation}\label{ass:sigma}
    \hbox{dist}(P_{\Sigma},\partial\om \setminus\Sigma)\geq r_0.
\end{equation}

\item\textit{Polyhedral inclusion.}
We say that a polyhedron $ D \subset \Omega $ belongs to the class $ \mathcal{P} = \mathcal{P}(r_0, M_1, \theta_0, M_0) $ if the following conditions hold:
\begin{enumerate}[label=\textbf{(\textit{H2-\arabic*})}, ref=(H2-\arabic*)]
\item \textit{Strict inclusion}:
\begin{equation}
	\label{distfron}
    \mathrm{dist}(D, \partial \Omega) \geq r_0.
\end{equation}

\item\label{ass:dihedral_angle} \textit{Dihedral angle non-degeneracy}: The dihedral angle between intersecting faces at each edge of $ D $ is $ \alpha $, satisfying
\begin{equation}\label{angolifacce}
    \alpha \in (\theta_0, \pi - \theta_0) \cup (\pi + \theta_0, 2\pi - \theta_0).
\end{equation}

\item\label{ass:edge} \textit{Edge non-degeneracy}: Each edge $ \sigma^D $ of $D$ satisfies
\begin{equation}\label{lunghlati}
    \mathrm{length}(\sigma^D) \geq r_0.
\end{equation}

\item\label{ass:face_angle} \textit{Face angle non-degeneracy}: Every internal angle $ \beta $ of each face $ F^D $ satisfies
\begin{equation}\label{angolinterni}
    \beta \in (\theta_0, \pi - \theta_0) \cup (\pi + \theta_0, 2\pi - \theta_0).
\end{equation}

\item\label{ass:Lipschitz_regularity} \textit{Lipschitz regularity}:  $ \Omega \setminus D $ has Lipschitz boundary with constants $ r_0 $ and $ M_0 $ and each face $F$ of $D$ has Lipschitz boundary with constants $ r_0 $ and $ M_0 $ in the plane containing $F$.

\end{enumerate}
\begin{remark}\label{rem1}
As noted in \cite{AspBerFraVes22}, the number of vertices $ V^D $, edges $ \sigma_{ij}^D $, and faces $ F_j^D $ of a polyhedron $D \in \mathcal{P}$ is bounded above by a constant $N_0$, which only depends on $M_1$, $M_0$ and $\theta_0$.
\end{remark}

\item\label{ass:elastic_material}\textit{Elastic material} 

We shall consider isotropic materials with \textit{constant} elastic tensor $\CC$ of Cartesian components
\begin{equation}
	\label{Notaz-compon-cart-C}
	C_{ijkl} = \lambda \delta_{ij} \delta_{kl} + \mu ( \delta_{ki}\delta_{lj}+\delta_{li}\delta_{kj}  ),
\end{equation}
where the Lamé moduli $\lambda, \mu$ satisfy the strong convexity conditions 
\begin{equation}
	\label{Notaz-Lamé-strong-convex}
\mu \geq \alpha_0>0, \quad 2\mu +3\lambda \geq \gamma_0 >0.
\end{equation}
\end{enumerate}

\noindent Let us observe that \eqref{Notaz-Lamé-strong-convex} implies
\begin{equation*}
	\label{Notaz-mu+lambda-pos}
	\mu + \lambda  \geq\frac{\alpha_0 + \gamma_0}{3}.
\end{equation*}
Under these assumptions we have
\begin{equation}
	\label{Notaz-Lamé}
	\CC A = 2 \mu \widehat{A} + \lambda (\hbox{tr}A) I_3, \quad \hbox{for every } 3 \times 3 \ \hbox{matrix } A,
	\end{equation}
\begin{equation}
	\label{Notaz-simmetrie-C}
	C_{ijkl} = C_{klij}=C_{lkij}, \quad i,j,k,l=1,2,3,
\end{equation}
and the strong convexity condition \eqref{Notaz-Lamé-strong-convex} is equivalent to
\begin{equation}
	\label{Notaz-forte-convex-C}
	\CC A \cdot A \geq \xi_0 |A|^2, \quad \hbox{for every } 3\times3 \hbox{ symmetric matrix } A,
\end{equation}
where $\xi_0= \min \{ 2\alpha_0, \gamma_0  \}$.

The Poisson's coefficient $\nu$ is given by
\begin{equation}
	\label{Notaz-Poisson}
	\nu = \dfrac{\lambda}{2(\lambda + \mu)}
\end{equation}
and, as for most of the current materials, we assume
\begin{equation}
	\label{Notaz-Poisson-bounds}
	0\leq \nu < \dfrac{1}{2}.
\end{equation}
Combining \eqref{Notaz-mu+lambda-pos}, \eqref{Notaz-Poisson} and \eqref{Notaz-Poisson-bounds}, we have
\begin{equation}
	\label{Notaz-lambda-pos}
	\lambda \geq 0.
\end{equation}
We shall consider an elastic body $\Omega$ containing a polyhedral inclusion $D$ made by different homogeneous isotropic material. We denote by $\mathbb{C}^i$, $\mathbb{C}^e$ the elastic tensor inside and outside the inclusion $D$, respectively, and by 
\begin{equation*}\label{eq:elastic_tensor_D}
\mathbb{C}^D=\mathbb{C}^i\chi_{D}+\mathbb{C}^e\chi_{\Omega\setminus D}
\end{equation*} 
the elastic tensor in the whole domain $\Omega$. We assume that $\mathbb{C}^i$, $\mathbb{C}^e$ are of the form \eqref{Notaz-compon-cart-C} and satisfy \eqref{Notaz-Lamé-strong-convex}, with the Poisson's coefficient $\nu^i$, $\nu^e$ satisfying \eqref{Notaz-Poisson-bounds}. Obviously, we assume the so-called \textit{visibility} condition 
\begin{equation}\label{eq:visibility_cond}
(\lambda^i-\lambda^e)^2+(\mu^i-\mu^e)^2\geq \eta_0,
\end{equation}
for some $\eta_0>0$.
We point out that, in order to obtain the Lipschitz stability estimate given in Theorem \ref{th:main}, a stronger assumption will be required, namely:
\begin{subequations}
	\label{eq:main}
	\begin{align}
		\label{eq:monotony-Lamè-moduli-Ce-Ci} 
		&\hbox{either }&\mu^e-\mu^i\geq \alpha_1,\ \ 2(\mu^e-\mu^i)+3(\lambda^e-\lambda^i)\geq \gamma_1, 
		\Longleftrightarrow 
		\mathbb{C}^e - \mathbb{C}^i \  \hbox{strongly convex}, \\
			\label{eq:monotony-Lamè-moduli-Ci-Ce}
		&\hbox{or }&\mu^i-\mu^e\geq \alpha_1,\ \ 2(\mu^i-\mu^e)+3(\lambda^i-\lambda^e)\geq \gamma_1,
		\Longleftrightarrow 
		\mathbb{C}^i - \mathbb{C}^e \  \hbox{strongly convex},
	\end{align}
\end{subequations}
for positive constants $\alpha_1, \gamma_1$.

In the sequel, we will refer to the set of parameters
\begin{equation*} 
\quad M_1,\quad M_0, \quad \theta_0,\quad \lambda^i, \quad \mu^i, \quad \lambda^e, \quad \mu^e.
\end{equation*}
as the {\it a priori data}, whereas the dimensional parameter $r_0$ shall appear explicitly in all our estimates. In the sequel we shall denote by $C$ a positive constant which may change {}from line to line.

\section{Inverse problem formulation and main result}
\label{Main-result}

We consider boundary measurements supported on the set $\Sigma \subset \partial \Omega$ satisfying  \eqref{ass:sigma} and, consequently, we will employ a local Dirichlet-to-Neumann map.
\begin{definition}[The Local Dirichlet-to-Neumann map.] 
Let $ \Omega $ be a bounded domain of Lipschitz class, and let $ \Sigma $ be an open subset of $ \partial \Omega $. We define $ H^{1/2}_{co}(\Sigma) $ as the function space
\begin{equation*}
H^{1/2}_{co}(\Sigma) := \left\{ \varphi \in H^{1/2}(\partial \Omega) \mid \mathrm{supp}(\varphi) \subset \Sigma \right\},
\end{equation*}
and $ H^{-1/2}_{co}(\Sigma) $ as its topological dual. The dual pairing between $ H^{1/2}_{co}(\Sigma) $ and $ H^{-1/2}_{co}(\Sigma) $ is denoted by $ \langle \cdot, \cdot \rangle $, and is based on the $ L^2(\Sigma) $ scalar product:
\begin{equation*}
\langle f, g \rangle = \int_{\partial \Omega} fg, \quad \text{for all } f, g \in L^2(\partial \Omega).
\end{equation*}
Then, given $ \psi \in H^{1/2}_{co}(\Sigma) $, there exists a unique weak solution $ u \in H^1(\Omega) $ to the Dirichlet problem
\begin{equation}\label{eq:elastic_pb}
\begin{cases}
\textup{div}(\mathbb{C}^D \widehat{\nabla} u) = 0,& \text{in}\ \Omega,\\ u = \psi, & \text{on}\ \partial \Omega. 
\end{cases}
\end{equation}
The local Dirichlet-to-Neumann linear map $ \Lambda^\Sigma_{D} $ is defined as follows:
\begin{equation*}
\Lambda^\Sigma_{D} : \psi \in H^{1/2}_{co}(\Sigma) \mapsto (\mathbb{C}^D \nabla u) n \big|_\Sigma \in H^{-1/2}_{co}(\Sigma),
\end{equation*}
where $ n $ is the outer unit normal to $ \Omega $.
\end{definition}

The map $ \Lambda^\Sigma_{D} $ can be identified with the bilinear form on $ H^{1/2}_{co}(\Sigma) \times H^{1/2}_{co}(\Sigma) $ by
\begin{equation*}
\Lambda^\Sigma_{D}(\psi, \varphi) := \langle \Lambda^\Sigma_{D} \psi, \varphi \rangle = \int_\Omega \mathbb{C}^D \nabla u \cdot  \nabla v,
\end{equation*}
for all $ \psi, \varphi \in H^{1/2}_{co}(\Sigma) $, where $ u $ solves \eqref{eq:elastic_pb} and $ v $ is any function in $ H^1(\Omega) $ such that $ v = \varphi $ on $ \partial \Omega $.
The norm of the local Dirichlet-to-Neumann map in the space of linear operators $\mathcal{L}\left(H^{1/2}_{co}(\Sigma),H^{-{1/2}}_{co}(\Sigma)\right)$ is defined by 
\begin{equation*}
    \|\Lambda^{\Sigma}_{D}\|_{\star}:=\sup_
    {\overset{\scriptstyle \psi \in H^{1/2}_{co}(\Sigma)  }{\scriptstyle
    		\psi \neq 0}} 
    \dfrac{\|\Lambda^{\Sigma}_{D}\psi\|_{H^{-{1/2}}_{co}(\Sigma)}}{\|\psi\|
    	_{H^{{1/2}}_{co}(\Sigma)}}.
\end{equation*}
Below, we present the main theorem of this work. The proof is deferred to Section \ref{sec:proof_main_th}, as it relies on several intermediate results that will be introduced in the upcoming sections.
\begin{theorem}\label{th:main}
Let $\Omega$ be a bounded domain satisfying \ref{ass:domain} and let $D_0$, $D_1\in\mathcal{P}$. Let the elastic tensors $\mathbb{C}^{D_0}$, $\mathbb{C}^{D_1}$ satisfy assumptions \ref{ass:elastic_material}. Let $\Sigma$ be an open portion of $\partial\Omega$ satisfying \eqref{ass:sigma}. Then, there exists $C>0$ only depending on the a priori data such that 
\begin{equation}
    d_H(\partial D_0, \partial D_1)\leq Cr_0^2\|\Lambda^\Sigma_{D_0}- \Lambda^\Sigma_{D_1}\|_*.
\end{equation}
\end{theorem}

\section{Some useful geometrical results}
\label{geometry}

In this section, we gather some geometric results on polyhedra in the class $ \mathcal{P} $, deriving essentially {}from \cite{ADCMR,AspBerFraVes22,R08}.

We first establish the relation between the Hausdorff distance of two polyhedra in $ \mathcal{P} $ and the Hausdorff distance of their boundaries.
\begin{proposition}[Proposition 2.4 in \cite{R08}]
	\label{prop1}
Let $ D_0 $ and $ D_1 \in \mathcal{P} $. Then there exists a positive constant $ C_1 > 1 $, depending only on $M_0$, $M_1$ and $\theta_0$, such that 
\begin{equation}
	\label{12}
    C_1^{-1} d_H(\partial D_0, \partial D_1) \leq d_H(D_0, D_1) \leq C_1 d_H(\partial D_0, \partial D_1).
\end{equation}
\end{proposition}

For $ D_0 $ and $ D_1 \in \mathcal{P} $, let $ \mathcal{G} $ be the connected component of $ \Omega \setminus (D_0 \cup D_1) $ that contains $ \partial \Omega $, and let us define
\begin{equation*}
    \Omega_{\mathcal{G}} = \Omega \setminus \mathcal{G}.
\end{equation*}
Since the value of $ d_H(\partial D_0, \partial D_1) $ may be attained at a point on $ \partial D_0 \cup \partial D_1 $ that is not necessarily on $ \partial \Omega_{\mathcal{G}} $, and hence cannot be reached {}from $ \partial \Omega $ without crossing $ \partial D_0 \cup \partial D_1 $, we introduce a modified distance as defined in \cite{ADCMR}.

\begin{definition}\label{distmod}
The modified distance $ d_\mu(D_0, D_1) $ is defined as
\begin{equation*}
d_\mu(D_0, D_1) = \max \left\{ \max_{x \in \partial D_0 \cap \partial \Omega_{\mathcal{G}}} \mathrm{dist}(x, D_1), \max_{x \in \partial D_1 \cap \partial \Omega_{\mathcal{G}}} \mathrm{dist}(x, D_0) \right\}.
\end{equation*}
\end{definition}

It is important to note that this is not a metric. It is straightforward to show (see \cite{ADCMR}) that
\begin{equation*}
    d_\mu(D_0, D_1) \leq d_H(\partial D_0, \partial D_1).
\end{equation*}

In general, $ d_\mu $ does not provide an upper bound for the Hausdorff distance, but for polyhedra in the class $ \mathcal{P} $ the following result holds.

\begin{proposition}[Proposition 3.4 in \cite{AspBerFraVes22}]
	\label{prop2}
There exists a constant $ C_2 > 1 $, depending only on $M_0$, $M_1$ and $\theta_0$, such that for $ D_0, D_1 \in \mathcal{P} $,
\begin{equation}
d_H(\partial D_0, \partial D_1) \leq C_2 d_\mu(D_0, D_1).
\end{equation}
\end{proposition}

Next, let us introduce an augmented domain $\oms$. As in \cite[Section $6$]{ARRV}, we can construct an open set $\Omega_0$ exterior to $\Omega$ such that:
\begin{enumerate}
	\item $\Sigma_0=\partial \om_0 \cap \partial \om\subset\Sigma$;
	\item the size of $\Sigma_0$ is at least $ \frac{r_0}{2}$ (see definition given in \eqref{ass:sigma});
	\item $\oms=\om\cup\Sigma_0\cup\om_0$ is an open and connected set having Lipschitz boundary with constants $r_\sharp$, $M_\sharp$,  $M_\sharp \geq M_0$,
	\begin{equation}
		\label{eq:def-r-sharp}
		r_\sharp=\zeta r_0, \quad 0<\zeta \leq 1,
	\end{equation}
	with $\zeta$ and $M_\sharp$ only depending on $M_0$;
	\item there exists $P_0\in\om_0$ such that
	\begin{equation}\label{eq:ball}
		B_{2r_\sharp}(P_0)\subset\om_0
	\end{equation}
	and, denoting by $S$ the segment joining $P_0$ and $P_\Sigma$, 
	\begin{equation}\label{eq:tubo-P0-Psigma}
		\{ x \in \RR^3 \ | \ \hbox{dist}(x, S) \leq 2r_\sharp  \} \subset \Omega^\sharp.
	\end{equation}
\end{enumerate}
We extend the elasticity tensor {}from $\Omega$ to $\oms$ defining it as $\CC^e$ in $\Omega_0\cup \Sigma_0$.

\begin{proposition}
	\label{th:geometric_lemma}
Let $ D_0, D_1 \in \mathcal{P} $. There exist positive constants $ C_3 $, $ d<1 $, $ \chi $, depending only on $M_0$, $M_1$ and $\theta_0$, and a point $ P \in \partial D_0 \cap \partial \Omega_{\mathcal{G}} $ such that
\begin{equation}\label{1-19}
    C_3 d_\mu^3(D_0, D_1) \leq r_0^2\mathrm{dist}(P, D_1),
\end{equation}
\begin{equation}
	\label{infaccia}
  \mathrm{dist}(P, \cup_{i \neq j}\sigma^{D_0}_{ij} ) \geq \chi r_0,
\end{equation}
and for any point $ \overline{P} \in B_{r_\sharp}(P_0) $, there exists a curve $ \mathfrak{c} $ joining $ \overline{P} $ to $Q= P + dr_0 n $, where $ n $ is the unit outer normal to $ \partial D_0 $, such that, denoting by $C_{PQ}^{dr_0}$ the closed cylinder with axis the segment joining $P$ and $Q$ and radius $dr_0$, the set
\begin{equation}
	\label{2-19}
   \mathcal{V}= \left\{x\in \RR^3\ | \ \mathrm{dist}(x, \mathfrak{c} )\leq d r_0 \right\}\cup C_{PQ}^{dr_0}
\end{equation}
is contained in $\oms\setminus \Omega_{\mathcal{G}}$.
\end{proposition}
\begin{proof}
The statement above is a refinement of Proposition 3.8 in \cite{AspBerFraVes22} and the proof follows by slight modifications of the arguments in \cite{AspBerFraVes22}. Let us notice that, without loss of generality, we have assumed that the point $P$ belongs to the boundary of the domain $D_0$.
\end{proof}

Finally, we recall that if the two polyhedra are sufficiently close, then they have the same number of faces, edges and vertices. 
\begin{proposition}[Proposition 3.9 in \cite{AspBerFraVes22}]
	\label{distvert}
There exist two constants $\delta_0 \in (0,1)$ and $C_4>1$, depending only on $M_0$, $M_1$ and $\theta_0$, such that, if for some $D_0$ and $D_1$ in $\mathcal{P}$,
\begin{equation}
	\label{eq:dH piccola}
    d_H\left(\partial D_0, \partial D_1\right)\leq \delta_0 r_0,
\end{equation}
then $D_0$ and $D_1$ have the same number $N$ of vertices $\left\{ V^{D_0}_i\right\}_{i=1}^N$ and $\left\{V^{D_1}_i\right\}_{i=1}^N$, respectively, which can be ordered in such a way that
\begin{equation}
	\label{eq:distvert}
    C_4^{-1}d_H\left(\partial D_0, \partial D_1\right) \leq |V^{D_0}_i - V^{D_1}_i|\leq C_4 d_H\left(\partial D_0, \partial D_1\right).
\end{equation}
Moreover, for each edge or face in $D_0$ there is an edge or a face in $D_1$ with corresponding vertices.
\end{proposition}

\section{A first step: logarithmic stability}
\label{sec:rough_log_stab_est}

\subsection{Construction of ad hoc Green's functions}

Let us introduce the Kelvin and Rongved fundamental solutions, which will be useful in the sequel.

Given $y \in \RR^3$ and a \textit{concentrated force} $l \delta(\cdot - y)$ applied at $y$,
with $l \in \RR^3$, $|l|=1$, the \textit{Kelvin normalized
	fundamental solution} $u^K \in L^1_{loc}(\RR^3, \RR^3)$ is defined by
\begin{equation}
	\label{eq:pbm-Kelvin}
	\left\{ \begin{array}{ll}
		\dive_x \left ( \CC^e\nabla_x u^K(x,y;l) \right ) =-l\delta(x-y),
		& \hbox{in}\ \RR^3,\\
		&  \\
		\lim_{|x| \rightarrow \infty} u^K(x,y;l)=0,\\
	\end{array}\right.
\end{equation}
where $\delta(\cdot - y)$ is the Dirac distribution supported at
$y$\textcolor{magenta}{.}

Notice that
\begin{equation*}
	u^K(x,y;l) = \Gamma^K(x,y)l
\end{equation*}
where $\Gamma^K$ is  the \textit{Kelvin fundamental matrix}

\begin{equation}
\label{eq:Kelvin-Gamma}
\Gamma^K(x,y)= \frac{1}{16\pi\mu^e(1-\nu^e)}
\cdot
\frac{1}{|x-y|}
\left (
\frac{(x-y)\otimes (x-y)   }{|x-y|^2} + (3-4\nu^e)Id
\right )
\end{equation}
that satisfies
\begin{equation}
	\label{eq:simm-Kelvin}
	\Gamma^K(x,y) = (\Gamma^K(y,x))^T, \quad \hbox{for every } x\in
	\RR^3, \ x \neq y,
\end{equation}
\begin{equation}
	\label{eq:stima-Kelvin}
	|\Gamma^K(x,y)| \leq C |x-y|^{-1}, \quad \hbox{for every } x\in
	\RR^3, \ x \neq y,
\end{equation}
\begin{equation}
	\label{eq:stima-gradiente-Kelvin}
	|\nabla_x \Gamma^K(x,y)| \leq C |x-y|^{-2}, \quad \hbox{for every } x\in
	\RR^3, \ x \neq y,
\end{equation}
where the constant $C>0$ only depends on $\mu^e$, $\lambda^e$.

Similarly,  given $y \in \RR^3$ and a \textit{concentrated force} $l \delta(\cdot - y)$ applied at $y$,
with $l \in \RR^3$, $|l|=1$, the \textit{Rongved normalized
	fundamental solution} $u^R \in L^1_{loc}(\RR^3, \RR^3)$ is defined by 
\begin{equation}
	\label{eq:pbm-Rongved}
	\left\{ \begin{array}{ll}
		\dive_x \left ( (\CC^e+(\CC^i-\CC^e)\chi_{\RR^3_-})\nabla_x u^R(x,y;l) \right ) =-l\delta(x-y),
		& \hbox{in}\ \RR^3,\\
		&  \\
		u^R((x_1, x_2, 0^+),y;l)=u^R((x_1, x_2, 0^-),y;l), \\
		&  \\
		(\CC^e \nabla_x u^R((x_1, x_2, 0^+),y;l))e_3=(\CC^i \nabla_x u^R((x_1, x_2, 0^-),y;l))e_3,\\
		&  \\
		\lim_{|x| \rightarrow \infty} u^R(x,y;l)=0,\\
	\end{array}\right.
\end{equation}
where
\begin{equation*}
	u^R(x,y;l) = \Gamma^R(x,y)l.
\end{equation*}
The explicit expression of the \textit{Rongved fundamental matrix} $\Gamma^R$ is reported in \cite[Section $10$]{ADCMR}. Estimates analogous to \eqref{eq:simm-Kelvin}--\eqref{eq:stima-gradiente-Kelvin} hold, precisely
\begin{equation}
	\label{eq:simm-Rongved}
	\Gamma^R(x,y) = (\Gamma^R(y,x))^T, \quad \hbox{for every } x\in
	\RR^3, \ x \neq y,
\end{equation}
\begin{equation}
	\label{eq:stima-Rongved}
	|\Gamma^R(x,y)| \leq C |x-y|^{-1}, \quad \hbox{for every } x\in
	\RR^3, \ x \neq y,
\end{equation}
\begin{equation}
	\label{eq:stima-gradiente-Rongved}
	|\nabla_x \Gamma^R(x,y)| \leq C |x-y|^{-2}, \quad \hbox{for every } x\in
	\RR^3, \ x \neq y,
\end{equation}
where the constant $C>0$ only depends on $\mu^e$, $\mu^i$, $\lambda^e$ and $\lambda^i$.
 
\medskip
Given $D_0\in\mathcal{P}$, let us define by
\begin{equation}
	\label{eq:PG4A.1}
\mathcal{D}=\bigcup_{\overset{\scriptstyle i,j}{\scriptstyle
			i\neq j}}\sigma_{ij}^{D_0}		
\end{equation}
the union of the edges of the polyhedron $D_0$. 

Following the lines of \cite{BFV-IPI-2014}, let us introduce suitable Green's functions. Precisely, we shall construct Green's functions $G^\sharp_0$ and $G^\sharp_1$ such that $G^\sharp_1$ is comparable to the Kelvin fundamental solution $\Gamma^K$, whereas $G^\sharp_0$ is comparable to the Rongved fundamental solution $\Gamma^R$ when its pole is close to the point $P\in \partial D_0$ appearing in Proposition \ref{th:geometric_lemma}. 

Let $y \in \oms \setminus \mathcal{D}$ and let
\begin{equation}
	\label{eq:PG4A.2}
	\overline{r} = \overline{r} (y)= \min 
	\left \{
	\dfrac{r_\sharp}{4}, \ \sin \left ( \dfrac{\theta_0 }{2}  \right ) \hbox{dist} ( y, \mathcal{D} \cup \partial \oms)  
	\right \}.
\end{equation}

Let us notice that, with this choice of the radius $\overline{r}$, by \textit{{(H2-2)}} and {\textit{(H2-5)}}, the ball $B_{\overline{r}}(y)$ cannot intersect more than one face of the boundary of $D_0$. Therefore, 
if $B_{\overline{r}}(y) \cap \partial D_0 \neq \emptyset$, then there exists a Cartesian coordinate system in which 
\begin{equation*}
	B_{\overline{r} }(y) \cap  D_0 =  B_{\overline{r}   }(y)  \cap \{   x_3 > a\},
\end{equation*}
for some $a$ such that $|a| < \overline{r}$. Let us define
\begin{equation}
	\label{eq:PG4A.6}
	\CC_y = \left\{ \begin{array}{ll}
		\CC^i,
		& \hbox{if}\ 	B_{\overline{r} }(y) \subset D_0,\\
		&  \\
		\CC^e,
		& \hbox{if}\ 	B_{\overline{r} }(y) \subset \oms \setminus D_0,\\
		&  \\
		\CC^e + ( \CC^i - \CC^e  )\chi_{ \{ x_3 >a  \} },
		& \hbox{if}\    B_{\overline{r} }(y) \cap \partial D_0 \neq \emptyset.
	\end{array}\right.
\end{equation}
 Let us introduce the fundamental matrix $\Gamma^0(x,y)$ solution to
\begin{equation}
	\label{eq:PG4B.1}
	\left\{ \begin{array}{ll}
		\dive_x \left ( \CC_y\nabla_x \Gamma ^0 (x,y) l \right ) =-l\delta(x-y),
		& \hbox{in}\ \RR^3\setminus \{y\},\\
		&  \\
		\lim_{|x| \rightarrow \infty} \Gamma^0(x,y)l=0,\\
	\end{array}\right.
\end{equation}
for every $l \in \RR^3$, $|l|=1$.

Let us observe that, in view of the definition \eqref{eq:PG4A.6}, the matrix $\Gamma^0$ in \eqref{eq:PG4B.1} coincides with the Kelvin fundamental solution $\Gamma^K$ in the first two cases, namely when $B_{\overline{r} }(y) \subset D_0$ or $B_{\overline{r} }(y) \subset \oms \setminus D_0$, being the inherent tensor constant, while it coincides with the Rongved fundamental solution $\Gamma^R$ in the last case since the inherent tensor has a jump accross the interface $\{ x_3 =a  \}$.

\begin{proposition}
	\label{prop:PGDzero}
For every $y \in \oms \setminus \mathcal{D}$ there exists a unique matrix $G^\sharp_0(x,y)$ continuous in $\oms \setminus \{y\}$ such that, for every $l \in \RR^3$ with $|l|=1$, 
\begin{equation}
	\label{eq:PG4B.2}
	\left\{ \begin{array}{ll}
		\dive_x \left ( \CC^{D_0} \nabla_x G^\sharp_0(x,y) l \right ) =-l\delta(x-y),
		&  x \in \oms,\\
		&  \\
		G^\sharp_0(x,y) l=0,      
		&    x \in \partial \oms,\\
	\end{array}\right.
\end{equation}
\begin{equation}
	\label{eq:PG4B.3}
	\begin{array}{ll}
		G^\sharp_0(x,y) = (G^\sharp_0(y,x)   )^T, \quad \hbox{for every } \ x,y \in \oms  \setminus \mathcal{D}.
	\end{array}
\end{equation}
Moreover, if  ${\rm{dist}}(y, \mathcal{D} \cup \partial \oms  ) \geq \gamma r_\sharp$ for some $\gamma >0$, then
\begin{equation}
	\label{eq:PG4B.4}
	\begin{array}{ll}
		\| G^\sharp_0(\cdot ,y) - \Gamma^0 (\cdot ,y)\|_{H^1(\oms)} \leq \dfrac{C}{r_0},
	\end{array}
\end{equation}
where the constant $C>0$ only depends on the a priori data and on $\gamma$,
and, for every $r$, $0<r \leq r_0$,
\begin{equation}
	\label{eq:PG4B.5}
	\begin{array}{ll}
		\| G^\sharp_0(\cdot ,y)\|_{H^1(\oms  \setminus B_{{r}}(y)  )    }  \leq \dfrac{C}{(r r_0)^{1/2}},
	\end{array}
\end{equation}
where the constant $C>0$ only depends on the a priori data.
\end{proposition}

\begin{proof}
	Let us consider the following problem 

	\begin{equation}
		\label{eq:PG4C.2}
		\left\{ \begin{array}{ll}
			\dive_x \left ( \CC^{D_0} \nabla_x w(x,y;l) \right ) =\dive_x\left((\CC_y-\CC^{D_0})\nabla_x\Gamma^0(x,y)l\right)
			&  x \in \oms,\\
			&  \\
			w(x,y;l)=-\Gamma^0(x,y)l,      
			&    x \in \partial \oms.\\
		\end{array}\right.
	\end{equation}

Since $\CC_y=\CC^{D_0}$ in $B_{\overline r}(y)$ and $\nabla \Gamma^0(\cdot,y)$ is smooth and bounded in $\oms\setminus B_{\overline r}(y)$, we have that  $f(\cdot,y;l)=\left((\CC_y-\CC^{D_0})\nabla_x\Gamma^0(\cdot,y)l\right)\in L^2(\oms)$ and $\Gamma^0(\cdot,y)l_{|\partial \oms}\in H^{1/2}(\partial \oms)$.
Then there exists a unique solution $w(x,y;l)\in H^1(\oms)$ to \eqref{eq:PG4C.2}.

Let 
\begin{equation*}
	G_0^\sharp(x,y)l=\Gamma^0(x,y)l+w(x,y;l).
\end{equation*}

By \eqref{eq:PG4B.1} and \eqref{eq:PG4C.2}, $G_0^\sharp$ satisfies \eqref{eq:PG4B.2}. 
The symmetry \eqref{eq:PG4B.3} follows by standard arguments, see for instance \cite{Evans}.

Moreover
\begin{equation*}
	\|w(\cdot,y;l)\|_{H^1(\oms)}\leq C\left(r_0\|f(\cdot,y;l)\|_{L^2(\oms)} + 
	\|\Gamma^0(\cdot,y)\|_{H^{1/2}(\partial\oms)} \right),
\end{equation*}
where $C>0$ only depends on $\mu^e$, $\mu^i$, $\lambda^e$, $\lambda^i$, $M_0$.

 Let us assume now that ${\rm{dist}}(y, \mathcal{D} \cup \partial \oms  ) \geq \gamma r_\sharp$. 
 Noticing that $f\equiv 0$ in $B_{\overline r}(y)$, $\overline r
 {\geq} r_\sharp\min \left \{ \frac{1}{4}, \gamma \sin \left ( \frac{\theta_0}{2} \right )  \right \}$, recalling the definition of $\Gamma^0$ and the estimates \eqref{eq:stima-gradiente-Kelvin},  \eqref{eq:stima-gradiente-Rongved}, we can easily compute
 
\begin{equation}
	\label{eq:PG4D.3-4}
	\|f(\cdot,y;l)\|_{L^2(\oms)}=\|f(\cdot,y;l)\|_{L^2(\oms\setminus B_{\overline r}(y))}\leq
	C\|\nabla \Gamma^0(\cdot,y)\|_{L^2(\oms\setminus B_{\overline r}(y))}\leq \frac{C}{r_0^2},
\end{equation}
where $C>0$ only depends on the a priori data and on $\gamma$.

 Next, by trace inequalities, estimates \eqref{eq:stima-Kelvin},  \eqref{eq:stima-Rongved} and 
 by \eqref{eq:PG4D.3-4}, we have
 \begin{equation*}
 	\|\Gamma^0(\cdot,y)\|_{H^{1/2}(\partial\oms)}\leq \frac{C}{r_0},
 \end{equation*}
 where $C>0$ only depends on the a priori data and on $\gamma$, so that \eqref{eq:PG4B.4} holds.
 Given $r$, $0<r\leq r_0$, arguing similarly, we have
  \begin{equation*}
 	\|G_0^\sharp(\cdot,y)l\|_{H^{1}(\oms\setminus B_r(y))}\leq
 	\|w(\cdot,y;l)\|_{H^{1}(\oms)}+ 	\|\Gamma^0(\cdot,y)\|_{H^{1}(\oms\setminus B_r(y))}\leq
 	 \frac{C}{(r_0 r)^{1/2}},
 \end{equation*}
that is we have \eqref{eq:PG4B.5}.
\end{proof}
Next, let $y \in \oms \setminus \partial D_1$ and let
\begin{equation*}
	\widetilde{r} = \widetilde{r}(y)=  \min 
	\left \{
	\dfrac{r_\sharp}{4}, \ \hbox{dist} ( y, \partial D_1 \cup \partial \oms)  
	\right \}.
\end{equation*}
With this choice of the radius $\widetilde{r}$, $B_{\widetilde{r} }(y)$ does not intersect $\partial D_1$.
Let us define
\begin{equation}
	\label{eq:PG4G.2}
	\CC_y = \left\{ \begin{array}{ll}
		\CC^i,
		& \hbox{if}\ 	B_{\widetilde{r} }(y) \subset \hbox{Int}(D_1),\\
		&  \\
		\CC^e,
		& \hbox{if}\ 	B_{\widetilde{r} }(y) \subset \oms \setminus D_1\\
	\end{array}\right.
\end{equation}
and let us introduce the fundamental matrix $\Gamma^1(x,y)$ solution, for every $l \in \RR^3$, $|l|=1$, to
\begin{equation}
	\label{eq:PG4G.3}
	\left\{ \begin{array}{ll}
		\dive_x \left ( \CC_y\nabla_x \Gamma^1 (x,y) l \right ) =-l\delta(x-y),
		& \hbox{in}\ \RR^3,\\
		&  \\
		\lim_{|x| \rightarrow \infty} \Gamma^1(x,y)l=0.\\
	\end{array}\right.
\end{equation}
By the same arguments used in the proof of Proposition \ref{prop:PGDzero}, the following proposition holds.

\begin{proposition}
	\label{prop:PGDuno}
	For every $y \in \oms \setminus \partial D_1$ there exists a unique matrix $G^\sharp_1(x,y)$ continuous in $\oms \setminus \{y\}$ such that, for every $l \in \RR^3$ with $|l|=1$, 
	\begin{equation}
		\label{eq:PG4G.4}
		\left\{ \begin{array}{ll}
			\dive_x \left ( \CC^{D_1} \nabla_x G^\sharp_1(x,y) l \right ) =-l\delta(x-y),
			&  x \in \oms,\\
			&  \\
			G^\sharp_1(x,y) l=0,      
			&    x \in \partial \oms,\\
		\end{array}\right.
	\end{equation}
	\begin{equation}
		\label{eq:PG4G.5}
		\begin{array}{ll}
			G^\sharp_1(x,y) = (G^\sharp_1(y,x)   )^T, \quad \hbox{for every } \ x,y \in \oms  \setminus \partial D_1.
		\end{array}
	\end{equation}
	Moreover, if for some $\gamma >0$ we have ${\rm{dist}}(y, \partial D_1 \cup \partial \oms  ) \geq \gamma r_\sharp$, then
	\begin{equation}
		\label{eq:PG4G.6}
		\begin{array}{ll}
			\| G^\sharp_1(\cdot ,y) - \Gamma^1 (\cdot ,y)\|_{H^1(\oms)} \leq \dfrac{C}{r_0},
		\end{array}
	\end{equation}
	where the constant $C>0$ only depends on the a priori data and on $\gamma$,
	and, for every $r$, $0<r \leq r_0$,
	\begin{equation}
		\label{eq:PG4H.1}
		\begin{array}{ll}
			\| G^\sharp_1(\cdot ,y)\|_{H^1(\oms  \setminus B_{{r}}(y)  )    }  \leq \dfrac{C}{(r r_0)^{1/2}},
		\end{array}
	\end{equation}
	where the constant $C>0$ only depends on the a priori data.
\end{proposition}

\subsection{Logarithmic stability result}

\begin{theorem}\label{Teo:stima_log}
	Let $\Omega$ be a bounded domain satisfying \ref{ass:domain} and let $D_0$, $D_1\in\mathcal{P}$. Let the elastic tensors $\mathbb{C}^{D_0}$, $\mathbb{C}^{D_1}$ satisfy assumptions \ref{ass:elastic_material}. Let $\Sigma$ be an open portion of $\partial\Omega$ satisfying \eqref{ass:sigma}.
	There exist $c >0$ and $C>0$, only depending on the a priori data, such that, if 
	\begin{eqnarray}
		\| \Lambda_{D_0}^{\Sigma} - \Lambda_{D_1}^{\Sigma}\|_* \le \frac{\epsilon}{r_0},
		\quad
		0<\epsilon < e^{-c},
	\end{eqnarray}
then
	\begin{eqnarray}
		\label{eq:stima-grezza-dH}
		d_H(\partial D_0, \partial D_1)\le r_0\widetilde{\omega}(\epsilon),
	\end{eqnarray}
where $\widetilde{\omega}$ is the increasing function defined as 
	\begin{eqnarray}
		\label{eq:stima grezza 2 log}
		\widetilde{\omega}(t) = C \left ( \log\left(\frac{|\log t|}{c}\right)\right)^{-1/9},\ \ \ \hbox{for every } \ \ 0<t<e^{-c}.
	\end{eqnarray}
\end{theorem}

\begin{remark}
	\label{rem:1 log}
We emphasize that, in order to obtain the above result, it is sufficient to assume the weaker visibility condition \eqref{eq:visibility_cond} instead of \eqref{eq:monotony-Lamè-moduli-Ce-Ci}--\eqref{eq:monotony-Lamè-moduli-Ci-Ce}.

Let us also notice that estimate \eqref{eq:stima grezza 2 log} could be improved to a single $log$ rate of convergence by refining the geometrical construction underlying the use of the three spheres inequality, as shown in \cite{AspBerFraVes22}. Nevertheless, the above estimate already ensures that hypothesis \eqref{eq:dH piccola} of Proposition \ref{distvert} holds, thereby activating the arguments that yield the Lipschitz stability estimate.

\end{remark}

Let us recall the following lemma, which is a special case of the fundamental regularity result by Li and Nirenberg \cite[Proposition $1.6$]{LiNir03}.
\begin{lemma}
\label{lem:Li-Nirenberg}
Let $B_r$ be a ball of radius $r>0$ centered at the origin, and let $B^{\pm}_{r}$ be the upper and the lower half ball and let $\CC^1$, $\CC^2$ be two constant elastic tensors satisfying \eqref{Notaz-compon-cart-C} and \eqref{Notaz-Lamé-strong-convex}. Let $v\in H^1(B_r)$ be a solution to
\begin{equation}
	\label{eq:equation_lemma_Li_Nir}
	\dive\left( ( \CC^1+(\CC^2-\CC^1)\chi_{B^+_r})\nabla v\right)=0,\qquad \textrm{in}\ B_r.
\end{equation}
Then $v\in C^{\infty}(\overline{B}^{\pm}_r)$ and for all $\delta$, $0< \delta <1$,  there exists a constant $C>0$ depending only on $\lambda_1$, $\mu_1$, $\lambda_2$, $\mu_2$ and $\delta$, such that
\begin{equation}
	\label{eq:estimate_Li_Nir}
	\|\nabla v\|_{L^{\infty}(B_{(1-\delta)r})}\leq \dfrac{C}{r} \|v\|_{L^2(B_r)}.
\end{equation}
\end{lemma}

An additional essential tool, employed both in this section and throughout the remainder of the paper, is the following three‑spheres inequality.

\begin{lemma}[$L^\infty$--Three Spheres Inequality]
	\label{lem:tre sfere}
Let $u \in H^1(B_s(Q))$ be a solution to 
	\begin{equation}
		\label{eq:LBD60.1}
		\begin{aligned}{}
			&
			\dive ( \CC^e \nabla u  ) =0, \quad \hbox{in } B_s(Q),
		\end{aligned}
	\end{equation}
	where the elastic tensor $\CC^e$ satisfies the hypothesis \ref{ass:elastic_material}. For every $s_1$, $s_2$, $s_3$, $0<s_1 <s_2<s_3 \leq s$, 
	\begin{equation}
		\label{eq:LBD60.2}
		\begin{aligned}{}
			&
			\|u\|_{L^\infty (B_{s_2}(Q))  } \leq C \|u\|_{L^\infty (B_{s_1}(Q))  }^\tau 
			\|u\|_{L^\infty (B_{s_3}(Q))  }^{1-\tau},
		\end{aligned}
	\end{equation}
	where $C>0$ and $\tau$, $0<\tau<1$, only depend on $\lambda^e$, $\mu^e$, $ \dfrac{s_2}{s_3}$, $ \dfrac{s_1}{s_3}$.
\end{lemma}
\begin{proof}
The result follows {}from \cite[Theorem 5.1]{AM01} and {}from the arguments developed in \cite[Theorem 1.10]{ARRV}. 
\end{proof}

The strategy adopted for the proof of the logarithmic stability estimate is based on the singular solutions method. Such solutions are introduced below. 
    
For any $y,w \in \oms \setminus \Omega_{\mathcal{G}}$ and for any $l,m\in \mathbb{R}^3$ such that $|l|=|m|=1$, we define
\begin{equation}
	\label{definiz-S}
	\begin{aligned}{}
S(y,w;l,m)&= \int_{D_0}(\CC^i-\CC^e)\nabla_x (G_0^{\sharp}(x,y) l)\cdot \nabla_x (G_1^{\sharp}(x,w)m) - \\
 &
 -\int_{D_1}(\CC^i-\CC^e)\nabla_x (G_0^{\sharp}(x,y) l)\cdot \nabla_x (G_1^{\sharp}(x,w)    m).
 \end{aligned}
\end{equation}
Let us fix $y=\overline{y}\in \Omega^{\sharp}\setminus \Omega_{\mathcal{G}}$ and $l\in \mathbb{R}^3, |l|=1$. We define 
\begin{eqnarray}\label{vettoref}
&&\widehat{S}_k(\overline{y}, w;l)=S(\overline{y},w; l ,e_k), \ \ \ k=1,2,3, \nonumber\\
&& \widehat{S}=(\widehat{S}_1,\widehat{S}_2,\widehat{S}_3).
\end{eqnarray}
Similarly, let us fix $w=\overline{w}$ in $\Omega^{\sharp}\setminus \Omega_{\mathcal{G}}$ and $m\in \mathbb{R}^3, |m|=1$. We define 
\begin{eqnarray*}
&&\widetilde{S}_j(y,\overline{w};m)=S(y,\bar{w};e_j,m), \ \ \ j=1,2,3,\nonumber\\
&& \widetilde{S}=(\widetilde{S}_1,\widetilde{S}_2,\widetilde{S}_3).
\end{eqnarray*}

Arguing as in Proposition $4.4$ in \cite{BFV-IPI-2014} we find the following. Given $y=\overline{y}\in \Omega^{\sharp}\setminus \Omega_{\mathcal{G}}$ and $l\in \mathbb{R}^3, |l|=1$,
the vector-valued function $\widehat{S}=\widehat{S}(\overline{y},w;l)$ satisfies the Lam\'e system 
\begin{eqnarray}\label{divf}
\mbox{div}_w (\CC^e\nabla_{w}\widehat{S})= 0, \ \ \ \mbox{for every }w \in \Omega^{\sharp}\setminus \Omega_{\mathcal{G}}.
\end{eqnarray}

Given $w=\overline{w}\in \Omega^{\sharp}\setminus \Omega_{\mathcal{G}}$ and $m\in \mathbb{R}^3, |m|=1$,
the vector-valued function $\widetilde{S}=\widetilde{S}(y,\overline{w};m)$ satisfies the Lam\'e system 
\begin{eqnarray*}
\mbox{div}_y (\CC^e\nabla_{y}\widetilde{S})= 0, \ \ \  \mbox{for every }y \in \Omega^{\sharp}\setminus \Omega_{\mathcal{G}}.
\end{eqnarray*}

By the linearity of $S$ with respect to $l$ and $m$, we have trivially that,
for every $ l,m \in \RR^3, \ |l|=|m|=1$, 
\begin{equation}
	\label{eq:doppia disug per S}
	\begin{aligned}{}
		&
		| S (y,\overline{w};l,m)|
		\leq 
		|\widetilde{S}(y, \overline{w};m)|\leq \sqrt{3} 
		\max_
		{\overset{\scriptstyle l \in \RR^3}{\scriptstyle
				|l|=1}} | S (y,\overline{w};l,m)|,
		\ \ \hbox{for every } y, \, \overline{w} \in \oms \setminus \Omega_{\mathcal{G}},
		\\
		&
		| S (\overline{y},w;l,m)|
		\leq 
		|\widehat{S}(\overline{y},w;l)|\leq \sqrt{3} 
		\max_
		{\overset{\scriptstyle m \in \RR^3}{\scriptstyle
				|m|=1}} | {S} (\overline{y},w;l,m)|, 
		\ \ \hbox{for every } \overline{y},w \in \oms \setminus \Omega_{\mathcal{G}}.
	\end{aligned}
\end{equation}

For sake of simplicity, in the following two theorems, we consider a Cartesian coordinate system in which the origin coincides with the point $P$ introduced in Proposition \ref{th:geometric_lemma} and $e_3= -n$, with $n  $ the outer unit normal to $D_0$ at $P\equiv O$.

\begin{theorem}[Upper bound of S]\label{stimadallalto}
For $\epsilon >0$, let us assume 
\begin{eqnarray}
\| \Lambda_{D_0}^{\Sigma} - \Lambda_{D_1}^{\Sigma}\|_* \le \frac{\epsilon}{r_0}. 
\end{eqnarray}
Under the notation of Proposition \ref{th:geometric_lemma}, let 
\begin{eqnarray*}
&&y_h=P-he_3,\\
&&w_h=P-\lambda_w he_3, \ \frac{2}{3}\leq\lambda_w<1,
\end{eqnarray*}
where $0<h\leq dr_0$, with $d$ introduced in Proposition \ref{th:geometric_lemma}.
There exist positive constants $\tau$, $K_1$, $K_2$, $\tau  \in (0,1)$ only depending on $\lambda^e$ and $\mu^e$,  $K_1$ only depending on the a priori data, and $K_2$ only depending on $M_1$, such that, for every $l ,m\in \RR^3$ with $|l|=|m|=1$, 
\begin{equation}
	\label{eq:alto}
	\begin{aligned}{}
		&
		| S(y_h,w_h;l, m)| \leq 
		\frac{K_1}{h} 
		\epsilon^{\tau^{K_2\left(\frac{r_0}{h}\right)^{3}}}.
	\end{aligned}
\end{equation}
\end{theorem}
\begin{proof}
The proof is divided into three steps. 

\medskip

\noindent 
\textit{Step 1.}
\textit {Let  $P_0\in \Omega_0$ be the point introduced in \eqref{eq:ball},
satisfying $\mbox{dist}(P_0,\partial\Omega)>2r_\sharp$. For every $y,w\in B_{r_\sharp (P_0)}$,
\begin{eqnarray}\label{step1}
|S(y,w;l,m)|\le C \frac{\varepsilon}{r_0},
\end{eqnarray}
with $C$ depending on the a priori data}.

{}From Alessandrini identity (see for instance \cite[Lemma 6.1]{ADCMR}) it follows that 
\begin{eqnarray*}
\langle (\Lambda_{D_0}^{\Sigma} - \Lambda_{D_1}^{\Sigma})G_0^{\sharp}(\cdot,y)l, G_1^{\sharp}(\cdot,w) m\rangle = S(y,w;l,m),
\end{eqnarray*}
being $G_0^{\sharp}(\cdot,y) l, G_1^{\sharp}(\cdot,w) m \in H^1(\Omega)$.
Hence, by standard trace estimates and by \eqref{eq:PG4B.5} and \eqref{eq:PG4H.1}, we have that 
\begin{equation*}
	\begin{aligned}
&|S(y,w;l,m)|\le r_0^2\| \Lambda_{D_0}^{\Sigma} - \Lambda_{D_1}^{\Sigma} \|_* \|G^{\sharp}_0(\cdot, y) \|_{H^1(\Omega)} \|G^{\sharp}_1(\cdot, w) \|_{H^1(\Omega)}\le \\
&\leq r_0\varepsilon \|G^{\sharp}_0(\cdot, y) \|_{H^1(\Omega^{\sharp}\setminus B_{r_\sharp}(y))} \|G^{\sharp}_1(\cdot, w) \|_{H^1(\Omega^{\sharp}\setminus B_{r_\sharp}(w))}\le \frac{C\epsilon}{r_0},
\end{aligned}
\end{equation*}
with $C$ depending on the a priori data.

\medskip

\noindent 
\textit{Step 2.}
\textit{For $y, w\in \Omega^{\sharp} \setminus \Omega_{\mathcal{G}}$, we have
\begin{eqnarray}\label{step2a}
	|S(y,w;l,m)|\le C d_{y}^{-\frac{1}{2}}d_{w}^{-\frac{1}{2}},
\end{eqnarray}
where $d_y=\min\left\{\frac{r_\sharp}{4},\hbox{dist}(y, \partial \Omega_{\mathcal{G}}\cup \partial \oms)\right\}$, $d_w=\min\left\{\frac{r_\sharp}{4},\hbox{dist}(w, \partial \Omega_{\mathcal{G}}\cup \partial \oms)\right\}$ and 
$C$ is a positive constant depending on the a priori data. If, in particular, $y\in B_{r_\sharp}(P_0)$, then
\begin{eqnarray}\label{step2b}
|S(y,w;l,m)|\le C r_0^{-\frac{1}{2}}d_{w}^{-\frac{1}{2}}
\end{eqnarray}
where $C$ is a positive constant depending on the a priori data.}

We have that 
\begin{eqnarray*}
|S(y,w;l,m)|\le  C\int_{D_0} \left| \nabla_x (G^{\sharp}_0(x,y) l)| \cdot |\nabla_x (G^{\sharp}_1(x,w) m)    \right| \mbox{d}x +\\
+\ C\int_{D_1} \left| \nabla_x (G^{\sharp}_0(x,y) l)| \cdot |\nabla_x (G^{\sharp}_1(x,w) m)    \right| \mbox{d}x,
\end{eqnarray*}
with $C$ only depending on $\mu^e$, $\mu^i$, $\lambda^e$ and $\lambda^i$.

By \eqref{eq:PG4B.5} and \eqref{eq:PG4H.1}, we obtain 
\begin{equation*}
	\begin{aligned}
&\int_{D_i} \left| \nabla_x (G^{\sharp}_0(x,y)l)|\cdot  |\nabla_x (G^{\sharp}_1(x,w) m)    \right| \mbox{d}x \le Cr_0^3 \|\nabla_x G^{\sharp}_0(\cdot,y)\|_{L^2(D_i)} \|\nabla_x G^{\sharp}_1(\cdot,w)\|_{L^2(D_i)}\le \\
&\leq  Cr_0 \|G_0^\sharp(\cdot,y)\|_{H^1(\Omega^{\sharp}\setminus B_{d_y}(y))} \|G_1^\sharp(\cdot,w)\|_{H^1(\Omega^{\sharp}\setminus B_{d_w}(w))} \le C d_{y}^{-\frac{1}{2}}d_{w}^{-\frac{1}{2}},
\end{aligned}
\end{equation*}
with $C$ depending on the a priori data, so that \eqref{step2a} follows.
If, in particular, $y\in B_{r_\sharp}(P_0)$, then $d_y=\frac{r_\sharp}{4}$ and \eqref{step2b} follows trivially {}from \eqref{step2a}.

\medskip

\noindent 
\textit{Step 3.}
\textit{Propagation of smallness}

We recall that the vector valued function $\widehat{S}$ introduced in \eqref{vettoref} satisfies the Lam\'e system \eqref{divf}. 
  
In this step we make an extensive use of the three spheres inequality.

Let $w\in B_{r_\sharp}(P_0)$. Using the notation introduced in Proposition \ref{th:geometric_lemma}, let $ \mathfrak{c}'$ be the curve obtained by gluing $ \mathfrak{c} $ (choosing $\overline{P}=w$) and the segment joining $Q=P+dr_0 n$ to $P$.

Let $\gamma(t)$ be a parametrization of $ \mathfrak{c}'$ with starting point $w$ and ending point $P$. 
Proceeding as in \cite[Theorem 5.1]{ARRV}, we iterate the three spheres inequality \eqref{eq:LBD60.2} over a chain of pairwise disjoint balls centered at points $\{ x_i \}_{i=1}^L$ belonging to $\mathfrak{c}'$ and defined as follows: $x_1 = w $, $x_{i+1}=\gamma(t_i)$, with $t_i = \max \{ t \ | \ |\gamma(t)-x_i|=\frac{h}{4}   \}$ if $|x_i-w_h| > \frac{h}{4}$; otherwise, let us set $i=L$ and stop the process. Since the balls $B_{h/8}(x_i)$, $i=1,\dots, L$, are pairwise disjoint by construction, we have
\begin{equation}
	\label{eq:numero-palle0}
	\begin{aligned}{}
		&
		L \leq C_L \left ( \dfrac{r_0}{h}  \right )^3,
	\end{aligned}
\end{equation}
where $C_L>0$ only depends on $M_1$. We choose in \eqref{eq:LBD60.2} the radii $s_1 = \frac{h}{8}$, $s_2 = \frac{3h}{8}$ and $s_3 = \frac{h}{2}$. Let us notice that $\cup_{i=1}^L B_{\frac{h}{2}}(x_i) \subset \left \{  x \in \mathcal V \ | \ \hbox{dist}(x, \partial D_0) \geq \frac{h}{6}   \right \}$.

By \eqref{eq:doppia disug per S} and \eqref{step2b}, for every $y \in B_{s_1}(P_0)$, we have
\begin{equation} 
	\label{eq:Step3_1}
\begin{aligned}{}
	&
	\|  \widehat{S} (y, \cdot; l) \|_{ L^\infty \left (  \cup_{i=1}^L B_{h/2}(x_i)   \right )   }
	\leq C(hr_0)^{-\frac{1}{2}},
\end{aligned}
\end{equation}
where $C>0$ depends on the a priori data.

By applying \eqref{eq:LBD60.2} and using \eqref{eq:Step3_1}, at the $i$th step we obtain
\begin{equation}
	\label{eq:Step3_2}
	\begin{aligned}{}
		&
		(hr_0)^{\frac{1}{2}} \| \widehat{S}(y,\cdot;l)\|_{ L^\infty (B_{s_1} (x_{i+1}))} \leq C 
		\left (
		(hr_0)^{\frac{1}{2}} \|  \widehat{S}(y,\cdot;l)\|_{ L^\infty (B_{s_1}(x_{i}))}		
		\right )^\tau,
	\end{aligned}
\end{equation}
where $C>0$ and $\tau$, $0<\tau <1$, depend on the a priori data. Using estimate \eqref{step1} and iterating, we have 
\begin{equation}
	\label{eq:Step3_3}
	\begin{aligned}{}
		&
		 \| \widehat{S}(y,\cdot;l)\|_{ L^\infty (B_{s_1} (w_h)  )} \leq 
		\frac{C}{(hr_0)^{\frac{1}{2}}} \left (  \left(\frac{h}{r_0}\right)^{\frac{1}{2}}\epsilon  \right )^{\tau^L}, 
		\quad \hbox{for every } y \in B_{s_1}(P_0),
	\end{aligned}
\end{equation}
where $C>0$  depends on the a priori data.
In particular, recalling \eqref{eq:doppia disug per S}, we have
\begin{equation}
	\label{eq:Step3_4}
	\begin{aligned}{}
		&
		\| \widetilde{S}(\cdot,w_h;m)\|_{ L^\infty (B_{s_1} (P_0)  )} \leq 
		\frac{C}{(hr_0)^{\frac{1}{2}}} \left (  \left(\frac{h}{r_0}\right)^{\frac{1}{2}}\epsilon  \right )^{\tau^L}, 
	\end{aligned}
\end{equation}
where $C>0$  depends on the a priori data.
Now we repeat the above process for the function $\widetilde{S}$, operating on the $y$-variable instead of the $w$-variable and considering the curve $\mathfrak{c}$ with starting point $P_0$.
By \eqref{step2a}, we have that
\begin{equation} 
	\label{eq:Step3_5}
	\begin{aligned}{}
		&
		\|  \widetilde{S} (\cdot, w_h; m) \|_{ L^\infty \left (  \cup_{i=1}^L B_{h/2}(x_i)   \right )   }
		\leq Ch^{-1},
	\end{aligned}
\end{equation}
\begin{equation}
	\label{eq:Step3_6}
	\begin{aligned}{}
		&
		\| \widetilde{S}(\cdot,w_h;m)\|_{ L^\infty (B_{s_1} (y_h)  )} \leq 
		\frac{C}{h}   \left(\frac{h}{r_0}\right)^{\tau^{2L}}
		\epsilon  ^{\tau^{2L}}<\frac{C}{h}   
		\epsilon  ^{\tau^{2L}},
	\end{aligned}
\end{equation}
where $C>0$  depends on the a priori data. Recalling \eqref{eq:numero-palle0}, \eqref{eq:alto} follows with $K_2=2C_L$.
\end{proof}

\begin{theorem}[Lower bound of $S$]\label{stimadalbasso}
Under the notation of Proposition \ref{th:geometric_lemma}, let 
\begin{eqnarray*}
y_h=P-he_3 \ .
\end{eqnarray*}

 For any $i=1,2,3$ there exist $\lambda_w \in \left\{\frac{2}{3}, \frac{3}{4}, \frac{4}{5} \right\}$ and $\bar{h}\in(0,\frac{1}{2})$, only depending on the a priori data, such that 
\begin{eqnarray}\label{basso}
|S(y_h,w_h,e_i,e_i)|\ge \frac{K_3}{h}\ \ \ \mbox{for any}\ \ h, 0<h\le \bar{h}\rho,
\end{eqnarray}
with 
\begin{eqnarray*}
w_h&=&P-\lambda_w he_3,\\
 \rho&=&\min\left \{ \mbox{dist}(P, D_1),  \frac{r_\sharp}{2}, dr_0, \frac{2\sin(\theta_0/2)}{1+\sin(\theta_0/2)} \chi r_0 \right \},
\end{eqnarray*}
where $d, \chi$ have been introduced in Proposition \ref{th:geometric_lemma}
and $K_3>0$ is a constant only depending on the a priori data. 

\end{theorem}

\begin{proof}
Let $0<h\le \bar{h}\rho$ with $\bar{h}\in(0,\frac{1}{2})$ to be chosen later.  
By recalling \eqref{definiz-S}, we may rewrite $S(y_h,w_h;l,m)$ as
\begin{eqnarray*}
	S(y_h,w_h;l,m)= S_{D_0}(y_h,w_h;l,m) - S_{D_1}(y_h,w_h;l,m),
\end{eqnarray*}
where
\begin{eqnarray*}
S_{D_0}(y_h,w_h;l,m) =\int_{D_0} (\CC^i-\CC^e)\nabla_x[G^\sharp_0(x,y_h)l]\cdot \nabla_x[G^\sharp_1(x,w_h) m ], \\ 
S_{D_1}(y_h,w_h;l,m) =\int_{D_1} (\CC^i-\CC^e)\nabla_x[G^\sharp_0(x,y_h)l]\cdot \nabla_x[G^\sharp_1(x,w_h) m ] .
\end{eqnarray*}
Hence, we have
\begin{eqnarray*}
|S(y_h,w_h;l,m)|\ge |S_{D_0}(y_h,w_h;l,m)| - |S_{D_1}(y_h,w_h;l,m)|.
\end{eqnarray*}
We now estimate {}from below $S_{D_0}$. We have 
\begin{equation*}
	\begin{aligned}
&S_{D_0}(y_h,w_h;l,m) = \int_{D_0 \cap B_{\rho}}(\CC^i-\CC^e)\nabla_x (\Gamma^R(x,y_h) l)\cdot \nabla_x (\Gamma^K(x,w_h) m) +\\
&+\int_{D_0 \cap B_{\rho}}(\CC^i-\CC^e)\nabla_x[G^\sharp_0(x,y_h) l - \Gamma^R(x,y_h) l ]\cdot \nabla_x[G^\sharp_1(x,w_h)m - \Gamma^K(x,w_h) m ]+\\
&+\int_{D_0 \cap B_{\rho}}(\CC^i-\CC^e)\nabla_x[G^\sharp_0(x,y_h) l - \Gamma^R(x,y_h) l ]\cdot \nabla_x[\Gamma^K(x,w_h) m ]+\\
&+ \int_{D_0 \cap B_{\rho}}(\CC^i-\CC^e)\nabla_x[\Gamma^R(x,y_h) l ]\cdot \nabla_x[G^\sharp_1(x,w_h) m - \Gamma^K(x,w_h)m ]+\\
&+\int_{D_0\setminus B_{\rho}}(\CC^i-\CC^e)\nabla_x [G^\sharp_0(x,y_h) l ] \cdot \nabla_x[G^\sharp_1(x,w_h)m]=\\
&= I_1 + I_2 +I_3 + I_4 + I_5.
\end{aligned}
\end{equation*}
{}From the above equality it follows that 
\begin{eqnarray}\label{somma}
|S_{D_0}(y_h,w_h;l,m)|\ge |I_1|-|I_2| - |I_3|-|I_4|-|I_5|.
\end{eqnarray}
We now estimate each term $I_i, i=1,\dots, 5$.
We start {}from $I_1$. Choosing $l=m=e_i, \ i=1,2,3$, by the results in  \cite[Section 9]{ADCMR} we have that there exists  $\lambda_w \in \left\{\frac{2}{3}, \frac{3}{4}, \frac{4}{5} \right\}$ such that
\begin{eqnarray}\label{I1}
|I_1|\ge \frac{C}{h} -\frac{\widetilde{C}}{\rho}.
\end{eqnarray}
where $C,\widetilde{C}>0$ are constants depending on the a priori data only (see formula $(9.20)$ in \cite{ADCMR}).

We now estimate {}from above the term $I_2$.
We have that 
\begin{eqnarray}\label{I2a}
|I_2|\le Cr_0^3 \| \nabla_x (G^\sharp_0(\cdot,y_h) -  \Gamma^R(\cdot,y_h))\|_{L^2(\Omega^{\sharp})}\| \nabla_x (G^\sharp_1(\cdot,w_h) -  \Gamma^K(\cdot,w_h))\|_{L^2(\Omega^{\sharp})}
\end{eqnarray}
We notice that, by our choice of $\rho$ and since $\bar h<\frac{1}{2}$, $B_{\bar{r}(y_h)}(y_h)\cap \partial D_0 \neq \emptyset$ and  $B_{\widetilde{r}(w_{h})}(w_h)\subset \Omega^{\sharp}\setminus D_1$ so that $\Gamma^0(\cdot, y_h)=\Gamma^R(\cdot,y_h)$ in $\mathbb{R}^3 \setminus \{y_h\}$ and  $\Gamma^1(\cdot, w_h)=\Gamma^K(\cdot,w_h)$ in $\mathbb{R}^3 \setminus \{w_h\}$. 

 Hence by \eqref{I2a}, \eqref{eq:PG4B.4} and \eqref{eq:PG4G.6} we have that 
\begin{eqnarray}\label{I2b}
|I_2|\le \frac{C}{r_0},
\end{eqnarray}
with $C$ depending on the a priori data.
We now estimate  the term $I_3$. We have, by the same argument used above, that 
\begin{eqnarray}\label{I3a}
|I_3|&\le& Cr_0^3 \| \nabla_x (G^\sharp_0(\cdot,y_h) -  \Gamma^R(\cdot,y_h))\|_{L^2(\Omega^{\sharp})}\| \nabla_x ( \Gamma^K(\cdot,w_h))\|_{L^2(D_0 \cap B_{\rho})}\le \nonumber\\
&\leq& Cr_0\| \nabla_x ( \Gamma^K(\cdot,w_h)) m\|_{L^2(D_0 \cap B_{\rho})},
\end{eqnarray}
with $C$ depending on the a priori data.
We have that 
\begin{eqnarray}\label{I3b}
\| \nabla_x ( \Gamma^K(\cdot,w_h)) m\|_{L^2(D_0 \cap  B_{\rho})}^2\le \frac{C}{r_0^3}\int_{D_0\cap B_{\rho}}|x-w_h |^{-4} \mbox{d}x.
\end{eqnarray}
We observe that if $x\in B_{\rho}\cap D_0$ then $|x-w_h|\ge \lambda_w h$ which leads to $B_{\rho}\cap D_0 \subset \mathbb{R}^3\setminus B_{\lambda_w h}(w_h)$.

Hence by \eqref{I3b} we have that 
\begin{eqnarray*}
\| \nabla_x ( \Gamma^K(\cdot,w_h)) m\|_{L^2(D_0 \cap B_{\rho})}^2\le \frac{C}{r_0^3}\int_{\mathbb{R}^3\setminus B_{\lambda_w h}(w_h)} |x-w_h|^{-4} \mbox{d}x=
\frac{C}{r_0^3}\int_{|y|\ge \lambda_w h}|y|^{-4}\mbox{d}y.
\end{eqnarray*}
By using spherical coordinates we infer that 
\begin{eqnarray}\label{I3d}
\| \nabla_x ( \Gamma^K(\cdot,w_h))m\|_{L^2(D_0 \cap B_{\rho})}^2\le
\frac{C}{r_0^3}\int _{\lambda_w h}^{+\infty} r^{-2}\mbox{d}r \le \frac{C}{hr_0^3}.
\end{eqnarray}
By combining \eqref{I3a} and \eqref{I3d} we have 
\begin{eqnarray}\label{I3e}
|I_3|\le C(hr_0)^{-\frac{1}{2}}.
\end{eqnarray}
With the same argument we can bound the term $I_4$ as follows:
\begin{eqnarray}\label{I4a}
|I_4|\le C(hr_0)^{-\frac{1}{2}}.
\end{eqnarray}
We now estimate the last term $I_5$. In this respect, let us observe that if $x\in B_{\frac{\bar{h}\rho}{2}}(y_h)$ then we have 
\begin{eqnarray*}
|x|\le |x-y_h| +|y_h|\le \frac{\bar{h}\rho}{2} + \bar{h}\rho= \frac{3\bar{h}\rho}{2}<\rho.
\end{eqnarray*} 
As a consequence we have that $B_{\frac{\bar{h}\rho}{2}}(y_h)\subset B_\rho$ and hence $D_0\setminus B_{\rho}\subset \Omega^{\sharp}\setminus B_{\frac{\bar{h}\rho}{2}}(y_h)$.
Moreover, by our choice of $\rho$, $\frac{\overline h \rho}{2}<\overline{r}(y_h)$, so that
\eqref{eq:PG4B.5} applies and we have that 
\begin{eqnarray*}
\int_{D_0\setminus B_{\rho}}| \nabla_x(G^{\sharp}_0(x,y_h) l)|^2\le C \int_{\Omega^{\sharp}\setminus B_{\frac{\bar{h}\rho}{2}}(y_h)}| \nabla_x(G^{\sharp}_0(x,y_h) l)|^2\le C \rho^{-1}.
\end{eqnarray*}
Analogously, observing that $\frac{\bar{h}\rho}{2}<\widetilde{r}(w_h)$, we can apply \eqref{eq:PG4H.1} obtaining
\begin{eqnarray*}
\int_{D_0\setminus B_{\rho}}| \nabla_x(G^{\sharp}_1(x,w_h) l)|^2\le C \int_{\Omega^{\sharp}\setminus B_{\frac{\bar{h}\rho}{2}}(w_h)}| \nabla_x(G^{\sharp}_1(x,w_h) l)|^2\le C \rho^{-1}.
\end{eqnarray*}

By the above inequalities we get 
\begin{eqnarray}\label{I5}
|I_5|\le C \rho^{-1}. 
\end{eqnarray}
By combining \eqref{somma},\eqref{I1},\eqref{I2b},\eqref{I3e},\eqref{I4a},\eqref{I5} we have 
\begin{eqnarray}\label{fD0}
|S_{D_0}(y_h,w_h;e_i,e_i)|\ge \frac{C}{h}\left(1- \frac{h}{\rho} - \frac{h}{r_0} -\left(\frac{h}{r_0}\right)^{\frac{1}{2}}\right),
\end{eqnarray}
with $C$ depending on the a priori data.
We now estimate {}from above the term $|S_{D_1}(y_h,w_h;l,m)|$. In this respect we have 
\begin{equation*}
|S_{D_1}(y_h,w_h;l,m)|\le Cr_0^3 \|\nabla_x G^{\sharp}_0(x,y_h)\|_{L^2(D_1)}  \|\nabla_x G^{\sharp}_1(x,w _h)\|_{L^2(D_1)}.
\end{equation*}
By our choice of $\rho$ and since $B_{\frac{\bar{h}\rho}{2}}(y_h)\subset B_{\frac{3}{2}\bar{h}\rho}$ we have that 
\begin{eqnarray*}
D_1\subset \Omega^{\sharp}\setminus B_{\frac{3}{2}\bar{h}\rho}\subset \Omega^{\sharp}\setminus B_{\frac{\bar{h}\rho}{2}}(y_h).
\end{eqnarray*}
Hence, in view of the above inclusion and by \eqref{eq:PG4B.5}, it follows that 
\begin{eqnarray}\label{eq:D1a}
\int_{D_1}|\nabla_x G_0^{\sharp}(x,y_h)|^2 \le \int_{\Omega^{\sharp}\setminus B_{\frac{\bar{h}\rho}{2}}(y_h)}|\nabla_x G_0^{\sharp}(x,y_h)|^2\le C \rho^{-1} \ .
\end{eqnarray}
By analogous argument and by \eqref{eq:PG4H.1}, we have that 
\begin{eqnarray}\label{eq:D1b}
\int_{D_1}|\nabla_x G_1^{\sharp}(x,w_h)|^2 \le C \rho^{-1} \ 
\end{eqnarray}
By combining \eqref{eq:D1a} and \eqref{eq:D1b}, we have 
\begin{eqnarray}\label{fD1}
|S_{D_1}(y_h,w_h;l,m)|\le C \rho^{-1},
\end{eqnarray}
with $C$ depending on the a priori data.
Finally, by \eqref{fD0} and \eqref{fD1} we find that, for $i=1,2,3$,
\begin{eqnarray*}
|S(y_h,w_h;e_i,e_i)|\ge |S_{D_0}(y_h,w_h;e_i,e_i)| - |S_{D_1}(y_h,w_h;e_i,e_i)|\ge \frac{C}{h}\left(1- \frac{h}{\rho} - \frac{h}{r_0} -\left(\frac{h}{r_0}\right)^{\frac{1}{2}}\right).
\end{eqnarray*}
Therefore, there exists $\bar{h}\in \left(0,\frac{1}{2}\right)$, only depending on the a priori data, such that for any $h, 0<h<\bar{h}\rho$, we have that 
\begin{eqnarray*}
|S(y_h, w_h; e_i, e_i)|\ge \frac{K_3}{h},
\end{eqnarray*}
for $i=1,2,3$, and this concludes the proof. 

\end{proof}

\begin{proof} [Proof of Theorem \ref{Teo:stima_log}]

By combining \eqref{eq:alto} and \eqref{basso} we deduce that 
\begin{eqnarray}\label{altobasso}
\frac{K_3}{h}\le |S(y_h,w_h; e_i, e_i)|\le \frac{K_1}{h}\epsilon^{\tau^
{K_2 \left(\frac{r_0}{h}\right)^3}},
\end{eqnarray}
for any $h, 0<h\le\bar{h}\rho$.  
By \eqref{altobasso} we have
\begin{eqnarray*}
\frac{K_3}{K_1}\le \epsilon^{\tau^
	{K_2 \left(\frac{r_0}{h}\right)^3}}.  
\end{eqnarray*}
Let $\alpha = \min \left\{  \frac{K_3}{K_1}, \frac{1}{2}   \right\}$, $\alpha <1$, and let us assume $\epsilon < \alpha$. Then
\begin{eqnarray*}
\left| \log \alpha  \right|\ge \tau^
	{K_2 \left(\frac{r_0}{h}\right)^3}|\log \epsilon|,
\end{eqnarray*}
which leads to 
\begin{eqnarray*}
h\le r_0\left(K_2|\log\tau|\right)^{\frac{1}{3}}
\left ( \log\left(\frac{|\log\epsilon|
	}{|\log  \alpha |}\right)\right )^{-\frac{1}{3}}.
\end{eqnarray*}

In particular choosing $h=\bar{h}\rho$ we find that 
\begin{eqnarray}
	\label{stimapreliminare}
\rho\le \widetilde{K} r_0
\left ( \log\left(\frac{|\log\epsilon|
}{|\log  \alpha |}\right)\right )^{-\frac{1}{3}}.
\end{eqnarray}
We now distinguish two cases. 

{\emph{First case}}: $\rho=\mbox{dist}(P,\partial D_1)$. 

By \eqref{1-19} we have that
\begin{eqnarray}
	\label{stimapreliminare2}
d_{\mu}^3(D_0,D_1)\le \frac{\rho r_0^2}{C_3}
\le  \frac{\widetilde{K}r_0^3}{C_3}
\left ( \log\left(\frac{|\log\epsilon|
}{|\log  \alpha |}\right)\right )^{-\frac{1}{3}}.
\end{eqnarray}
Proposition \ref{prop2} and \eqref{stimapreliminare2} imply that 
\begin{eqnarray*}
d_{H}(\partial D_0, \partial D_1)\le C_2 d_{\mu}(D_0,D_1)\le  {C} r_0
\left ( \log\left(\frac{|\log\epsilon|
}{|\log  \alpha |}\right)\right )^{-\frac{1}{9}}.
\end{eqnarray*}

{\emph{Second case}}: $\rho = \gamma r_0$, with $\gamma=\min\left\{\frac{\zeta}{2}, d, \frac{2\sin(\theta_0/2)\chi}{1+\sin(\theta_0/2)}\right\}$.
To treat this case, we simply observe that by \eqref{stimapreliminare}
\begin{eqnarray*}
d_{H}(\partial D_0, \partial D_1)\le M_1 r_0=  M_1 \frac{\rho}{\gamma}\le  C r_0\left ( \log\left(\frac{|\log\epsilon|
}{|\log  \alpha |}\right)\right )^{-\frac{1}{9}},
\end{eqnarray*} 
and \eqref{eq:stima grezza 2 log} follows with $c=| \log \alpha |$.
\end{proof}

\section{Domain derivative of the local Dirichlet-to-Neumann map}
\label{sec:domain derivative}
In this section, we take advantage of the property that if  $D_0$ and $D_1$ are closed enough, then they must have the same number of vertices. This allows in Subsection 6.1 to build a homotopy, induced by a vector field, mapping $D_0$ into $D_1$. This construction allows then in Subsection 6.2 to derive  the Gateaux derivative of the local Dirichlet to Neumann map. Finally, in Subsection 6.3 we derive a quantitative lower bound via an elastic moment tensor.

\subsection{A suitable homotopy map}
\label{omotopia}

In this subsection we provide two basic results: the construction of an ad hoc Lipschitz vector field mapping any vertex $V_i^{D_0}$ in the vector $V_i^{D_1}-V_i^{D_0}$ and having $W^{1,\infty}(\Omega)$-norm bounded by 
$d_H(\partial D_0, \partial D_1)$ (Proposition \ref{vectorfield}) and the construction of a homotopy map between $D_0$ and $D_1$ (Proposition \ref{prop:property_Phi}).
Throughout this section, we assume that \eqref{eq:dH piccola} holds, so that
$D_0$ and $D_1$ have the same number of vertices and
$$
\hbox{dist}(V^{D_0}_i,V^{D_1}_i)\leq C_4 d_H(\partial D_0, \partial D_1),\qquad \textrm{for}\ i=1,\dots,N.
$$
For brevity, let us set
\begin{equation*}
	d_H=d_H(\partial D_0, \partial D_1).
\end{equation*}

Let $\mathcal{W}\subset \om$ be a tubular neighborhood of $D_0$ with width $\frac{r_0}{4}$, so that 
\begin{equation*}
	\rm{dist}(\mathcal{W},\partial\om)\geq \frac{r_0}{2}.
\end{equation*}

As a first step, let us introduce the following geometrical lemma.

\begin{lemma}
	\label{triangoli_isosceli}
Let us define

\begin{equation}\label{h0}
	h_0=\frac{r_0}{2}\tan\gamma,
\end{equation}	
where 
\begin{equation}
	\label{gamma}
	\gamma=\min\left\{ \frac{\theta_0}{3},\frac{\pi}{2}-\arctan (M_0) \right\}.
\end{equation}	
Let $F$ be a face of $D_0$, $\alpha$ the plane containing $F$ and $\sigma$ an edge of $F$. Let $M$ be the midpoint of $\sigma$, $E= M-h_0n$, where $n$ denotes the unit outer normal to $F$ at $M$ in the plane $\alpha$, and let $T$ be the isoscele triangle of basis $\sigma$ and third vertex $E$. 

Then $T\subset F$ and  
all the triangles defined as above on the edges of $F$ are pairwise disjoint, except for the common vertex of adjacent triangles. Moreover, if $T$ and $\widetilde T$ have not adjacent bases, then

\begin{equation}\label{dist_triangoli}
	\hbox{dist}(T,\widetilde T)\geq r_0\cos(\arctan (M_0)).
\end{equation}	

\end{lemma}
\noindent

\medskip
Let $\mathcal T_0$ be the family of all the triangles $T$ constructed as described above on all edges of $D_0$.

The next proposition follows by further assuming that $d_H \leq C r_0$, with $C>0$ a suitable constant only depending on $M_0$, $M_1$, $\theta_0$.

\begin{proposition}[Vector field]
		\label{vectorfield}
	There exists a vector field $\mathcal{U}\in W^{1,\infty}(\mathbb{R}^3, \RR^3)$ satisfying the following properties 
	\begin{align}
		&\mathcal{U}(V^{D_0}_i)=V^{D_1}_i-V^{D_0}_i,\qquad \hbox{for every }\ i=1,\ldots,N,\label{ass1:U}\\
		&\rm{supp} \, (\mathcal{U}) \subset \overline{\mathcal{W}}, \label{ass2:U}\\
		& \mathcal{U}\ \text{is affine on each }  T \in \mathcal{T}_0 ,\label{ass3:U}\\
		& \|\mathcal{U}\|_\infty +r_0 \| D\mathcal{U}\|_\infty \leq \widetilde{C} d_H,\label{ass4:U}
	\end{align}
	where $D\mathcal{U}$ denotes the Jacobian matrix of $\mathcal{U}$ and $\widetilde{C}$ is a constant depending only on $\theta_0$ and $M_1$.  
\end{proposition}

\begin{proposition}[Homotopy]
	\label{prop:property_Phi}
	Let us further assume that $ d_H \leq \dfrac{1}{3\widetilde{C}} r_0$. Let
		\begin{equation*}
		\Phi_t=Id+t\,\mathcal{U}, \qquad t\in[0,1].
	\end{equation*}
	Then we have that
	\begin{align}
		& \|D \mathcal{U}\|_\infty \leq \dfrac{1}{3},
		\label{eq:DU-infty-bound}
		\\
		& \Phi_t\in W^{1,\infty}(\Omega,\Omega),\label{eq:Phi regularity}\\
		& \Phi_1(D_0)=D_1,\label{eq:D0 va in D1}\\
		& \Phi_t \  \textrm{is invertible},\label{eq:property_phi_2}\\
		&  D\Phi_t=I_3+tD \mathcal{U}, 
		\label{eq:property_phi_1bis}\\
		& (D \Phi_t)^{-1} =I_3-tD\mathcal{U}(D \Phi_t)^{-1}, 
		\label{eq:property_phi_2bis} \\
		& \|D \Phi_t \|_\infty \leq \sqrt{3}+ \dfrac{1}{3},  \quad \|(D \Phi_t)^{-1} \|_\infty \leq \frac{3}{2}, \label{eq:property_phi_3} \\
		& |\det(D\Phi_t)-1-t\,\dive\,\mathcal{U}|\leq 6t^2|D\mathcal{U}|^2+6t^3|D\mathcal{U}|^3, 
		\label{eq:property_phi_1ter}\\
		&  \dfrac{1}{2} \leq \det (D\Phi_t) \leq \dfrac{3}{2}.
		 \label{eq:property_phi_4}
	\end{align}
\end{proposition}

Note that similar results have been derived in \cite{AspBerFraVes22} except for property \ref{eq:D0 va in D1}. For this reason, and to provide a more accurate estimation of the constants involved, we decided to provide more detailed proofs.
The proofs of Lemma \ref{triangoli_isosceli}, Proposition \ref{vectorfield} and Proposition \ref{prop:property_Phi} are postponed to Section \ref{Appendix}.

\subsection{Gateaux derivative of the local Dirichlet-to-Neumann map}
\label{Gateaux}

For $t \in [0,1]$, let  $D_t = \Phi_t(D_0)$ and 
\begin{equation}
	\label{eq:GAT1.1}
	\CC^{D_t}(x) = \CC^{D_0}( \Phi^{-1}_t (x)), \quad x \in \Omega.
\end{equation}
Given $f \in H^{1/2}(\partial \om)$, let us consider the weak solution $u_t \in H^1(\om)$ to
\begin{equation}
	\label{eq:GAT1.2}
	\left\{ \begin{array}{ll}
		\dive \left ( \CC^{D_t} \nabla u_t \right ) =0,
		&  \hbox{in } \om,\\
		&  \\
		u_t =f,      
		&    \hbox{on }\partial \om.\\
	\end{array}\right.
\end{equation}
In particular, for $t=0$, $u_0 \in H^1(\om)$ satisfies
\begin{equation}
	\label{eq:GAT1.3}
	\left\{ \begin{array}{ll}
		\dive \left ( \CC^{D_0} \nabla u_0 \right ) =0,
		&  \hbox{in } \om,\\
		&  \\
		u_0 =f,      
		&    \hbox{on } \partial \om.\\
	\end{array}\right.
\end{equation}

\begin{definition}
	\label{def:material-derivative}
	The material derivative $\dot{u}_0 \in H^1_0(\om)$ of $u_t$ in $t=0$, when it exists, is defined as
	\begin{equation}
		\label{eq:GAT2.2}
		\dot{u}_0 = \dfrac{d}{dt} \left ( u_t \circ \Phi_t \right )|_{t=0}
		=
		\lim_{ t \rightarrow 0} \dfrac{u_t   \circ \Phi_t - u_0 }{t} ,
	\end{equation}
	where the limit is meant in the  $H^1(\om)$-topology.
\end{definition}

\begin{proposition}
	\label{prop:GAT2.1}
	There exists the material derivative $\dot{u}_0 \in H^1_0(\om)$ and, for every $\psi \in H^1_0(\om)$, it satisfies 
	\begin{equation}
		\label{eq:GAT2.1}
		\int_\om \CC^{D_0} \nabla \dot{u}_0 \cdot \nabla \psi = - \int_\om 
		\left \{
		\CC^{D_0} \nabla u_0 \left (  (\dive \mathcal U) I_3 - ( D \mathcal U  )^T  \right )
		- 
		\CC^{D_0} (   \nabla u_0  D \mathcal U  )
		\right \} \cdot \nabla \psi. 
	\end{equation}
\end{proposition}

\begin{proof}
Let $\widetilde{f} \in H^1(\om)$ be an extension of $f$ such that $\hbox{supp} (\widetilde{f}) \cap \hbox{supp} (\mathcal{U}) = \emptyset$, $\|\widetilde{f}\|_{H^1(\om)} \leq C \|f\|_{H^{1/2}(\partial \om)}$, with $C$ only depending on $M_0$. 

The functions defined as
\begin{equation*}
	w_t= u_t - \widetilde f, \quad w_0= u_0 - \widetilde f
\end{equation*}
belong to $H^1_0(\om)$ and satisfy, respectively,
\begin{equation*}
	\left\{ \begin{array}{ll}
		\dive ( \CC^{D_t} \nabla w_t ) = 
		- \dive ( \CC^{D_t} \nabla \widetilde f ),
		&  \hbox{in } \om,\\
		&  \\
		w_t =0,      
		&    \hbox{on } \partial \om,\\
	\end{array}\right.
\end{equation*}
and
\begin{equation*}
	\left\{ \begin{array}{ll}
		\dive ( \CC^{D_0} \nabla w_0 ) = 
		- \dive ( \CC^{D_0} \nabla \widetilde f ),
		&  \hbox{in }\om,\\
		&  \\
		w_0 =0,      
		&    \hbox{on }\partial \om.\\
	\end{array}\right.
\end{equation*}
The weak formulations of the above problems are
\begin{equation}
	\label{eq:GAT3.4}
	\int_\om \CC^{D_t}(x) \nabla_x w_t(x) \cdot \nabla_x \psi(x) dx =
	-
	\int_\om \CC^{D_t}(x) \nabla_x \widetilde f(x) \cdot \nabla_x \psi(x) dx, \quad \hbox{for every } \psi \in H^1_0(\om),
\end{equation}
\begin{equation}
	\label{eq:GAT3.5}
	\int_\om \CC^{D_0}(y) \nabla_y w_0(y) \cdot \nabla_y \psi(y) dy =
	-
	\int_\om \CC^{D_0}(y) \nabla_y \widetilde f(y) \cdot \nabla_y \psi(y) dy, \quad \hbox{for every } \psi \in H^1_0(\om).
\end{equation}
Let us apply the change of variables $x=\Phi_t(y)$ to \eqref{eq:GAT3.4} and let us define $\widetilde w_t (y)=  w_t(\Phi_t(y))$, $\widetilde \psi (y)=  \psi (\Phi_t(y))$. By noticing that $\mathcal{U}=0$, $\Phi_t=Id$, $(\Phi_t)^{-1}=Id$ in $\hbox{supp} (\widetilde f)$ and $\Phi_t(\om)=\om$, we have that, for every $\psi \in H^1_0(\om)$, 
\begin{equation}
	\label{eq:GAT6.2}
	\begin{aligned}{}
		&	\int_\om \CC^{D_0}(y) [ \nabla_y \widetilde w_t (y) (D \Phi_t)^{-1}(y)    ]
		\cdot
		[ \nabla_y \widetilde \psi (y) (D \Phi_t)^{-1}(y)    ]
		\det ( D \Phi_t) (y) dy =
		\\
		&	
		=
		-
		\int_\om 
		\CC^{D_0}(y)
		\nabla_y \widetilde f (y)
		\cdot\nabla_y \widetilde \psi (y) dy,
		\quad
		\hbox{for every } t \in [0,1]. 
	\end{aligned}
\end{equation}
To simplify the notation, we omit the dependence on the variable $y$ and we relabel $\widetilde \psi$ with $\psi$, since $\widetilde \psi$ cover $H^1_0(\om)$ as $\psi$ varies in $H^1_0(\om)$. 

By comparing \eqref{eq:GAT3.5} and \eqref{eq:GAT6.2}, we obtain, for every $\psi \in H^1_0(\om)$,
\begin{equation}
	\label{eq:GAT6.3}
	\begin{aligned}{}
		&	\int_\om \CC^{D_0} [ \nabla \widetilde w_t  (D \Phi_t)^{-1}    ]
		\cdot
		[ \nabla \psi  (D \Phi_t)^{-1}    ]
		\det ( D \Phi_t) 
		=
		\int_\om 
		\CC^{D_0}
		\nabla w_0
		\cdot
		\nabla \psi ,
		\quad
		\hbox{for every } t \in [0,1].
	\end{aligned}
\end{equation}
	By subtracting to both sides of \eqref{eq:GAT6.3} the quantity
\begin{equation*}
	\begin{aligned}{}
		&	\int_\om \CC^{D_0} [ \nabla w_0 (D \Phi_t)^{-1}  ]
		\cdot
		[ \nabla \psi  (D \Phi_t)^{-1}    ]
		\det ( D \Phi_t),
	\end{aligned}
\end{equation*}
by dividing by $t$ and using the identity
\begin{equation}
	\label{eq:identity-3-matrices}
		%
		A \cdot (CB)= (A B^T) \cdot C, \quad
		\hbox{for every } 3 \times 3 \hbox{ matrices},
	%
\end{equation}
we have
\begin{equation}
	\label{eq:GAT7.1bis-rinom}
	\begin{aligned}{}
		&	\int_\om \CC^{D_0} \left [ \nabla
		\left (
		\dfrac{\widetilde w_t -w_0}{t}
		\right )
		 (D \Phi_t)^{-1}    \right ]
		\cdot
		[ \nabla \psi  (D \Phi_t)^{-1}    ]
		\det ( D \Phi_t) 
		=
		\\
		&
		=
		\int_\om 
		\dfrac{1}{t}
		\left \{
		\CC^{D_0}
		\nabla w_0
		-
		\det (D \Phi_t)
		\left (
		\CC^{D_0}
		\left [
		\nabla w_0
		(D \Phi_t)^{-1}
		\right ]
		\right )
		(D \Phi_t)^{-T}
		\right \}
		\cdot
		\nabla \psi.
	\end{aligned}
\end{equation}
By choosing $\psi = \dfrac{\widetilde w_t - w_0}{t}$, we obtain
\begin{equation}
	\label{eq:GAT7.1bis}
	\begin{aligned}{}
		&	\int_\om \CC^{D_0} \left [ \nabla
		\left (
		\dfrac{\widetilde w_t -w_0}{t}
		\right )
		(D \Phi_t)^{-1}    \right ]
		\cdot
		\left [ \nabla \left (  \dfrac{\widetilde w_t - w_0}{t}  \right )  (D \Phi_t)^{-1}    \right ]
		\det ( D \Phi_t) 
		=
		\\
		&
		=
		\int_\om 
		\dfrac{1}{t}
		\left \{
		\CC^{D_0}
		\nabla w_0
		-
		\det (D \Phi_t)
		\left (
		\CC^{D_0}
		\left [
		\nabla w_0
		(D \Phi_t)^{-1}
		\right ]
		\right )
		(D \Phi_t)^{-T}
		\right \}
		\cdot
		\nabla \left (  \dfrac{\widetilde w_t - w_0}{t}  \right ).
	\end{aligned}
\end{equation}
By using \eqref{eq:property_phi_2bis}, we may rewrite the left hand side of \eqref{eq:GAT7.1bis} as
\begin{equation*}
		%
		\int_\om \CC^{D_0} \left [ \nabla
		\left (
		\dfrac{\widetilde w_t -w_0}{t}
		\right )
		(D \Phi_t)^{-1}    \right ]
		\cdot
		\left [ \nabla \left (  \dfrac{\widetilde w_t - w_0}{t}  \right )  (D \Phi_t)^{-1}    \right ]
		\det ( D \Phi_t) 
		=I_1+I_2+I_3+I_4,
\end{equation*}
where
\begin{equation*}
	%
	I_1=\int_\om \CC^{D_0} \nabla
	\left (
	\dfrac{\widetilde w_t -w_0}{t}
	\right )   
	\cdot
	 \nabla \left (  \dfrac{\widetilde w_t - w_0}{t}  \right )   
	\det ( D \Phi_t),
\end{equation*}
\begin{equation*}
	%
	I_2=-t\int_\om\left(\CC^{D_0}
	\left[ \nabla
	\left (
	\dfrac{\widetilde w_t -w_0}{t}
	\right ) 
	D\mathcal{U}(D\Phi_t)^{-1}  
	\right]
	\right)
	\cdot
	\nabla \left (  \dfrac{\widetilde w_t - w_0}{t}  \right )   
	\det ( D \Phi_t),
\end{equation*}

\begin{equation*}
	%
	I_3=-t\int_\om\CC^{D_0}
	\nabla
	\left (
	\dfrac{\widetilde w_t -w_0}{t}
	\right ) 
	\cdot
	\left[
	\nabla
	\left (
	\dfrac{\widetilde w_t -w_0}{t}
	\right ) 
	D\mathcal{U}(D\Phi_t)^{-1}  
	\right]
	\det ( D \Phi_t),
\end{equation*}
\begin{equation*}
	%
	I_4=t^2\int_\om\left(\CC^{D_0}
	\left[ \nabla
	\left (
	\dfrac{\widetilde w_t -w_0}{t}
	\right ) 
	D\mathcal{U}(D\Phi_t)^{-1}  
	\right]
	\right)
	\cdot
	\left[
	\nabla
	\left (
	\dfrac{\widetilde w_t -w_0}{t}
	\right ) 
	D\mathcal{U}(D\Phi_t)^{-1}  
	\right]
	\det ( D \Phi_t).
\end{equation*}
By \eqref{eq:property_phi_4}, by \eqref{Notaz-forte-convex-C} and by the first Korn's inequality,
we have
\begin{equation*}
	%
	I_1\geq 
	C_1\int_\om\left|
	\nabla
	\left(	
	\dfrac{\widetilde w_t -w_0}{t}
	\right)
	\right|^2=
	C_1r_0^3
		\left\|\nabla
\left(	
\dfrac{\widetilde w_t -w_0}{t}
\right)\right\|_{L^2(\om)}^2,
\end{equation*}
with $C_1$ only depending on $\lambda^i$, $\mu^i$, $\lambda^e$, $\mu^e$. By using \eqref{ass4:U}, \eqref{eq:DU-infty-bound}, \eqref{eq:property_phi_3} and \eqref{eq:property_phi_4}, we have
\begin{equation*}
	%
	|I_2|+ |I_3| + |I_4|\leq 
	C_2t\frac{d_H}{r_0}
	\int_\om\left|
	\nabla
	\left(	
	\dfrac{\widetilde w_t -w_0}{t}
	\right)
	\right|^2=
	C_2d_Hr_0^2
	\left\|\nabla
	\left(	
	\dfrac{\widetilde w_t -w_0}{t}
	\right)\right\|_{L^2(\om)}^2,
\end{equation*}
where $C_2$ only depends on $\lambda^i$, $\mu^i$, $\lambda^e$, $\mu^e$, $\theta_0$, $M_1$.
Therefore, we can bound {}from below the left hand side of \eqref{eq:GAT7.1bis} with
\begin{equation}
	\label{eq:GAT11.1}
	I_1-(|I_2|+ |I_3| + |I_4|)\geq 
	\left(C_1-C_2\frac{d_H}{r_0}\right)r_0^3
	\left\|\nabla
	\left(	
	\dfrac{\widetilde w_t -w_0}{t}
	\right)\right\|_{L^2(\om)}^2\geq \frac{C_1}{2}r_0^3
	\left\|\nabla
	\left(	
	\dfrac{\widetilde w_t -w_0}{t}
	\right)\right\|_{L^2(\om)}^2,
\end{equation}
assuming $\frac{d_H}{r_0}\leq \frac{C_1}{2C_2}$.

By using 
\eqref{eq:property_phi_2bis}, we may rewrite \eqref{eq:GAT7.1bis} as 
\begin{equation}
	\label{eq:GAT13.1}
	\begin{aligned}
	&\int_\om \CC^{D_0} \left [ \nabla
	\left (
	\dfrac{\widetilde w_t -w_0}{t}
	\right )
	(D \Phi_t)^{-1}    \right ]
	\cdot
	\left [ \nabla \left (  \dfrac{\widetilde w_t - w_0}{t}  \right )  (D \Phi_t)^{-1}    \right ]
	\det ( D \Phi_t)= 
	\\
	&
	=
	\int_\om \frac{1-\det(D\Phi_t)}{t}\CC^{D_0}\nabla w_0\cdot\nabla \left(	
	\dfrac{\widetilde w_t -w_0}{t}
	\right)
	+
	\\
	&
	\int_\om 
	\left [
	\left\{
	(\CC^{D_0}\nabla w_0)(D\Phi_t)^{-T}D\mathcal{U}^T+
	\CC^{D_0}(\nabla w_0D\mathcal{U}(D\Phi_t)^{-1})-
	t
	\left [	
	\CC^{D_0}(\nabla w_0D\mathcal{U}(D\Phi_t)^{-1})
	\right ]
	((D\Phi_t)^{-T}D\mathcal{U}^T)
	\right\}
	\cdot
	\right.
	\\
	&
	\left.
	\cdot
	\nabla \left(	
	\dfrac{\widetilde w_t -w_0}{t}
	\right)  \det(D\Phi_t)
	\right  ].
	\end{aligned}
\end{equation}
By \eqref{ass4:U}, \eqref{eq:property_phi_1ter}, \eqref{eq:DU-infty-bound}, \eqref{eq:property_phi_3}, \eqref{eq:property_phi_4}
we can bound {}from above the right hand side of \eqref{eq:GAT7.1bis} with
\begin{equation}
	\label{eq:GAT14.1}
	Cr_0^3\|D\mathcal{U}\|_\infty \|\nabla w_0\|_{L^2(\om)}
\left\|\nabla \left(	
	\dfrac{\widetilde w_t -w_0}{t}
	\right)\right\|_{L^2(\om)}
	\leq Cd_Hr_0^2\|\nabla w_0\|_{L^2(\om)}
	\left\|\nabla \left(	
	\dfrac{\widetilde w_t -w_0}{t}
	\right)\right\|_{L^2(\om)},
\end{equation}
with $C$ only depending on $\lambda^i$, $\mu^i$, $\lambda^e$, $\mu^e$, $\theta_0$, $M_1$.
By comparing \eqref{eq:GAT11.1} and \eqref{eq:GAT14.1}, we obtain
\begin{equation}
	\label{eq:GAT15.0}
	\left\|\nabla \left(	
	\dfrac{\widetilde w_t -w_0}{t}
	\right)\right\|_{L^2(\om)}\leq
	C\frac{d_H}{r_0}
\|\nabla w_0\|_{L^2(\om)},
\end{equation}
where $C$ only depends on $\lambda^i$, $\mu^i$, $\lambda^e$, $\mu^e$, $\theta_0$, $M_1$.

By \eqref{eq:GAT15.0} and by Poincar\'e inequality, $\dfrac{\widetilde w_t -w_0}{t}$ is bounded in $H^1_0(\om)$ for $t\in (0,1]$ so that there exist a sequence $t_n\rightarrow 0$ and a function $\dot{w}_0\in H^1_0(\om)$ such that  $\dfrac{\widetilde w_{t_n} -w_0}{t_n}$ converges to $\dot{w}_0$ weakly in $H^1_0(\om)$.
Let us rewrite the right hand side of \eqref{eq:GAT7.1bis-rinom} operating as above in deriving \eqref{eq:GAT13.1}. We have 

\begin{equation*}
	%
	\begin{aligned}{}
		&	\int_\om \CC^{D_0} \left [ \nabla
		\left (
		\dfrac{\widetilde w_t -w_0}{t}
		\right )
		(D \Phi_t)^{-1}    \right ]
		\cdot
		[ \nabla \psi  (D \Phi_t)^{-1}    ]
		\det ( D \Phi_t) 
		=\int_\om \frac{1-\det(D\Phi_t)}{t}\CC^{D_0}\nabla w_0\cdot\nabla \psi+\\
		&\int_\om \left\{
		(\CC^{D_0}\nabla w_0)(D\Phi_t)^{-T}D\mathcal{U}^T+
		\CC^{D_0}(\nabla w_0D\mathcal{U}(D\Phi_t)^{-1})-t
		\left [
		\CC^{D_0}(\nabla w_0D\mathcal{U}(D\Phi_t)^{-1})
		\right ]
		((D\Phi_t)^{-T}D\mathcal{U}^T)
		\right\}
		\cdot
		\\
		&
		\cdot 
		\nabla \psi
		\det(D\Phi_t). 
	\end{aligned}
\end{equation*}
Passing to the limit in the above equation for $t=t_n\rightarrow 0$, and using \eqref{eq:property_phi_1ter}, \eqref{eq:property_phi_2bis}, it follows that, for every $\psi\in H^1_0(\om)$,
\begin{equation}
	\label{eq:GAT17.1}
	\begin{aligned}{}
		&	\int_\om \CC^{D_0} \nabla \dot{w}_0
		\cdot
		\nabla \psi 
		=-\int_\om \{\CC^{D_0}\nabla w_0((\dive \mathcal{U})I_3-D\mathcal{U}^T)
		- \CC^{D_0}(\nabla w_0D\mathcal{U})\}
		\cdot
		\nabla \psi. 
	\end{aligned}
\end{equation}
	
By applying Lax-Milgram theorem to the above identity, $\dot{w}_0$ is uniquely determined in $H^1_0(\om)$, so that
\begin{equation}
	\label{eq:GAT22.1}
	\dot{w}_0= \lim_{t\rightarrow 0} \frac{\widetilde w_t-w_0}{t} \quad \hbox{weakly in } H^1_0(\om).
\end{equation}

Passing to the limit as $t\rightarrow 0$ in \eqref{eq:GAT13.1} and by applying the identity  \eqref{eq:GAT17.1}

\begin{equation}
	\label{eq:GAT18.1}
	\begin{aligned}{}
		&\lim_{t\rightarrow 0}	\int_\om \CC^{D_0} \left [ \nabla
		\left (
		\dfrac{\widetilde w_t -w_0}{t}
		\right )
		(D \Phi_t)^{-1}    \right ]
		\cdot
		 \left [ \nabla
		\left (
		\dfrac{\widetilde w_t -w_0}{t}
		\right )
		(D \Phi_t)^{-1}    \right ]
		\det ( D \Phi_t)=\\
		&=-\int_\om \{(\CC^{D_0}\nabla w_0)((\dive \mathcal{U})I_3-D\mathcal{U}^T)
		- \CC^{D_0}(\nabla w_0D\mathcal{U})\}
		\cdot
		\nabla \dot{w}_0=
		\int_\om \CC^{D_0}\nabla\dot{w}_0\cdot \nabla\dot{w}_0
	\end{aligned}
\end{equation}

	In order to prove the strong convergence of $\dfrac{\widetilde w_t -w_0}{t}$ to $\dot{w}_0$ in $H^1_0(\om)$, let us express the integral
$\displaystyle{
\int_\om \CC^{D_0} \nabla
\left (
\dfrac{\widetilde w_t -w_0}{t}
\right )
\cdot
\nabla
\left (
\dfrac{\widetilde w_t -w_0}{t}
\right )}$ in terms of the integral appearing in the left hand side of \eqref{eq:GAT18.1}. 

A standard computation using \eqref{eq:property_phi_1ter} gives
\begin{equation*}
	%
	\begin{aligned}{}
		&\int_\om \CC^{D_0}  \nabla
		\left (
		\dfrac{\widetilde w_t -w_0}{t}
		\right )
		\cdot
		\nabla
		\left (
		\dfrac{\widetilde w_t -w_0}{t}
		\right )=
		B_1+tB_2-t^2B_3-B_4,
	\end{aligned}
\end{equation*}
where
\begin{equation*}
	B_1=\int_\om \CC^{D_0} \left [ \nabla
	\left (
	\dfrac{\widetilde w_t -w_0}{t}
	\right )
	(D \Phi_t)^{-1}    \right ]
	\cdot
	\left [ \nabla
	\left (
	\dfrac{\widetilde w_t -w_0}{t}
	\right )
	(D \Phi_t)^{-1}    \right ]
	\det ( D \Phi_t),
\end{equation*}
\begin{equation*}
	%
	\begin{aligned}{}
		&
	B_2=\int_\om \CC^{D_0}  \nabla
	\left (
	\dfrac{\widetilde w_t -w_0}{t}
	\right ) 
	\cdot
	\left [ \nabla
	\left (
	\dfrac{\widetilde w_t -w_0}{t}
	\right )
	D\mathcal{U}(D \Phi_t)^{-1}    \right ]
	\det ( D \Phi_t)
	+
	\\
	&
	+
	\int_\om \CC^{D_0}
	\CC^{D_0} 
	\left [ \nabla
	\left (
	\dfrac{\widetilde w_t -w_0}{t}
	\right )
	D\mathcal{U}(D \Phi_t)^{-1}    \right ]
	\cdot
	 \nabla
	\left (
	\dfrac{\widetilde w_t -w_0}{t}
	\right ) 
	\det ( D \Phi_t),
	\end{aligned}{}
\end{equation*}
\begin{equation*}
	B_3=\int_\om 
	\CC^{D_0} 
	\left [ \nabla
	\left (
	\dfrac{\widetilde w_t -w_0}{t}
	\right )
	D\mathcal{U}(D \Phi_t)^{-1}    \right ]
	\cdot
		\left [ \nabla
	\left (
	\dfrac{\widetilde w_t -w_0}{t}
	\right )
	D\mathcal{U}(D \Phi_t)^{-1}    \right ]
	\det ( D \Phi_t),
\end{equation*}
\begin{equation*}
	B_4=\int_\om 
	\CC^{D_0} 
	\left [ \nabla
	\left (
	\dfrac{\widetilde w_t -w_0}{t}
	\right )
	(D \Phi_t)^{-1}    \right ]
	\cdot
	\left [ \nabla
	\left (
	\dfrac{\widetilde w_t -w_0}{t}
	\right )
	(D \Phi_t)^{-1}    \right ]
	(\det(D\Phi_t)-1).
\end{equation*}
By \eqref{ass4:U}, \eqref{eq:property_phi_1ter}, \eqref{eq:property_phi_3} and \eqref{eq:GAT15.0}, we can estimate
\begin{equation*}
	|B_2|\leq C d_H^3 \|\nabla w_0\|^2_{L^2(\om)},
\end{equation*}
\begin{equation*}
	|B_3|\leq C d_H^3 \|\nabla w_0\|^2_{L^2(\om)},
\end{equation*}
\begin{equation*}
	|B_4|\leq Ct d_H^3\|\nabla w_0\|^2_{L^2(\om)},
\end{equation*}
with $C$ only depending on $\lambda^i$, $\mu^i$, $\lambda^e$, $\mu^e$, $\theta_0$, $M_1$.

Therefore, by \eqref{eq:GAT18.1}, 
\begin{equation}
	\label{eq:10.1Edi}
	\lim_{t\rightarrow 0}\int_\om \CC^{D_0}  \nabla
		\left (
		\dfrac{\widetilde w_t -w_0}{t}
		\right )
		\cdot
		\nabla
		\left (
		\dfrac{\widetilde w_t -w_0}{t}
		\right )=
		\int_\om \CC^{D_0} \nabla \dot{w}_0\cdot \nabla \dot{w}_0.
\end{equation}
By the major symmetry \eqref{Notaz-simmetrie-C} and the strong convexity \eqref{Notaz-forte-convex-C} of the tensor $\CC^{D_0}$, by first Korn's inequality, by \eqref{eq:10.1Edi} and by the weak convergence \eqref{eq:GAT22.1}, we derive the strong convergence in $L^2(\om)$ of $\nabla
\left (
\dfrac{\widetilde w_t -w_0}{t}
\right )$ to $\dot{w}_0$:
\begin{equation*}
	\begin{aligned}{}
	&C\int_\om\left|\nabla  	\left (
	\dfrac{\widetilde w_t -w_0}{t} -\dot{w}_0
	\right )\right|^2\leq
	\int_\om \CC^{D_0} \nabla  	\left (
	\dfrac{\widetilde w_t -w_0}{t} -\dot{w}_0
	\right )\cdot  \nabla  	\left (
	\dfrac{\widetilde w_t -w_0}{t} -\dot{w}_0
	\right )=\\
	&=\int_\om \CC^{D_0}  \nabla
	\left (
	\dfrac{\widetilde w_t -w_0}{t}
	\right )
	\cdot
	\nabla
	\left (
	\dfrac{\widetilde w_t -w_0}{t}
	\right )-2
	\int_\om \CC^{D_0}  \nabla
	\dot{w}_0
	\cdot
	\nabla
	\left (
	\dfrac{\widetilde w_t -w_0}{t}
	\right )+
	\int_\om \CC^{D_0} \nabla \dot{w}_0\cdot \nabla \dot{w}_0 	\overset{t\rightarrow 0^+}{\longrightarrow}0.
	\end{aligned}{}
\end{equation*}
%
%
Moreover, by Poincar\'e inequality, $
\dfrac{\widetilde w_t -w_0}{t}
$ strongly converges to ${w}_0$ in $H^1(\om)$.

Next, since $\Phi_t=Id$ in $\hbox{supp}(\widetilde{f})$, we have that $u_t\circ \Phi_t = w_t\circ \Phi_t+ \widetilde{f}\circ \Phi_t=\widetilde{w}_t+\widetilde{f}$. On the other hand, $u_0=w_0+\widetilde{f}$ and therefore
 \begin{equation}
 	\label{eq:GAT22.2}
 	\frac{u_t\circ \Phi_t-u_0}{t}=\frac{\widetilde{w}_t-w_0}{t},
 \end{equation} 
so that the material derivative $\dot u_0$ exists and coincides with $\dot w_0$. By \eqref{eq:GAT17.1}, recalling that $\mathcal{U}\equiv 0$ in 
$\hbox{supp}(\widetilde{f})$, it follows that $\dot u_0$ satisfies \eqref{eq:GAT2.1}.
\end{proof}
In order to generalize the definition of the material derivative of $u_t$ for $t_0 \in (0,1]$, we define the following map:
\begin{equation}
	\label{eq:GAT25.1}
	\Phi_{t_0,t}: \om \rightarrow \om, \quad 
	\Phi_{t_0,t}=Id+t\mathcal{U} \circ \Phi_{t_0}^{-1}, \quad t_0 \in (0,1], \ t \in [-t_0,1-t_0].
\end{equation}
\begin{definition}
	\label{def:material-derivative-t0}
	The material derivative $\dot{u}_{t_0} \in H^1_0(\om)$ of $u_t$ in $t=t_0 \in (0,1]$, when it exists, is defined as
	\begin{equation}
		\label{eq:GAT2.2-t0}
		\dot{u}_{t_0} = 
		\lim_{ t \rightarrow 0} \dfrac{u_{t_0+t}   \circ \Phi_{t_0,t} - u_{t_0} }{t} ,
	\end{equation}
	where the limit is meant in the  $H^1(\om)$-topology.
\end{definition}
	
\begin{proposition}
	\label{prop:GAT28.1}
	For every $t_0 \in (0,1]$ there exists the material derivative $\dot{u}_{t_0} \in H^1_0(\om)$ and it satisfies 
	\begin{equation}
			\begin{aligned}{}
		\label{eq:GAT28.2}
		&\int_\om \CC^{D_{t_0}} \nabla \dot{u}_{t_0} \cdot \nabla \psi =\\
		& = - \int_\om 
		\left \{
		(\CC^{D_{t_0}} \nabla u_{t_0} ) \left (  \dive (\mathcal U  \circ \Phi_{t_0}^{-1}   ) I_3 - 
		( D   (   \mathcal U \circ \Phi_{t_0}^{-1} ) )^T  \right )
		- 
		\CC^{D_{t_0}} (   \nabla u_{{t_0}}  D (   \mathcal U \circ \Phi_{t_0}^{-1} )  )
		\right \} \cdot \nabla \psi, 
		\end{aligned}
	\end{equation}
	for every $\psi \in H^1_0(\om)$.
\end{proposition}
\begin{proof}
The proof follows the lines of Proposition \ref{prop:GAT2.1}, by noticing that  $\CC^{D_{t_0+t}}(x)=\CC^{D_{t_0}}((\Phi_{t_0,t})^{-1}(x))$ and operating the change of variables  $x=\Phi_{t_0,t}(y)$ inside the integrals. 
\end{proof}

Given $f, g \in H^{1/2}_{co}(\Sigma)$ and for $t \in [0,1]$, let $u_t$, $v_t$ be the solutions to
\begin{equation}
	\label{eq:GAT29.1}
	\left\{ \begin{array}{ll}
		\dive \left ( \CC^{D_t} \nabla u_t \right ) =0,
		&  \hbox{in } \om,\\
		&  \\
		u_t =f,      
		&    \hbox{on } \partial \om,\\
	\end{array}\right.
\end{equation}
\begin{equation}
	\label{eq:GAT29.2}
	\left\{ \begin{array}{ll}
		\dive \left ( \CC^{D_t} \nabla v_t \right ) =0,
		&  \hbox{in }\om,\\
		&  \\
		v_t =g,    
		&    \hbox{on }\partial \om.\\
	\end{array}\right.
\end{equation}
Let us notice that, since $\mathcal U=0$ on $\partial \Omega$, $\CC^{D_t}= \CC^e$  on $\partial \Omega$, for $t \in [0,1]$.
\begin{definition}
	\label{def:Fgrande}
	Given $f, g \in H^{1/2}_{co}(\Sigma)$ and $t \in [0,1]$, let us define
	\begin{equation}
		\label{eq:GAT29.3}
		F(t,f,g)= \int_\Omega \CC^{D_t} \nabla u_t \cdot \nabla v_t = \int_{\partial \Omega} (\CC^{e} \nabla u_t ) n \cdot g = \langle \Lambda_{D_t}^\Sigma f, g \rangle.
	\end{equation}
\end{definition}
\begin{proposition}
	\label{prop:Derivata-di-F}
For every $t_0 \in [0,1]$ there exists the derivative of $F(t,f,g)$ in $t_0$:
\begin{multline}
	\label{eq:GAT33.2}
	F'(t_0,f,g) = \\
	=\int_\Omega 
	\left \{
	\left (\CC^{D_{t_0}} \nabla u_{t_0} \right )
	\left [
	\dive 
	\left (
	\mathcal U \circ \Phi_{t_0}^{-1}
	\right )
	I_3 -
	\left (
	D
	(
	\mathcal U \circ  \Phi_{t_0}^{-1}
	)
	\right )^T	
	\right ]
	-
	\CC^{D_{t_0}} 
	\left (
	\nabla u_{t_0} 
	D
	(
	\mathcal U \circ  \Phi_{t_0}^{-1}
	)	
	\right )
	\right \}
	\cdot 
	\nabla v_{t_0}.
\end{multline}
\end{proposition}
\begin{proof}
Let us see the details in the case $t_0=0$, since the general case can be handled similarly. Since $\Phi_t=Id$ on $\partial \Omega$, we have $u_t \circ \Phi_t = u_t$ on $\partial \Omega$ and $ \left ( \CC^e \nabla \left (   \dfrac{u_t - u_0}{t}     \right ) \right ) n $ converges to $ (\CC^e \nabla  \dot{u}_0) n $ weakly in $H^{-1/2}(\partial \Omega)$ as $t \rightarrow 0$. Therefore
\begin{equation}
	\label{eq:GAT30.4}
	\dfrac{F(t,f,g)-F(0,f,g)}{t} =
	\left \langle
	\left (
	\CC^e 
	\nabla \left (   \dfrac{u_t - u_0}{t}     \right ) 
	\right )
	n,g
	\right \rangle
	\rightarrow
	\left \langle
	\left (
	\CC^e 
	\nabla \dot{u}_0
	\right )
	n,g
	\right \rangle,
	 \quad
	 \hbox{as } t \rightarrow 0.
\end{equation}
By the weak formulation \eqref{eq:GAT2.1} for $\dot{u}_0$, we have
\begin{equation}
	\label{eq:GAT31.1}
	\dive 
	\left (
	\CC^{D_0} \nabla \dot{u}_0 + 
	\left ( \CC^{D_0} \nabla u_0 \right ) \left (  (\dive \, \mathcal U) I_3 - ( D \mathcal U  )^T  \right )
	- 
	\CC^{D_0} (   \nabla u_0  D \mathcal U  )
	\right )=0, \quad \hbox{in } \Omega.
\end{equation}
By multiplying \eqref{eq:GAT31.1} by $v_0$, integrating by parts in $\om$ and recalling that $D\mathcal{U}=0$, $v_0=g$, $\CC^{D_0}=\CC^e$ on $\partial \Omega$, we obtain 
\begin{equation}
	\label{eq:GAT32.1}
	\int_{\partial \Omega} ( \CC^e  \nabla \dot{u}_0  )n \cdot g =
	\int_{\om} 
	\left 
	\{
	\CC^{D_0} \nabla \dot{u}_0 + 
	\left ( \CC^{D_0} \nabla u_0  \right ) \left (  (\dive \, \mathcal U) I_3 - ( D \mathcal U  )^T  \right )
	- 
	\CC^{D_0} (   \nabla u_0  D \mathcal U  )
	\right \}
	\cdot
	\nabla v_0.
\end{equation}
Let us notice that
\begin{equation}
	\label{eq:GAT32.2}
	\int_\om \CC^{D_0} \nabla \dot{u}_0 \cdot \nabla v_0=0.
\end{equation}
In fact, by using the major symmetry \eqref{Notaz-simmetrie-C} of $\CC^{D_0}$, integrating by parts and recalling that $\dot{u}_0 \in H^1_0(\Omega)$, we have
\begin{equation*}
	\int_\om \CC^{D_0} \nabla \dot{u}_0 \cdot \nabla v_0=
	\int_\om \CC^{D_0} \nabla {v}_0 \cdot \nabla \dot{u}_0=
	\int_{\partial \Omega} \left ( \CC^{D_0} \nabla {v}_0 \right )n \cdot \dot{u}_0 - 
	\int_\om   \left (  \dive \left (   \CC^{D_0} \nabla {v}_0 \right )   \right ) \cdot \dot{u}_0=0.
\end{equation*}
By \eqref{eq:GAT30.4}, \eqref{eq:GAT32.1}, \eqref{eq:GAT32.2} we have
\begin{equation}
	\label{eq:GAT32.3}
	F'(0,f,g) = \int_\Omega 
	\left \{
	\left (\CC^{D_{0}} \nabla u_{0} \right )
	\left (
	(\dive \,
	\mathcal U 
	)
	I_3 -
	\left (
	D
	\mathcal U 
	\right )^T	
	\right )
	-
	\CC^{D_{0}} 
	\left (
	\nabla u_{0} 
	D
	\mathcal U 
	\right )
	\right \}
	\cdot 
	\nabla v_{0}.
\end{equation}
\end{proof}
\begin{proposition}
	\label{prop:Continuità-derivata-di-F}
	There exists a constant $C>0$ only depending on the a priori data such that, for every $t \in [0,1]$ and for every $f,g \in H^{1/2}_{co}(\Sigma)$,
	\begin{equation}
		\label{eq:GAT33.2bis}
		|F'(t,f,g)- F'(0,f,g)| 
		\leq C t \frac{d_H^2}{r_0} \|f\|_{H^{1/2}_{co}(\Sigma)}
		\|g\|_{H^{1/2}_{co}(\Sigma)}.
	\end{equation}
\end{proposition}
\begin{proof}
By \eqref{eq:GAT33.2} and \eqref{eq:GAT32.3} we have
\begin{equation*}
	\begin{aligned}{}
		&
	F'(t,f,g) -F'(0,f,g)= 
	\\
	&
	=\int_\Omega 
	\left \{
	\left (\CC^{D_{t}} \nabla u_{t} \right )
	\left [
	\dive 
	\left (
	\mathcal U \circ \Phi_{t}^{-1}
	\right )
	I_3 -
	\left (
	D
	(
	\mathcal U \circ  \Phi_{t}^{-1}
	)
	\right )^T	
	\right ]
	-
	\CC^{D_{t}} 
	\left (
	\nabla u_{t} 
	D
	(
	\mathcal U \circ  \Phi_{t}^{-1}
	)	
	\right )
	\right \}
	\cdot 
	\nabla v_{t}
	-
	\\
	&
	-
	\int_\Omega 
	\left \{
	\left (\CC^{D_{0}} \nabla u_{0} \right )
	\left (
	(\dive \,
	\mathcal U 
	)
	I_3 -
	\left (
	D
	\mathcal U 
	\right )^T	
	\right )
	-
	\CC^{D_{0}} 
	\left (
	\nabla u_{0} 
	D
	\mathcal U 
	\right )
	\right \}
	\cdot 
	\nabla v_{0}.
	\end{aligned}
\end{equation*}
To simplify the notation, let us set $H= (D \Phi_t)^{-1}$. Let us apply the change of variables $x=\Phi_t(y)$ in the first integral, obtaining 
\begin{equation*}
	\begin{aligned}{}
		&
		F'(t,f,g) -F'(0,f,g)= I_1 + I_2 + I_3,
	\end{aligned}
\end{equation*}
where
\begin{equation*}
	\begin{aligned}{}
		&
		I_1 = 
		\int_\om
		\det (D\Phi_t) \hbox{tr}(D\mathcal UH)
		\CC^{D_0} ( \nabla \widetilde u_t H   ) \cdot ( \nabla \widetilde v_t H   )- (\dive \, {\mathcal U}) \, \CC^{D_0} \nabla u_0 \cdot \nabla v_0,
	\end{aligned}
\end{equation*}
\begin{equation*}
	\begin{aligned}{}
		&
		I_2 =
		- 
		\int_\om
		\det (D\Phi_t) 
		\left (
		\CC^{D_0}
		( \nabla \widetilde u_t H  )
		\right )
		(H^T D\mathcal U^T)
		\cdot 
		( \nabla \widetilde v_t H   )
		-
		\left ( \CC^{D_0} \nabla u_0  \right )
		D \mathcal U^T
		\cdot 
		\nabla v_0,
	\end{aligned}
\end{equation*}
\begin{equation*}
	\begin{aligned}{}
		&
		I_3 =
		- 
		\int_\om
		\det (D\Phi_t) 
		\left (
		\CC^{D_0}
		(
		\nabla \widetilde u_t H \,
		D\mathcal U \, H
		)
		\right )
		\cdot
		(
		\nabla
		\widetilde v_t H
		)
		-
		\CC^{D_0}
		(
		\nabla u_0
		\,
		\nabla \mathcal U
		)
		\cdot
		\nabla 
		v_0,
	\end{aligned}
\end{equation*}
where $\widetilde u_t=u_t\circ\Phi_t$, $\widetilde v_t=v_t\circ\Phi_t$.

Let us estimate $I_2$, the bounds of $I_1$ and $I_3$ being analogous.

Let us rewrite $I_2$ as
\begin{equation*}
	\label{eq:GAT36.I2'+I2''}
	I_2 = I_2' + I_2'',
\end{equation*}
where
\begin{equation*}
	\begin{aligned}{}
		&
		I_2' =
		\int_\om
		(1-\det (D\Phi_t)) 
		\left (
		\CC^{D_0}
		( \nabla \widetilde u_t H  )
		\right )
		(H^T D\mathcal U^T)
		\cdot 
		( \nabla \widetilde v_t H   ),
	\end{aligned}
\end{equation*}
\begin{equation*}
	\begin{aligned}{}
		&
		I_2'' =
		- 
		\int_\om
		\left (
		\CC^{D_0}
		( \nabla \widetilde u_t H  )
		\right )
		(H^T D\mathcal U^T)
		\cdot 
		( \nabla \widetilde v_t H   )
		-
		\left ( \CC^{D_0} \nabla u_0  \right )
		D \mathcal U^T
		\cdot 
		\nabla v_0,
	\end{aligned}
\end{equation*}
By \eqref{ass4:U},  \eqref{eq:property_phi_1ter}, \eqref{eq:DU-infty-bound}, \eqref{eq:property_phi_3}, we have that

\begin{equation*}
	|I'_2|\leq Ct\left(\frac{d_H}{r_0}\right)^2\int_\om |\nabla \widetilde u_t||\nabla \widetilde v_t|.
\end{equation*}

	By operating the change of variables $x=\Phi_t(y)$, applying Schwarz inequality and recalling that $u_t$, $v_t$ are solutions to the direct problems \eqref{eq:GAT29.1}, \eqref{eq:GAT29.2} respectively, 
	
	\begin{equation*}
		|I'_2|\leq Ctr_0d_H^2\|\nabla u_t\|_{L^2(\om)}
		\|\nabla v_t\|_{L^2(\om)}\leq Ct\frac{d_H^2}{r_0}\|f\|_{H^{1/2}_{co}(\Sigma)}
		\|g\|_{H^{1/2}_{co}(\Sigma)},
	\end{equation*}
	with $C$ only depending in $\lambda^i$, $\mu^i$, $\lambda^e$, $\mu^e$, $\theta_0$, $M_0$, $M_1$.
By using \eqref{eq:property_phi_2bis}, we may rewrite $I''_2$ as
	\begin{equation*}
	I''_2=B_1+B_2+tA_1+t^2A_2+t^3A_3,
\end{equation*}
	where
	\begin{equation*}
	B_1=-\int_\om[\left(\CC^{D_0}\nabla(\widetilde u_t-u_0)\right)D\mathcal{U}^T]\cdot\nabla \widetilde v_t,
	\end{equation*}
		\begin{equation*}
		B_2=-\int_\om[\left(\CC^{D_0}\nabla u_0\right)D\mathcal{U}^T]\cdot\nabla (\widetilde v_t-v_0),
	\end{equation*}
			\begin{equation*}
		A_1=\int_\om[\left(\CC^{D_0}\nabla \widetilde u_t\right)D\mathcal{U}^T]\cdot
		(\nabla \widetilde v_t D\mathcal{U}H)+
		[\left(\CC^{D_0}\nabla \widetilde u_t\right)(H^TD\mathcal{U}^T D\mathcal{U}^T)]\cdot
		\nabla \widetilde v_t+
			[\CC^{D_0}(\nabla \widetilde u_tD\mathcal{U}H)D\mathcal{U}^T]\cdot
		\nabla \widetilde v_t
		,
	\end{equation*}
	\begin{equation*}
		\begin{aligned}{}
			&
		A_2=-\int_\om	[\left(\CC^{D_0}\nabla \widetilde u_t\right)(H^TD\mathcal{U}^TD\mathcal{U}^T)]\cdot
		(\nabla \widetilde v_t D\mathcal{U}H)+
			[(\CC^{D_0}(\nabla \widetilde u_tD\mathcal{U}H))D\mathcal{U}^T]\cdot
		(\nabla \widetilde v_t D\mathcal{U}H)+\\
		&+
			[\left(\CC^{D_0}(\nabla \widetilde u_t D\mathcal{U}H)\right)(H^TD\mathcal{U}^TD\mathcal{U}^T)]\cdot
		\nabla \widetilde v_t
		,
		\end{aligned}
	\end{equation*}
	\begin{equation*}
			A_3=\int_\om
			[\CC^{D_0}(\nabla \widetilde u_t D\mathcal{U}H)](H^TD\mathcal{U}^TD\mathcal{U}^T)
			\cdot
			(\nabla \widetilde v_t D\mathcal{U}H).
	\end{equation*}

By the same arguments used in the estimate of $I'_2$, we have that
\begin{equation*}
	\left|tA_1+t^2A_2+t^3A_3\right|\leq Ct\frac{d_H^2}{r_0}\|f\|_{H^{1/2}_{co}(\Sigma)}
	\|g\|_{H^{1/2}_{co}(\Sigma)},
\end{equation*}
where the constant $C$ depends on the a priori data.

Under the notation of Proposition \ref{prop:GAT2.1}, 
recalling \eqref{eq:GAT22.2}, we have that $\widetilde u_t-u_0=\widetilde w_t-w_0$. By \eqref{eq:GAT15.0}, we obtain
\begin{equation*}
	%
	\left\|\nabla \left(	
	\widetilde w_t -w_0	\right)\right\|_{L^2(\om)}\leq
	Ct\frac{d_H}{r_0}
	\|\nabla w_0\|_{L^2(\om)},
\end{equation*}
with $C$ only depending on $\alpha_0$, $\gamma_0$, $\lambda^i$, $\mu^i$, $\lambda^e$, $\mu^e$, $\theta_0$, $M_1$. Since $w_0=u_0-\widetilde f$, we have
\begin{equation*}
	%
	\|\nabla w_0\|_{L^2(\om)}
	\leq\|\nabla u_0\|_{L^2(\om)}+\|\nabla \widetilde f\|_{L^2(\om)}
	\leq \frac{C}{r_0}\|f\|_{H^{1/2}_{co}(\Sigma)}
	,
\end{equation*}
with $C$ only depending on $M_0$, $\lambda^i$, $\mu^i$, $\lambda^e$, $\mu^e$. Therefore
\begin{equation*}
	%
	\left\|\nabla \left(	
	\widetilde u_t -u_0	\right)\right\|_{L^2(\om)}\leq
	Ct\frac{d_H}{r_0^2}
	\|f\|_{H^{1/2}_{co}(\Sigma)}
	,
\end{equation*}
so that
\begin{equation*}
	%
	|B_1|\leq
	Ct\frac{d_H^2}{r_0}\|f\|_{H^{1/2}_{co}(\Sigma)}
	\|g\|_{H^{1/2}_{co}(\Sigma)},
\end{equation*}
with $C$ only depending on the a priori data. The term $B_2$ can be estimated analogously.
\end{proof}
\begin{remark}
Note that in the above arguments, we never used the fact that $\Phi_t$ was constructed so that \eqref{eq:D0 va in D1} holds. This means that Propositions \ref{prop:Derivata-di-F}  and \ref{prop:Continuità-derivata-di-F} hold true for the derivative in any arbitrary direction $\mathcal{U}$ as long as \eqref{ass2:U} and \eqref{ass4:U} hold.
\end{remark}

\subsection{Lower bound of the Gateaux derivative}
\label{lower-bound}

\begin{proposition}
\label{prop:LBD1.1}
There exists a constant $m_0 >0$, only depending on the a priori data, such that
\begin{equation}
	\label{eq:LBD1.1}
	\| F'(0)\|_{*} \geq m_0 d_H,	
\end{equation}
where
\begin{equation*}
	\| F'(0)\|_{*} =  \sup 
	\left \{
	\dfrac
	{|F'(0,f,g)|}
	{\|f\|_{H^{1/2}_{co}(\Sigma)}
		\|g\|_{H^{1/2}_{co}(\Sigma)}	}
		,
		\ 
		f,g \not\equiv 0
	\right \}.
\end{equation*}
\end{proposition}
The proof of the above proposition is rather involved and therefore we find it convenient to premise some preliminary results.
First of all, we rewrite the expression of $F'(0,f,g)$ given in \eqref{eq:GAT32.3} in an equivalent form suitable for next developments.
\begin{lemma}
	\label{lem:funz-integranda-di-F'}
	\begin{equation}
		\label{eq:LBD2.1}
		F'(0,f,g) = 
		\int_\Omega 
		(
		\dive \,
		\mathcal U
		)
		\CC^{D_{0}} \nabla u_{0}
		\cdot
		\nabla v_{0}
		-
		\CC^{D_{0}} \nabla u_{0}
		\cdot 
		(
		\nabla v_0 D\mathcal U
		)
		-
		\CC^{D_{0}} \nabla v_{0}
		\cdot 
		(
		\nabla u_0 D\mathcal U
		)
		.
	\end{equation}
\end{lemma}
\begin{proof}
The proof follows by elaborating the integrand function appearing in \eqref{eq:GAT32.3} using the major symmetry \eqref{Notaz-simmetrie-C} of the elastic tensor $\CC^{D_0}$ and the identity \eqref{eq:identity-3-matrices}.
\end{proof}
Let us define the vector $W$ whose components are the displacements of the vertices induced by the deformation of the polyhedron $D_0$ into $D_1$:
\begin{equation*}
	W = (V_1^{D_1}-V_1^{D_0}, \dots, V_N^{D_1}-V_N^{D_0}  ).
\end{equation*}
By \eqref{eq:distvert}, we have
\begin{equation}
	\label{eq:LBD3.2}
	\sqrt{N}C^{-1} d_H \leq |W| \leq \sqrt{N}C d_H,
\end{equation}
where the constant $C>1$ only depends on $M_0$, $M_1$ and $\theta_0$.

It is not restrictive to assume $|W|>0$. Let us normalize the vector field $\mathcal U$ as follows
\begin{equation}
	\label{eq:LBD3.3}
	\widetilde{\mathcal U} = \dfrac{\mathcal U}{|W|}
\end{equation}
and let us define 
\begin{equation}
	\label{eq:LBD3.4}
	H(f,g) = \dfrac{F'(0,f,g)}{|W|},
\end{equation}
so that
\begin{equation*}
	H(f,g) = 
	\int_\Omega 
	(
	\dive \, 
	\widetilde{\mathcal U}
	)
	\CC^{D_{0}} \nabla u_{0}
	\cdot
	\nabla v_{0}
	-
	\CC^{D_{0}} \nabla u_{0}
	\cdot 
	(
	\nabla v_0 D\widetilde{\mathcal U}
	)
	-
	\CC^{D_{0}} \nabla v_{0}
	\cdot 
	(
	\nabla u_0 D\widetilde{\mathcal U}
	)
	.
\end{equation*}
\begin{lemma}
	\label{lem:LBD10.1}
In $\om \setminus D_0$ and in $D_0$ separately, we have
\begin{equation}
	\label{eq:LBD10.1}
	(
	\dive \,
	\widetilde{\mathcal U}
	)
	\CC^{D_{0}} \nabla u_{0}
	\cdot
	\nabla v_{0}
	-
	\CC^{D_{0}} \nabla u_{0}
	\cdot 
	(
	\nabla v_0 D\widetilde{\mathcal U}
	)
	-
	\CC^{D_{0}} \nabla v_{0}
	\cdot 
	(
	\nabla u_0 D\widetilde{\mathcal U}
	)
	=
	\dive \, b,
\end{equation}
where
\begin{equation}
	\label{eq:LBD10.2}
	b=
	(\CC^{D_{0}} \nabla u_0 \cdot \nabla v_0)\widetilde{\mathcal U}
	-
	(\CC^{D_{0}} \nabla u_0)  ( \nabla v_0 \, \widetilde{\mathcal U} )
	-
	(\CC^{D_{0}} \nabla v_0)  ( \nabla u_0 \, \widetilde{\mathcal U} )
	.
\end{equation}
\end{lemma}
\begin{proof}
To simplify the notation, let us set $u=u_0$, $v=v_0$, $\CC=\CC^{D_0}$, $\mathcal U =\widetilde{\mathcal U}$, $T= \CC \nabla u$. Let us notice that $T=T^T$ by the minor symmetry of $\CC$.

We have
\begin{equation}
	\label{eq:LBD11.1}
	\begin{aligned}{}
		&
		\dive ((T  \cdot \nabla v)\mathcal U)=
		( T_ {ij} v_{i,j}\mathcal U_k)_{,k}
		=	T_ {ij,k} v_{i,j}\mathcal U_k +
		T_ {ij} v_{i,jk}\mathcal U_k +
		T_ {ij} v_{i,j}\mathcal U_{k,k}=
		\\
		&
		=
		(  \CC \nabla u   )_{ij,k} v_{i,j} \mathcal U_k
		+
		\CC \nabla u \cdot (( \nabla^2 v ) \mathcal U)
		+
		(\dive \mathcal U) \CC \nabla u \cdot \nabla v,
	\end{aligned}
\end{equation}
where we have defined $(( \nabla^2 v )\mathcal U)_{ij}= v_{i,jk}\mathcal U_k$.

Moreover, by the symmetry of $T$ and of the Hessian matrix, we have
\begin{equation*}
	\begin{aligned}{}
		&
		\dive ((T  \nabla v) \mathcal U)=
		( T_ {ij} v_{j,k}\mathcal U_k)_{,i}
		=
		T_{ij,i}v_{j,k}\mathcal U_k + T_{ij}v_{j,ki}\mathcal U_k
		+
		T_{ij}v_{j,k}\mathcal U_{k,i}=
		\\
		&
		=(\dive ( \CC \nabla u) ) \cdot (  \nabla v \, \mathcal U)
		+
		\CC \nabla u \cdot (\nabla^2 v \, \mathcal U )
		+
		\CC \nabla u \cdot (  \nabla v D \mathcal U  ).
	\end{aligned}
\end{equation*}
Similarly, inverting the role of $u$ and $v$, 
\begin{equation*}
	\begin{aligned}{}
		&
		\dive ((T  \nabla u) \mathcal U)=
		(\dive ( \CC \nabla v) ) \cdot (  \nabla u \, \mathcal U)
		+
		\CC \nabla v \cdot (\nabla^2 u \, \mathcal U )
		+
		\CC \nabla v \cdot (  \nabla u D \mathcal U  ).
	\end{aligned}
\end{equation*}
By using the major symmetry \eqref{Notaz-simmetrie-C} of $\CC$ and the fact that $\CC$ is constant in $\om \setminus D_0$ and in $D_0$, the first term in \eqref{eq:LBD11.1} can be rewritten as
\begin{equation*}
	\begin{aligned}{}
		&
		(  \CC \nabla u   )_{ij,k} v_{i,j} \mathcal U_k
		=
		( C_{ijrs}u_{r,s}  )_{,k} v_{i,j} \mathcal U_k
		=
		u_{r,sk} C_{rsij} v_{i,j} \mathcal U_k
		=
		((  \nabla^2 u  ) \mathcal U) \cdot \CC \nabla v.
	\end{aligned}
\end{equation*}
Therefore, recalling that $u$ and $v$ are solutions to the homogeneous Lam\'e system in $\Omega \setminus D_0$ and in $D_0$ separately, the thesis follows.
\end{proof}
We preliminary observe the following result
\begin{lemma}\label{lemma:elast_mom_tens}
    In $\partial D_0 \setminus \bigcup_{i \neq j} \sigma_{ij}^{D_0}$, there exists a fourth-order tensor $\mathbb{M}$ such that
    \begin{equation}
    (\mathbb{C}^e-\mathbb{C}^i)\nabla u^i_0=\mathbb{M}\nabla u^e_0
\end{equation}
and satisfying 
\begin{equation}
	\label{eq:LBD-M-simm}
	\begin{aligned}{}
		&
		\mathbb M A \cdot B = A \cdot \mathbb M B, \quad \hbox{for every } \ 3\times 3 \ \hbox{symmetric matrices } A, B.
	\end{aligned}
\end{equation}
Moreover, if assumption \eqref{eq:monotony-Lamè-moduli-Ci-Ce} holds, then
\begin{equation}
	\label{eq:LBD-M-convex+}
	\begin{aligned}{}
		&
		{-}\mathbb M A \cdot A \geq  \sigma |A|^2, \quad \hbox{for every } \ 3\times 3 \ \hbox{symmetric matrix } A;
	\end{aligned}
\end{equation}
if, otherwise, assumption \eqref{eq:monotony-Lamè-moduli-Ce-Ci} holds, then 
\begin{equation}
	\label{eq:LBD-M-convex-}
	\begin{aligned}{}
		&
		\mathbb M A \cdot A \geq  \sigma |A|^2, \quad \hbox{for every } \ 3\times 3 \ \hbox{symmetric matrix } A,
	\end{aligned}
\end{equation}
where ${\sigma=\min\{2|\mu^e-\mu^i|,|2(\mu^e-\mu^i)+3(\lambda^e-\lambda^i)|\}}$.
\end{lemma}

\begin{proof}
We begin by deriving a general expression for the elastic tensor \( \mathbb{M} \), following the approach introduced in~\cite{FraMur86} and further developed in~\cite{beretta2012small}.
Specifically, we consider the transmission conditions on a face $F$ of the polyhedron \( D_0 \):
\begin{align}
\nabla u^e {\tau_j} &= \nabla u^i{\tau_j}, \qquad j=1,2 \label{eq:transm_cond1}\\
\mathbb{C}^e \widehat{\nabla} u^e n &= \mathbb{C}^i \widehat{\nabla} u^i n,\label{eq:transm_cond2}
\end{align}
where $\{\tau_1,\tau_2,n\}$ is an orthonormal system.
Using \eqref{eq:transm_cond1}, we can write:
\begin{equation}\label{eq:def_delta}
\nabla u^e = \nabla u^i + \delta \otimes n, 
\end{equation} 
where $\delta=(\nabla u^e - \nabla u^i)n$. 
We define the second-order tensor 
\( Q \in \mathbb{R}^{3 \times 3} \) by
\[ 
(Q^{-1} {\zeta}) \cdot {\eta} = \mathbb{C}^i({\zeta} \otimes n) \cdot ({\eta} \otimes n), \quad \forall {\zeta}, {\eta} \in \mathbb{R}^3. 
\]
Since \( \mathbb{C}^i \) is positive definite, so is \( Q^{-1} \), and thus \( Q \) is well-defined.\\
Using \eqref{eq:transm_cond2}, for all \( {\xi} \in \mathbb{R}^3 \):
\[ 
(Q^{-1} \delta) \cdot {\xi} = (\mathbb{C}^i - \mathbb{C}^e)\widehat{\nabla} u^e n \cdot \xi, 
\]
thus,
\[ 
Q^{-1} \delta = (\mathbb{C}^i - \mathbb{C}^e) \widehat{\nabla} u^e n,  
\]
which implies that
\[
\delta = Q \left((\mathbb{C}^i - \mathbb{C}^e) \widehat{\nabla} u^e n\right).
\]
Therefore, using the last equality in \eqref{eq:def_delta} we get 
\[ 
\nabla u^e = \nabla u^i + Q\left((\mathbb{C}^i - \mathbb{C}^e) \widehat{\nabla} u^e n\right) \otimes n. 
\]
Applying \( \mathbb{C}^e - \mathbb{C}^i \) to the last equality, we deduce that
\[
(\mathbb{C}^e - \mathbb{C}^i) \widehat{\nabla} u^i = (\mathbb{C}^e - \mathbb{C}^i) \widehat{\nabla} u^e - (\mathbb{C}^e - \mathbb{C}^i) \left[ Q \left((\mathbb{C}^i - \mathbb{C}^e) \widehat{\nabla} u^e n\right) \otimes n \right],
\]
hence, for every $3\times3$ symmetric matrix $A$, we define the elastic moment tensor \( \mathbb{M} \) by
\[
\mathbb{M}A = (\mathbb{C}^e - \mathbb{C}^i)A + (\mathbb{C}^e - \mathbb{C}^i) \left[ Q\left( (\mathbb{C}^e - \mathbb{C}^i)A n\right) \otimes n \right].
\]
The tensor \( \mathbb{M} \) inherits both minor and major symmetries. The minor symmetries follow directly {}from those of \( \mathbb{C}^e - \mathbb{C}^i \). The major symmetry follows {}from the symmetries of $\mathbb{C}^e - \mathbb{C}^i$, of $Q$, and {}from the fact that
\[
(\mathbb{C}^e - \mathbb{C}^i)[(Qv) \otimes n] \cdot B = (\mathbb{C}^e - \mathbb{C}^i)[(Qw) \otimes n] \cdot A,
\]
where 
\( 
v = (\mathbb{C}^e - \mathbb{C}^i)An, \; w = (\mathbb{C}^e - \mathbb{C}^i)Bn. 
\) \\
Assuming \eqref{eq:monotony-Lamè-moduli-Ce-Ci}, 
we infer the positivity of $\mathbb{M}$. In fact, for every $3\times3$ symmetric matrix $A$, 
\[
\mathbb{M}A \cdot A = (\mathbb{C}^e - \mathbb{C}^i)A \cdot A + Q\gamma \cdot \gamma,
\]
where $\gamma= (\mathbb{C}^e - \mathbb{C}^i)An$, hence, since $Q\gamma \cdot \gamma\geq 0$, for all $\gamma\in\mathbb{R}^3$, we get
\[
\mathbb{M}A \cdot A\geq (\mathbb{C}^e - \mathbb{C}^i)A \cdot A\geq \sigma |A|^2>0.
\]
By proceeding similarly if \eqref{eq:monotony-Lamè-moduli-Ci-Ce} holds, we have
$-\mathbb{M}A  \cdot A \geq \sigma |A|^2>0$ for every $3 \times 3$ symmetric matrix $A$.
\end{proof}

For the sake of clarity, let us explicitly denote by $b^e$, $b^i$ the restriction of the vector field $b$ as in \eqref{eq:LBD10.2} in $\Omega \setminus D_0$, $D_0$, respectively:
\begin{equation}
	\label{eq:LBD12.3}
	b= \left \{
	\begin{aligned}{}
		&
		b^e \quad \hbox{in } \Omega \setminus D_0,
		\\
		&
		b^i \quad \hbox{in } D_0.
	\end{aligned}
	\right .
\end{equation}
\begin{proposition}
	\label{prop:LBD15.1}
	In $\partial D_0 \setminus \bigcup_{i \neq j} \sigma_{ij}^{D_0}$
	the following identity holds
\begin{equation}
	\label{eq:LBD15.1}
	\begin{aligned}{}
		&
		(b^i-b^e) \cdot n ={-} ( \widetilde {\mathcal{U}} \cdot n  )
		\mathbb M \widehat{\nabla} u_0^e \cdot  \widehat{\nabla} v_0^e,
	\end{aligned}
\end{equation}
where $n$ is the unit outer normal to $D_0$, $u_0^e=u_0|_{\Omega \setminus D_0}$, $v_0^e=v_0|_{\Omega \setminus D_0}$, {and $\mathbb M$ is the fourth order tensor appearing in Lemma \ref{lemma:elast_mom_tens}}. 
\end{proposition}
\begin{proof}
To simplify the notation let us set $u=u_0$, $v=v_0$, $\mathbb{C}=\mathbb{C}^{D_0}$ and $\mathcal{U}=\widetilde{\mathcal{U}}$. Using \eqref{eq:LBD10.2}, we find that
\begin{equation}\label{eq:bi-be}
\begin{aligned}
    (b^i-b^e)\cdot n&=\left(\mathbb{C}^i\nabla u^i \cdot \nabla v^i -\mathbb{C}^e\nabla u^e \cdot \nabla v^e\right) (\mathcal{U}\cdot n)\\
    &-\left(\mathbb{C}^i\nabla u^i n \otimes \mathcal{U} \cdot \nabla v^i-\mathbb{C}^e\nabla u^e n \otimes \mathcal{U} \cdot \nabla v^e \right)\\
    &-\left(\mathbb{C}^i\nabla v^i n \otimes \mathcal{U} \cdot \nabla u^i-\mathbb{C}^e\nabla v^e n \otimes \mathcal{U} \cdot \nabla u^e \right):=T_1-T_2-T_3
\end{aligned}
\end{equation}
Now, using \eqref{eq:def_delta} and a similar expression for $v$, that is $\nabla v^e=\nabla v^i+\eta\otimes n$, where $\eta\in\mathbb{R}^3$ is such that $\eta=(\nabla v^e-\nabla v^i)n$, we get that
\begin{equation}\label{eq:T1}
\begin{aligned}
    T_1&=\left(\mathbb{C}^i\nabla u^i \cdot (\nabla v^e-\eta\otimes n)-\mathbb{C}^e(\nabla u^i+\delta\otimes n) \cdot \nabla v^e\right)(\mathcal{U}\cdot n)\\
    &=\left( (\mathbb{C}^i-\mathbb{C}^e)\nabla u^i \cdot \nabla v^e-\mathbb{C}^i\nabla u^i \cdot (\eta\otimes n)-\mathbb{C}^e\delta\otimes n \cdot \nabla v^e\right)(\mathcal{U}\cdot n)\\
    &=\left( -\mathbb{M}\nabla u^e \cdot \nabla v^e-\mathbb{C}^i\nabla u^i \cdot (\eta\otimes n)-\mathbb{C}^e\delta\otimes n \cdot \nabla v^e\right)(\mathcal{U}\cdot n).
\end{aligned}
\end{equation}
Moreover,
\begin{equation}\label{eq:T2}
    \begin{aligned}
        T_2&=\mathbb{C}^i\nabla u^i n \otimes \mathcal{U} \cdot \nabla v^i-\mathbb{C}^i\nabla u^i n \otimes \mathcal{U} \cdot (\nabla v^i+\eta\otimes n)\\
        &=-(\mathbb{C}^i\nabla u^i n \otimes \mathcal{U}) \cdot (\eta\otimes n)=-[(\mathbb{C}^i\nabla u^i \cdot (\eta\otimes n)](\mathcal{U}\cdot n).
    \end{aligned}
\end{equation}
Analogously,
\begin{equation}\label{eq:T3}
        T_3=-[(\mathbb{C}^e\nabla v^e) \cdot (\delta\otimes n)](\mathcal{U}\cdot n).
\end{equation}
Inserting \eqref{eq:T3}, \eqref{eq:T2} and \eqref{eq:T1} into \eqref{eq:bi-be}, we get the assertion \eqref{eq:LBD15.1}. 
\end{proof}

\begin{proof}[Proof of Proposition \ref{prop:LBD1.1}]
Let us define 
\begin{equation}
	\label{eq:LBD3.6}
	\begin{aligned}{}
		&
		m_1 = \| H\|_*,
	\end{aligned}
\end{equation}
so that
\begin{equation*}
	\begin{aligned}{}
		&
		|H(f,g)| \leq m_1 \|f\|_{H^{1/2}_{co}(\Sigma)} \|g\|_{H^{1/2}_{co}(\Sigma)}, \quad \hbox{for every } f,g \in 
				H^{1/2}_{co}(\Sigma),
	\end{aligned}
\end{equation*}
Recall that, by \eqref{eq:LBD3.4},
\begin{equation}
	\label{eq:LBD3.8}
	\begin{aligned}{}
		&
		\|F'(0)\|_* = |W| \| H\|_*.
			\end{aligned}
\end{equation}
\medskip
\textit{Step 1}. 

For $y, z \in B_{r_\sharp}(P_0)$, where $P_0$ satisfies \eqref{eq:ball}, and for $l$, $m \in \RR^3$, $|l|=|m|=1$, let us consider the functions 
\begin{equation}
	\label{eq:LBD5.2}
	\begin{aligned}{}
		&
		u_0(x,y;l)=  G_0^\sharp (x,y) l,
		\quad
		v_0(x,z;m)=  G_0^\sharp (x,z) m,
	\end{aligned}
\end{equation}
where the matrix $G_0^\sharp$ has been defined in Proposition \ref{prop:PGDzero}. Since $y, z \in \Omega_0$, the functions $u_0, v_0 \in H^1(\Omega)$.

We consider the special Dirichlet boundary data
\begin{equation*}
	\begin{aligned}{}
		&
		f^\sharp (x) = G_0^\sharp (x,y) l |_{x \in \partial \Omega} \in H^{1/2}_{co}(\Sigma),
		\quad 
		g^\sharp (x) = G_0^\sharp (x,z) m |_{x \in \partial \Omega} \in H^{1/2}_{co}(\Sigma).		
	\end{aligned}
\end{equation*}
Let us define 
\begin{equation}
	\label{eq:LBD6.5}
	\begin{aligned}{}
		\Theta(y,z;l,m) &= H(f^\sharp, g^\sharp)=\\
		&= 
			\int_\Omega 
			(
			\dive \,
			\widetilde{\mathcal U}
			)
			\CC^{D_{0}} \nabla u_{0}
			\cdot
			\nabla v_{0}
			-
			\CC^{D_{0}} \nabla u_{0}
			\cdot 
			(
			\nabla v_0 D\widetilde{\mathcal U}
			)
			-
			\CC^{D_{0}} \nabla v_{0}
			\cdot 
			(
			\nabla u_0 D\widetilde{\mathcal U}
			)
			.
	\end{aligned}
\end{equation}
By trace inequalities and by \eqref{eq:PG4B.5}, 
\begin{equation*}
	\begin{aligned}{}
		&
		\|f^\sharp\|_{H_{co}^{1/2}(\Sigma)} \leq C \| G_0^\sharp (\cdot,y)\|_{H^1(\Omega)}\leq C 		 \| G_0^\sharp (\cdot,y)\|_{H^1(\Omega^\sharp \setminus B_{r_\sharp}(y))}
		\leq \dfrac{C}{r_0},
	\end{aligned}
\end{equation*}
where $C>0$ only depends on the a priori data. A similar estimate holds for $g^\sharp$ and, therefore, for every $y,z \in B_{r_\sharp}(P_0)$ and for every $l,m \in \RR^3$, $|l|=|m|=1$, we have
\begin{equation}
	\label{eq:LBD7.4}
	\begin{aligned}{}
		&
		|\Theta (y,z;l,m)|  \leq \dfrac{Cm_1}{r_0^2},
	\end{aligned}
\end{equation}
where $C>0$ only depends on the a priori data.

\medskip

\noindent \textit{Step 2.}

Given $\sigma$, $0<\sigma\leq \frac{r_0}{2}$, let us consider the tubular neighborhood
of the edges $\left \{ \sigma_{ij}^{D_0}  \right  \}$, with $ i \neq j$, of width $\sigma$ (to be chosen later):
\begin{equation}
	\label{eq:LBD8.1}
	\begin{aligned}{}
		&
		\mathcal B 
		=
		\bigcup_{\overset{\scriptstyle i,j}{\scriptstyle
				i\neq j}}
		\left \{
		x \in \mathbb{R}^3 \ | \ \hbox{dist}(x,\sigma_{ij}^{D_0})
		\leq \sigma
				\right \},
	\end{aligned}
\end{equation}
which satisfies
\begin{equation*}
	\begin{aligned}{}
		&
		\hbox{dist} ( \mathcal B, \partial \Omega) \geq \dfrac{r_0}{2}.
	\end{aligned}
\end{equation*}
It is useful to rewrite $\Theta(y,z;l,m)$ in \eqref{eq:LBD6.5} as 
\begin{equation}
	\label{eq:LBD9.2}
	\begin{aligned}{}
		&
		\Theta(y,z;l,m) = 
		\int_{\Omega \setminus \mathcal B} 
		(
		\dive \,
		\widetilde{\mathcal U}
		)
		\CC^{D_{0}} \nabla u_{0}
		\cdot
		\nabla v_{0}
		-
		\CC^{D_{0}} \nabla u_{0}
		\cdot 
		(
		\nabla v_0 D\widetilde{\mathcal U}
		)
		-
		\CC^{D_{0}} \nabla v_{0}
		\cdot 
		(
		\nabla u_0 D\widetilde{\mathcal U}
		)
		+
		\\
		&
		+
		\int_{\mathcal B} 
		(
		\dive \,
		\widetilde{\mathcal U}
		)
		\CC^{D_{0}} \nabla u_{0}
		\cdot
		\nabla v_{0}
		-
		\CC^{D_{0}} \nabla u_{0}
		\cdot 
		(
		\nabla v_0 D\widetilde{\mathcal U}
		)
		-
		\CC^{D_{0}} \nabla v_{0}
		\cdot 
		(
		\nabla u_0 D\widetilde{\mathcal U}
		).
	\end{aligned}
\end{equation}
By applying Lemma \ref{lem:LBD10.1}, we may rewrite the first integral on the right hand side of \eqref{eq:LBD9.2} as 
\begin{equation}
	\label{eq:LBD12.2}
	\begin{aligned}{}
		&
		\int_{\Omega \setminus \mathcal B} 
		(
		\dive \,
		\widetilde{\mathcal U}
		)
		\CC^{D_{0}} \nabla u_{0}
		\cdot
		\nabla v_{0}
		-
		\CC^{D_{0}} \nabla u_{0}
		\cdot 
		(
		\nabla v_0 D\widetilde{\mathcal U}
		)
		-
		\CC^{D_{0}} \nabla v_{0}
		\cdot 
		(
		\nabla u_0 D\widetilde{\mathcal U}
		)
		=
		\\
		&
		= 
		\int_{\Omega \setminus ( D_0 \cup \mathcal B)}
		\dive \ b^e 
		+
		\int_{D_0 \setminus \mathcal B}
		\dive \ b^i.				 
	\end{aligned}
\end{equation}
By denoting with $n$ the unit outer normal either to $\mathcal B$ or to $D_0$ according to the case, integrating by parts in the first term on the right hand side of \eqref{eq:LBD12.2} we have
\begin{equation}
	\label{eq:LBD13.3}
	\begin{aligned}{}
		&
		\int_{\Omega \setminus ( D_0 \cup \mathcal B)}
		\dive \ b^e 
		=
		- \int_{\partial D_0 \setminus \mathcal B} b^e \cdot n - 	
		\int_{\partial \mathcal B \cap (\Omega \setminus D_0)} b^e \cdot n.
	\end{aligned}
\end{equation}
By integrating by parts in the second term, we have
\begin{equation}
	\label{eq:LBD14.1}
	\begin{aligned}{}
		&
		\int_{D_0 \setminus \mathcal B}
		\dive \ b^i 
		=
		\int_{\partial D_0 \setminus \mathcal B} b^i \cdot n - 	
		\int_{\partial \mathcal B \cap D_0} b^i \cdot n.
	\end{aligned}
\end{equation}
By \eqref{eq:LBD13.3} and \eqref{eq:LBD14.1}, we have
\begin{equation}
	\label{eq:LBD14.2}
	\begin{aligned}{}
		&
		\int_{\Omega \setminus \mathcal B} 
		(
		\dive \,
		\widetilde{\mathcal U}
		)
		\CC^{D_{0}} \nabla u_{0}
		\cdot
		\nabla v_{0}
		-
		\CC^{D_{0}} \nabla u_{0}
		\cdot 
		(
		\nabla v_0 D\widetilde{\mathcal U}
		)
		-
		\CC^{D_{0}} \nabla v_{0}
		\cdot 
		(
		\nabla u_0 D\widetilde{\mathcal U}
		)
		=
		\\
		&
		=
		\int_{\partial D_0 \setminus \mathcal B} (b^i-b^e) \cdot n - 	
		\int_{\partial \mathcal B \cap (\Omega \setminus D_0)} b^e \cdot n
		-
		\int_{\partial \mathcal B \cap D_0} b^i \cdot n.
	\end{aligned}
\end{equation}
By inserting \eqref{eq:LBD14.2} in \eqref{eq:LBD9.2}, and applying Proposition \ref{prop:LBD15.1}, we have
\begin{equation}
	\label{eq:LBD40.1}
	\begin{aligned}{}
		&
		\Theta(y,z;l,m) = 
		\int_{\mathcal B} 
		(
		\dive \,
		\widetilde{\mathcal U}
		)
		\CC^{D_{0}} \nabla u_{0}
		\cdot
		\nabla v_{0}
		-
		\CC^{D_{0}} \nabla u_{0}
		\cdot 
		(
		\nabla v_0 D\widetilde{\mathcal U}
		)
		-
		\CC^{D_{0}} \nabla v_{0}
		\cdot 
		(
		\nabla u_0 D\widetilde{\mathcal U}
		)
		-
		\\
		&		
		-
		\left (
		\int_{\partial \mathcal B \cap ( \Omega \setminus D_0  )}
		b^e \cdot n + 
		\int_{\partial \mathcal B \cap D_0 }
		b^i \cdot n
		\right )
		-
		\int_{\partial D_0 \setminus \mathcal B}
		( \widetilde {\mathcal U} \cdot n  )
		\mathbb M \widehat \nabla u_0^e \cdot  \widehat \nabla v_0^e
		=I_1^\Theta - I_2^\Theta - I_3^\Theta.		
	\end{aligned}
\end{equation}
\begin{remark}
Let us notice that the definition of $\Theta$ given in \eqref{eq:LBD40.1} is well-defined for every  ${y}, {z} \in \oms \setminus ( D_0 \cup \mathcal B  )$, since the domains of integration do not contain the variables $y$, $z$, which are the poles of the Green's functions $G_0^\sharp(\cdot,y)$, $G_0^\sharp(\cdot,z)$ entering in the definition \eqref{eq:LBD5.2} of $u_0$ and $v_0$.
\end{remark}

\textit{Step 3.} 

In this step we propagate the estimate \eqref{eq:LBD7.4} up to points which are close to the boundary of $D_0$, but far {}from its edges. 

Let $\overline{z}, \overline{y} \in \oms \setminus ( D_0 \cup \mathcal B  )$ and let us consider the vector fields
\begin{equation}
	\label{eq:LBD41.3}
	\begin{aligned}{}
		&
		f(y, \overline{z};m) = \sum_{r=1}^3 \Theta ( y,\overline{z};e_r,m )e_r,
	\end{aligned}
\end{equation}
\begin{equation}
	\label{eq:LBD41.4}
	\begin{aligned}{}
		&
		\widetilde{f}(\overline{y},z;l) = \sum_{r=1}^3 {\Theta} ( \overline{y},z;l,e_r )e_r.
	\end{aligned}
\end{equation}
As stated in the following lemma, the functions $f$ and $\widetilde{f}$ are solutions to the homogeneous Lam\'e system with constant elastic tensor $\mathbb C^{D_0}=\mathbb C^e$ in $\oms \setminus ( D_0 \cup \mathcal B  )$.
\begin{lemma}
	\label{lem:LBD41.1}
	For given $\overline{z} \in \oms \setminus ( D_0 \cup \mathcal B  )$ and $m \in \RR^3$, $|m|=1$, the function $f=f(y, \overline{z};m)$ is a solution to 
\begin{equation}
	\label{eq:LBD41.5}
	\dive_y ( \CC^{D_0}(y) \nabla_y f )=0, \quad \hbox{for every }
	y \in \oms \setminus ( D_0 \cup \mathcal B  ).
\end{equation}
For given $\overline{y} \in \oms \setminus ( D_0 \cup \mathcal B  )$ and $l \in \RR^3$, $|l|=1$, the function $\widetilde{f}=\widetilde{f}(\overline{y},z;l)$ is a solution to 
\begin{equation}
	\label{eq:LBD41.6}
	\dive_z ( \CC^{D_0}(z) \nabla_z \widetilde{f} )=0, \quad \hbox{for every }
	z \in \oms \setminus ( D_0 \cup \mathcal B  ).
\end{equation}
\end{lemma}
\begin{proof}
It is sufficient to prove \eqref{eq:LBD41.5} and, by the linearity of the expression in $\nabla u_0$, it is not restrictive to focus on the first term of the first integral:
\begin{equation*}
	\begin{aligned}{}
		&
		I_r(y, \overline{z}; m)=
		\int_{\mathcal B} 
		(
		\dive_x \,
		\widetilde{\mathcal U}(x)
		)
		\CC^{D_{0}}(x) \nabla_x u_{0}^{(r)}(x,y)
		\cdot
		\nabla_x v_{0}(x, \overline{z};m)
		,
	\end{aligned}
\end{equation*}
where
\begin{equation}
	\label{eq:LBD42.2}
	\begin{aligned}{}
		&
		u_{0}^{(r)}(x,y)
		=
		u_0(x,y;l=e_r)
		=
		G_0^\sharp (x,y)e_r
		,
		\quad 
		r=1,2,3.
	\end{aligned}
\end{equation}
In Cartesian components, we want to prove that
\begin{equation}
	\label{eq:LBD42.3}
	\begin{aligned}{}
		&
		\sum_{j=1}^3
		\dfrac{\partial }{\partial y_j} 
		\left (
		\sum_{r,k=1}^3
		C^{D_0}_{ijrk}(y)
		\dfrac{\partial I_r(y, \overline{z}; m)}{\partial y_k}	
		\right )
		=
		0, \quad i=1,2,3,
	\end{aligned}
\end{equation}
with
\begin{equation}
	\label{eq:LBD42.4}
	\begin{aligned}{}
		&
		I_r(y, \overline{z}; m) = \int_{\mathcal B}
		\left (
		\dive_x \widetilde{\mathcal U}(x)
		\right )
		\sum_{m,n,p,q=1}^3
		C^{D_0}_{mnpq}(x)
		\dfrac{\partial u_{0p}^{(r)}(x,y)}{\partial x_q}
		\dfrac{\partial v_{0m}(x, \overline{z};m)}{\partial x_n}.		
	\end{aligned}
\end{equation}
For a given $i$, $i=1,2,3$, by inserting the expression \eqref{eq:LBD42.4} of $f_r(y)$ in \eqref{eq:LBD42.3}, we obtain
\begin{equation}
	\label{eq:LBD43.1}
	\begin{aligned}{}
		&
		\sum_{j=1}^3
		\dfrac{\partial }{\partial y_j} 
		\left (
		\sum_{r,k=1}^3
		C^{D_0}_{ijrk}(y)
		\dfrac{\partial I_r(y, \overline{z}; m)}{\partial y_k}		
		\right )
		=	
		\\
		&
		=
		\int_{\mathcal B}
		\left (
		\dive_x \widetilde{\mathcal U}(x)
		\right )
		\sum_{m,n,p,q=1}^3
		C^{D_0}_{mnpq}(x)
		\dfrac{\partial }{\partial x_q}
		\left [
		\sum_{j=1}^3
		\dfrac{\partial }{\partial y_j}
		\left (
		\sum_{r,k=1}^3
		C^{D_0}_{ijrk}(y)
		\dfrac{\partial u_{0p}^{(r)}(x,y)}{\partial y_k}
		\right )
		\right ]
		\dfrac{\partial v_{0m}(x, \overline{z};m)}{\partial x_n}.
	\end{aligned}
\end{equation}
By the symmetry property \eqref{eq:PG4B.3} of the Green's function $G_0^\sharp$ and by the definition \eqref{eq:LBD42.2} of $u_0^{(r)}$, we have
\begin{equation*}
	\begin{aligned}{}
		&
	u_{0p}^{(r)}(x,y)= 	G_0^\sharp (x,y) e_r \cdot e_p = e_r \cdot G_0^\sharp (y,x)e_p = u_{0r}^{(p)}(y,x),
	\end{aligned}
\end{equation*}
the square bracket of \eqref{eq:LBD43.1} vanishes and equation \eqref{eq:LBD42.3} is satisfied.
\end{proof}
The smallness propagation argument is based on the following geometrical construction. Given a face $F_j$ of $D_0$ and a triangle $T \in \mathcal{T}_0$, $ T\subset F_j$ with base $\sigma_{ij}$, let $P$ the incenter of $T$. 
By elementary calculations, we have that 
\begin{equation}
	\label{distPbordoT}
\hbox{dist}(P, \partial T)\geq c_0r_0, 
\end{equation}
with $c_0=\frac{\tan \gamma}{2M_1+2\sqrt{M_1^2+\tan^2\gamma}}$.
Moreover, $\hbox{dist}(P,F_i)\geq( c_0\sin\theta_0)r_0$, for any face $F_i$ adjacent to $F_j$ and, by the estimate \eqref{dist-facce-non-adia} obtained in the proof of Proposition \ref{vectorfield} in Section \ref{Appendix},  $\hbox{dist}(P,F_i)\geq \cos(\arctan(M_0))r_0$, for any face $F_i$ not adjacent to $F_j$. Therefore, denoting $c_1=\min\{c_0\sin\theta_0,\cos(\arctan(M_0))  \}$, we have that $\hbox{dist}(P,\mathcal D)\geq c_1r_0$.

 Furthermore, there exists a positive constant $c_2$, $0<c_2<\frac{1}{2}$, $c_2$ only depending on $M_0$ and $\theta_0$, such that the set $F_r=\{x\in \oms \ |\ \hbox{dist}(x,D_0)\geq r)\}$ is connected for every $r\leq c_2r_0$, see \cite[Proposition 5.5]{ARRV} for details.

 Let $c_*=\frac{1}{2}\min \{
 c_1, c_2, \zeta\}$, only depending on $\theta_0$, $M_0$ and $M_1$, where $r_\sharp=\zeta r_0$ has been introduced in \eqref{eq:def-r-sharp}.

 Let $r^*=c_*r_0$ and let us choose $\sigma=\frac{r^*}{4}$ as width of the tubular neighborhood $\mathcal B$ introduced in \eqref{eq:LBD8.1}.
 	
 Let $n_P$ be the outer unit normal to $\partial D_0$ at $P$ and let us consider the point $Q=P+r^*n_P$. We have that
 $\hbox{dist}(Q, F_j)=r^*$,  $\hbox{dist}(Q, \partial D_0\setminus F_j)\geq \hbox{dist}(P, \partial D_0\setminus F_j)-r^*\geq 2r^*-r^*=r^*$, so that  $\hbox{dist}(Q, D_0)= r^* $.
 
 Since both the points $P_0$ and $Q$ belongs to $F_{r^*}$, there exists a curve
$\widetilde{\mathfrak{c}}$, $\widetilde{\mathfrak{c}}\subset F_{r^*}$ joining 
$P_0$ and $Q$. Let $S=\left\{P+t n_P,\ t\in\left[0,r^*\right]\right\}$ be the segment joining $P$ and $Q$. Let 
\begin{equation*}
	\begin{aligned}
		\mathcal{K}&=\left\{x\in \RR^3:\ dist(x,\widetilde{\mathfrak{c}})< \frac{r^*}{2}\right\}\cup C^{r^*/2}_{PQ},
	\end{aligned}
\end{equation*}
where $C^{r^*/2}_{PQ}$ denotes the closed cylinder with axis the segment $S$ and radius $\frac{r^*}{2}$. 
We have that $\mathcal{K}\subset \overline{\oms\setminus D_0}$. Moreover, $\hbox{dist}(P,\mathcal{B}) \geq c_1 r_0 -\frac{r^*}{4} \geq 2r^* -\frac{r^*}{4} = \frac{7r^*}{4}$, so that $\hbox{dist}(S,\mathcal{B}) \geq \hbox{dist}(P,\mathcal{B}) - r^* \geq \frac{7r^*}{4}-r^*= \frac{3r^*}{4}$. Since $\hbox{dist} (\widetilde{\mathfrak{c}}, \mathcal{B}) \geq \hbox{dist} (\widetilde{\mathfrak{c}}, D_0) - \frac{r^*}{4}=  \frac{3r^*}{4}$, we obtain $\hbox{dist}(  \mathcal{K}, \mathcal{B} ) \geq \frac{r^*}{4}$.

In order to control the solutions $f$, $\widetilde f$ to \eqref{eq:LBD41.5}, \eqref{eq:LBD41.6} respectively, we  estimate the function $\Theta(y,z;l,m)$ given in \eqref{eq:LBD40.1} for $y,z \in \mathcal{K}$ and for every $l,m \in \RR^3$, $|l|=|m|=1$. 

By \eqref{ass4:U} and \eqref{eq:LBD3.2}, we have
\begin{equation}
	\label{eq:LBD48.2}
	\begin{aligned}
		\| D \widetilde{\mathcal U}\|_{\infty} \leq \dfrac{C}{r_0},
	\end{aligned}
\end{equation}
where $C>0$ only depends on $M_0$, $M_1$, $\theta_0$.

By \eqref{eq:LBD48.2}, by H\"{o}lder's inequality and by \eqref{eq:PG4B.5} with $r=\frac{r^*}{4}$, we obtain
\begin{equation}
	\label{eq:LBD49.1}
	\begin{aligned}
		|I_1^\Theta| \leq \dfrac{C}{r_0^2},
	\end{aligned}
\end{equation}
where the constant $C>0$ depends on the a priori data.

In order to estimate $I_2^\Theta$, let us introduce the following neighborhoods of $\partial \mathcal B$:
\begin{equation*}
	\begin{aligned}
		\mathcal M = \left \{ x \in \RR^3  \ | \ \hbox{dist} (x, \partial \mathcal B   )  \leq \dfrac{r^*}{8}    \right \}\subset\Omega,
		\quad
		\mathcal M' = \left \{ x \in \RR^3  \ | \ \hbox{dist} (x, \partial \mathcal B   )  \leq \dfrac{r^*}{16}    \right \}\subset\Omega.
	\end{aligned}
\end{equation*}
By using Lemma \ref{lem:Li-Nirenberg} and inequality \eqref{eq:PG4B.5}, we have
\begin{eqnarray}
	\label{eq:LBD50.1}
		|I_2^\Theta| &\leq& C r_0^2 \| \nabla G_0^\sharp (\cdot,y) \|_{L^\infty (\mathcal M' )    }
		\| \nabla G_0^\sharp (\cdot,z) \|_{L^\infty (\mathcal M' )    }
		\leq \nonumber
		\\
		&
		\leq&
		C 
		 \| G_0^\sharp (\cdot,y) \|_{L^2 (\mathcal M )    }
		\| G_0^\sharp (\cdot,z) \|_{L^2 (\mathcal M)    }
		\leq 
		\\
		&
		\leq& C \| G_0^\sharp (\cdot,y) \|_{H^1 ( \oms \setminus B_{ \frac{r^*}{8}  }(y)  )    }
		\| G_0^\sharp (\cdot,z) \|_{H^1 ( \oms \setminus B_{ \frac{r^*}{8}  }(z)  )    }
		\leq
		\dfrac{C}{r_0^2}, \nonumber
\end{eqnarray}
where the constant $C>0$ depends on the a priori data.

It is useful to rewrite $I_3^\Theta$ as follows
\begin{equation}
	\label{eq:LBD51.1}
	\begin{aligned}{}
		&
		I_3^\Theta = 
		\int_{\partial D_0 \setminus   \left (  \mathcal B \cup
		B_{r^*}(P) \right )}
	\!\!\!\!\!\!\!\!\!\!\!\!\!\!\!\!\!\!( \widetilde {\mathcal U} \cdot n  )
		\mathbb M \nabla u_0^e \cdot  \nabla v_0^e
		+
		\int_{\partial D_0   \cap
			B_{r^*}(P) }
		\!\!\!\!\!\!\!\!\!\!\!\!\!\!\!\!\!\!( \widetilde {\mathcal U} \cdot n  )
		\mathbb M \nabla u_0^e \cdot  \nabla v_0^e
		=
		I_{FAR} + I_{CL}.
	\end{aligned}
\end{equation}
Let us first estimate $I_{FAR}$. For every $x_0 \in \partial D_0 \setminus   \left (  \mathcal B \cup
B_{r^*}(P) \right )$ and for every $y,z \in \mathcal K$, we have that $|x_0-y| \geq \dfrac{r^*}{2}$ and $|x_0-z| \geq \dfrac{r^*}{2}$. We use again Lemma \ref{lem:Li-Nirenberg} with $r=\dfrac{r^*}{4}$ and $\delta= \dfrac{1}{2}$, and estimate \eqref{eq:PG4B.5}, obtaining
\begin{eqnarray*}
		|\nabla_x u_0^e (x_0,y;l)| &\leq&\| \nabla_x u_0^e (\cdot,y;l)\|_{ L^\infty ( B_{r^*/8}(x_0)  ) }
		\leq  \dfrac{C}{r_0} \| u_0^e (\cdot,y;l)\|_{ L^2 ( B_{r^*/4}(x_0)  ) }
		\leq
		\\
		&
		\leq& 
		\dfrac{C}{r_0}
		\| G_0^\sharp (\cdot,y)\|_{ H^1 (\oms \setminus B_{r^*/4}(y)  ) }
		\leq
		\dfrac{C}{r_0^2},
\end{eqnarray*}
where $C>0$ depends on the a priori data. Therefore, proceeding similarly with $v_0$, we obtain
\begin{equation}
	\label{eq:LBD52.2}
	\begin{aligned}{}
		&
		|I_{FAR}| 
		\leq
		\dfrac{C}{r_0^2},
	\end{aligned}
\end{equation}
where $C>0$ depends on the a priori data. 

Let us consider now $I_{CL}$. By H\"{o}lder inequality, we have
\begin{equation}
	\label{eq:LBD53.1}
	\begin{aligned}{}
		&
		| I_{CL}| \leq 
		Cr_0^2 
		\| \nabla u_0^e \|_{L^2 ( \partial D_0 \cap  B_{r^*}(P)    )   }
		\| \nabla v_0^e \|_{L^2 ( \partial D_0 \cap  B_{r^*}(P)    )   },
	\end{aligned}
\end{equation}
with $C>0$ depending on the a priori data.

Let us handle the term with $u_0^e$, since the other term can be estimated similarly. We define 
\begin{equation*}
	\begin{aligned}{}
		&
		\overline{d}_y= \hbox{dist} (y, \partial D_0).
	\end{aligned}
\end{equation*}
If  $\overline{d}_y \geq \dfrac{r^*}{2}$ we can simply apply the Li-Nirenberg estimate  and \eqref{eq:PG4B.5}, obtaining \[ \| \nabla u_0^e \|_{L^2 ( \partial D_0 \cap  B_{r^*}(P)    )   } \leq \dfrac{C}{r_0^2},\] with $C>0$ depending on the a priori data.

Let us now restrict to the case $\overline{d}_y \leq \dfrac{r^*}{2}$ for which the proof is more involved and requires an approximation of $u_0^e$ in terms of the singular solution $\Gamma^0$ defined by \eqref{eq:PG4B.1}.
Let us write
\begin{equation}
	\label{eq:LBD53.5}
	\begin{aligned}{}
		&
		\| \nabla u_0^e \|_{L^2 ( \partial D_0 \cap  B_{r^*}(P)    )   }
		\leq
		\| \nabla  w      \|_{L^2 ( \partial D_0 \cap  B_{r^*}(P)    )   }
		+
		\| \nabla  \Gamma^0(\cdot,y)  \|_{L^2 ( \partial D_0 \cap  B_{r^*}(P)    )   },
	\end{aligned}
\end{equation}
where $w(\cdot,y) =G_0^\sharp (\cdot, y) - \Gamma^0(\cdot,y)$.

Let us introduce the tubular neighborhood $U_a=\{ x \in  \RR^3 \ | \  \hbox{dist}(x,  \partial D_0 \cap  B_{r^*}(P)   ) < a \}$, with $a > 0$. By applying Lemma \ref{lem:Li-Nirenberg} and Proposition \ref{prop:PGDzero} we obtain
\begin{equation*}
	\begin{aligned}{}
		\| \nabla w(\cdot,y) \|_{L^2 ( \partial D_0 \cap  B_{r^*}(P)    )   }
		&\leq
		C \| \nabla w(\cdot,y)       \|_{L^\infty ( U_{r^*/8}   )   }
		\leq
		\dfrac{C}{r_0} \| w(\cdot,y)       \|_{L^2 ( U_{r^*/4}   )   }
        \leq \\
		&\leq\dfrac{C}{r_0} \| w(\cdot,y)       \|_{H^1 (\oms  )   }
		\leq
		\dfrac{C}{r_0^2},		
	\end{aligned}
\end{equation*}
where the constant $C>0$ depends on the a priori data.

Let $y^\perp$ be the orthogonal projection of $y$ on the plane containing the face $F_j$ of $\partial D_0$, and let us choose a Cartesian coordinate system with origin at $y^\perp$ and with axis ${X}_3$ orthogonal to $F_j$. For brevity, let us denote $B'_{ \overline{d}_y   }( y^\perp )= B_{ \overline{d}_y   } ( y^\perp ) \cap \partial D_0$.

Let us split the second term on the right hand side of \eqref{eq:LBD53.5} as follows
\begin{equation*}
	\begin{aligned}{}
		&
		\int_{  \partial D_0 \cap  B_{r^*}(P) }
		| \nabla_x \Gamma^0(x,y)|^2
		\leq
		 \int_{  B'_{ \overline{d}_y  }  ( y^\perp )}
		 | \nabla_x \Gamma^0(x,y)|^2
		+
		 \int_{  \RR^2 \setminus B'_{ \overline{d}_y  }( y^\perp )  }
		| \nabla_x \Gamma^0(x,y)|^2
		=
		I_a + I_b.
	\end{aligned}
\end{equation*}
By using \eqref{eq:stima-gradiente-Kelvin}  and \eqref{eq:stima-gradiente-Rongved} and noticing that, for every $x \in \RR^2$, $|x-y| \geq \overline{d}_y$, we have
\begin{equation*}
	\begin{aligned}{}
		&
		I_a \leq \dfrac{C}{\overline{d}_y^2},
	\end{aligned}
\end{equation*}
where the constant $C>0$ only depends on $\lambda^i$, $\mu^i$, $\lambda^e$, $\mu^e$.

For every $x \in \{x_3=0\}$, $x=(x',0)$, we have $|x-y|^2=|(x',0)-(0,y_3)|^2=x_1^2+x_2^2+\overline{d}_y^2$. Therefore, by \eqref{eq:stima-gradiente-Rongved} we obtain
\begin{equation*}
	\begin{aligned}{}
		&
		I_b \leq C \int_{ \overline{d}_y  }^{+\infty} \left ( \dfrac{1}{\rho^2 +  \overline{d}_y^2}  \right )^2 \rho d\rho 
		\leq  \dfrac{C}{\overline{d}_y^2},
	\end{aligned}
\end{equation*}
where the constant $C>0$ only depends on $\lambda^i$, $\mu^i$, $\lambda^e$, $\mu^e$.

By recalling \eqref{eq:LBD53.1}, $\overline{d}_y \leq \dfrac{r^*}{2}$, and estimating similarly $\| \nabla v_0^e \|_{L^2 ( \partial D_0 \cap  B_{r^*}(P)    )   }$, we have
\begin{equation}
	\label{eq:LBD56.4}
	\begin{aligned}{}
		&
		|I_{CL}|  \leq r_0^2 \left ( \dfrac{C}{r_0^2} + \dfrac{C}{r_0  \overline{d}_y }    \right )
		\left ( \dfrac{C}{r_0^2} + \dfrac{C}{r_0  \overline{d}_z }    \right )
		\leq
		\dfrac{C}{\overline{d}_y \overline{d}_z},
	\end{aligned}
\end{equation}
where the constant $C>0$ depends on the a priori data.

By \eqref{eq:LBD49.1}, \eqref{eq:LBD50.1}, \eqref{eq:LBD51.1}, \eqref{eq:LBD52.2}, \eqref{eq:LBD56.4}, for every $y,z \in \mathcal{K}$ and for every $l,m \in \RR^3$, $|l|=|m|=1$, we have
\begin{equation}
	\label{eq:LBD57.1}
	\begin{aligned}{}
		&
		|\Theta(y,z;l,m)| \leq
		\dfrac{C}{\overline{d}_y \overline{d}_z},
	\end{aligned}
\end{equation}
where the constant $C>0$ depends on the a priori data.

By the definitions \eqref{eq:LBD41.3} and \eqref{eq:LBD41.4}, and by the linearity of $\Theta$ with respect to $l$ and $m$, we have that, for every 
$l, m \in \RR^3, |l|=|m|=1$, 
\begin{equation}
	\label{eq:LBD63.1}
	\begin{aligned}{}
		&
		| \Theta (y,\overline{z};l,m)|
		\leq 
		|f(y, \overline{z};m)|\leq \sqrt{3} 
		\max_
		{\overset{\scriptstyle l \in \RR^3}{\scriptstyle
				|l|=1}} | \Theta (y,\overline{z};l,m)|,
			\ \ \forall y, \overline{z} \in \oms \setminus ( D_0 \cup \mathcal B  ), 
			\\
		&
		| \Theta (\overline{y},z;l,m)|
		\leq 
		|\widetilde{f}(\overline{y},z;l)|\leq \sqrt{3} 
		\max_
		{\overset{\scriptstyle m \in \RR^3}{\scriptstyle
				|m|=1}} | {\Theta} (\overline{y},z;l,m)|, 
			\ \ \forall\overline{y},z \in \oms \setminus ( D_0 \cup \mathcal B  ).
	\end{aligned}
\end{equation}
We are interested in estimating the quantity $|\Theta(y_r,y_r;l,m)|$ for  
\begin{equation}
	\label{eq:def_y_r}
	\begin{aligned}{}
		&
		y_r = P + r  n_P, \quad 0<r< \dfrac{r^*}{2},
	\end{aligned}
\end{equation}
and for every $l,m \in \RR^3$, $|l|=|m|=1$.

Let $\gamma(t)$ be a parametrization of the curve $\widetilde{\mathfrak{c}}'$ obtained gluing $\widetilde{\mathfrak{c}}$ with the segment $S$, with starting point $P_0$ and ending point $P$.

Proceeding as in \cite[Theorem 5.1]{ARRV}, we iterate the three spheres inequality \eqref{eq:LBD60.2} over a chain of pairwise disjoint balls centered at points $\{ x_i \}_{i=1}^L$ belonging to $\widetilde{\mathfrak{c}}'$ and defined as follows: $x_1 = \overline z \in B_{ r/8}(P_0)$, $x_{i+1}=\gamma(t_i)$, with $t_i = \max \{ t \ | \ |\gamma(t)-x_i|=\frac{r}{4}   \}$ if $|x_i-y_r| > \frac{r}{4}$; otherwise, let us set $i=L$ and stop the process. Since the balls $B_{r/8}(x_i)$, $i=1,\dots, L$, are pairwise disjoint by construction, we have
\begin{equation}
	\label{eq:numero-palle}
	\begin{aligned}{}
		&
		L \leq C_L \left ( \dfrac{r_0}{r}  \right )^3,
	\end{aligned}
\end{equation}
where $C_L>0$ only depends on $M_1$. We choose in \eqref{eq:LBD60.2} the radii $s_1 = \frac{r}{8}$, $s_2 = \frac{3r}{8}$ and $s_3 = \frac{r}{2}$. Let us notice that $\cup_{i=1}^L B_{\frac{r}{2}}(x_i) \subset \left \{  x \in \mathcal K \ | \ \hbox{dist}(x, \partial D_0) \geq \frac{r}{2}   \right \}$, so that, by \eqref{eq:LBD57.1},
\begin{equation}
	\label{eq:END3-1.1}
	\begin{aligned}{}
		&
		\|  \Theta ( \cdot, \overline z; l,m) \|_{ L^\infty \left (  \cup_{i=1}^L B_{r/2}(x_i)   \right )   }
		\leq \dfrac{C}{r r_0},
	\end{aligned}
\end{equation}
where $C>0$ depends on the a priori data. By applying \eqref{eq:LBD60.2} and using \eqref{eq:END3-1.1}, at the $i$th step we obtain
\begin{equation*}
	\begin{aligned}{}
		&
		rr_0 \| f(\cdot, \overline z;m)\|_{ L^\infty (B_{s_1} (x_{i+1}))} \leq C 
		\left (
		rr_0 \| f(\cdot, \overline z;m)\|_{ L^\infty (B_{s_1}(x_{i}))}		
		\right )^\tau,
	\end{aligned}
\end{equation*}
where $C>0$ and $\tau$, $0<\tau <1$, depend on the a priori data. Using estimate \eqref{eq:LBD7.4} and iterating, we have 
\begin{equation*}
	\begin{aligned}{}
		&
		rr_0 \| f(\cdot, \overline z;m)\|_{ L^\infty (B_{s_1} (y_r)  )} \leq 
		C \left (  \dfrac{r}{r_0}m_1  \right )^{\tau^L}, 
		\quad \hbox{for every } \overline z \in B_{s_1}(P_0),
	\end{aligned}
\end{equation*}
where $C>0$  depends on the a priori data.

By using \eqref{eq:LBD63.1} in the above inequality, we obtain
\begin{equation}
	\label{eq:END3-2.2}
	\begin{aligned}{}
		&
		\| \widetilde f (y_r, \cdot; l) \|_{ L^\infty (B_{s_1} (P_0)  )}
		\leq
		\dfrac{C}{r r_0}
		 \left (  \dfrac{r}{r_0}m_1  \right )^{\tau^L}.
	\end{aligned}
\end{equation}
Next, by \eqref{eq:LBD57.1} we have
\begin{equation}
	\label{eq:END3-2.3}
	\begin{aligned}{}
		&
		\|  \widetilde f  ( y_r, \cdot ; l) \|_{ L^\infty \left (  \cup_{i=1}^L B_{\frac{r}{2}}(x_i)   \right )   }
		\leq \dfrac{C}{r^2},
	\end{aligned}
\end{equation}
where $C>0$ depends on the a priori data.

By repeating the above process for the function $\widetilde f$, operating on the $z$-variable instead of the $y$-variable and using \eqref{eq:END3-2.2}, \eqref{eq:END3-2.3}, we obtain 
\begin{equation*}
	\begin{aligned}{}
		&
		r^2 \| \widetilde f(y_r,\cdot;l)\|_{ L^\infty (B_{s_1} (y_r)  )} \leq 
		C \left (
		\dfrac{r}{r_0}
		\left (
		  \dfrac{r}{r_0}m_1  \right )^{\tau^L}
		  \right )^{\tau^L},
	\end{aligned}
\end{equation*}
and, therefore by \eqref{eq:LBD63.1}, for every $l, \ m \in \RR^3$, $|l|=|m|=1$, 
\begin{equation}
	\label{eq:END3-2.5}
	\begin{aligned}{}
		&
		|\Theta ( y_r, y_r;l, m)| \leq \dfrac{C}{r^2} m_1 ^{\tau^{2L}}\left (  \dfrac{r}{r_0} \right )^{2\tau^{2L}},
		\quad
		\hbox{for every } r < \dfrac{r^*}{2},
	\end{aligned}
\end{equation}
where $C>0$ depends on the a priori data.

\medskip

\textit{Step 4.}

In this step we estimate {}from below the quantity $|\Theta ( y_r, y_r;l, m)|$ for $l=m=n$, where $n$ is the unit outer normal to $\partial D_0$ at $P$, and $r < \dfrac{\sin  ( \theta_0/2  ) }{2(1+ \sin  ( \theta_0/2  ))} r^*$. This choice ensures that $r < \overline{r}(y_r)$, where $\overline{r}(y_r)$ has been defined in \eqref{eq:PG4A.2} and, therefore, the singular solution $\Gamma^0$ coincides with the Rongved fundamental solution $\Gamma^R$ defined in \eqref{eq:pbm-Rongved} and  having interface coincident with the plane through $F_j$.

{}From \eqref{eq:LBD40.1} we have
\begin{eqnarray*}
		\Theta(y_r,y_r;n,n)\!\! &=& \!\!
		\int_{\mathcal B} 
		(
		\dive \,
		\widetilde{\mathcal U}
		)
		\CC^{D_{0}} \nabla u_{0}
		\cdot
		\nabla v_{0}
		-
		\CC^{D_{0}} \nabla u_{0}
		\cdot 
		(
		\nabla v_0 D\widetilde{\mathcal U}
		)
		-
		\CC^{D_{0}} \nabla v_{0}
		\cdot 
		(
		\nabla u_0 D\widetilde{\mathcal U}
		)
		-
		\\
		&		
		-&
		\left (
		\int_{\partial \mathcal B \cap ( \Omega \setminus D_0  )}
		b^e \cdot n + 
		\int_{\partial \mathcal B \cap D_0 }
		b^i \cdot n
		\right )
		+
		\int_{\partial D_0 \setminus \left (  \mathcal B \cup B_{r^*}(P) \right )      }
		\!\!\!\!\!\!\!\!\!\!( \widetilde {\mathcal U} \cdot n  )
		\mathbb M  \widehat{\nabla} u_0^e \cdot   \widehat{\nabla} v_0^e+
		\\&+&
		\int_{\partial D_0 \cap B_{r^*}(P)   }
		\!\!\!\!\!\!\!\!\!\!( \widetilde {\mathcal U} \cdot n  )
		\mathbb M \widehat{\nabla} u_0^e \cdot   \widehat{\nabla} v_0^e
		= I_1^\Theta -I_2^\Theta + I_{FAR} + I_{CL},
\end{eqnarray*}
where, by definition \eqref{eq:LBD5.2}, 
\begin{equation*}
	\begin{aligned}{}
		&
		u_0=u_0(x,y_r;n)=  G_0^\sharp (x,y_r) n,
		\quad
		v_0=v_0(x,y_r;n)=  G_0^\sharp (x,y_r) n,
	\end{aligned}
\end{equation*}
and $b^e$, $b^i$ are defined by \eqref{eq:LBD10.2}, \eqref{eq:LBD12.3}.

We will obtain the desired bound by estimating {}from below $I_{CL}$ and {}from above the other terms:
\begin{equation}
	\label{eq:LBD68.3}
	\begin{aligned}{}
		&
		|\Theta(y_r,y_r;n,n)| \geq |I_{CL}| - |I_1^\Theta| -|I_2^\Theta| - |I_{FAR}|.
	\end{aligned}
\end{equation}
By \eqref{eq:LBD49.1},  \eqref{eq:LBD50.1}, \eqref{eq:LBD52.2} we have
\begin{equation}
	\label{eq:LBD69.1}
	\begin{aligned}{}
		&
		|I_1^\Theta| + |I_2^\Theta| + |I_{FAR}| \leq \dfrac{C}{r_0^2},
	\end{aligned}
\end{equation}
where $C>0$ depends on the a priori data.

In estimating $I_{CL}$ we approximate $G_0^\sharp (\cdot, y_r)$ with the Rongved fundamental solution $\Gamma^R (\cdot, y_r)$ and we write $(\widetilde {\mathcal U} \cdot n )(x)= (\widetilde {\mathcal U} \cdot n )(P) + (  (\widetilde {\mathcal U} \cdot n )(x) - (\widetilde {\mathcal U} \cdot n )(P)   )$, obtaining 
\begin{equation*}
	\begin{aligned}{}
		&
		I_{CL} = 
		\int_{\partial D_0 \cap B_{r^*}(P)   }
		( \widetilde {\mathcal U} \cdot n  )(P)
		\mathbb M  \widehat{\nabla} (\Gamma^R(x,y_r) n) \cdot \widehat{\nabla} (\Gamma^R(x,y_r) n)
		+
		\\
		&
		+
		\int_{\partial D_0 \cap B_{r^*}(P)   }
		 (  (\widetilde {\mathcal U} \cdot n )(x) - (\widetilde {\mathcal U} \cdot n )(P)   )
		\mathbb M  \widehat{\nabla} (\Gamma^R(x,y_r) n) \cdot \widehat{\nabla} (\Gamma^R(x,y_r) n)+
		\\
		&
		+
		\int_{\partial D_0 \cap B_{r^*}(P)   }
		(\widetilde {\mathcal U} \cdot n )(x) 
		\mathbb M  \widehat{\nabla} (( G_0^\sharp (x,y_r)- \Gamma^R(x,y_r)) n) \cdot \widehat{\nabla}(( G_0^\sharp (x,y_r)- \Gamma^R(x,y_r)) n) +
		\\
		&
		+
		2
		\int_{\partial D_0 \cap B_{r^*}(P)   }
		(\widetilde {\mathcal U} \cdot n )(P) 
		\mathbb M  \widehat{\nabla} (\Gamma^R(x,y_r)n )\cdot \widehat{\nabla}(( G_0^\sharp (x,y_r)- \Gamma^R(x,y_r)) n) +
		\\
		&
		+
		2
		\int_{\partial D_0 \cap B_{r^*}(P)   }
		 (  (\widetilde {\mathcal U} \cdot n )(x) - (\widetilde {\mathcal U} \cdot n )(P)   ) 
		\mathbb M  \widehat{\nabla} (\Gamma^R(x,y_r)n) \cdot \widehat{\nabla}(( G_0^\sharp (x,y_r)- \Gamma^R(x,y_r)) n)
		=
		\\
		&
		=
		J_1 + J_2 + J_3 + J_4 + J_5.
		\end{aligned}
\end{equation*}
We estimate $I_{CL}$ {}from below as follows:
\begin{equation}
	\label{eq:LBD71.2}
	\begin{aligned}{}
		&
		|I_{CL}| \geq |J_1| - |J_2| - |J_3| - |J_4| - |J_5|.
	\end{aligned}
\end{equation}
By \eqref{ass4:U} and \eqref{eq:LBD3.2} we have
\begin{equation}
	\label{eq:LBD71.3}
	\begin{aligned}{}
		&
		\left| (\widetilde {\mathcal U} \cdot n )(x) - (\widetilde {\mathcal U} \cdot n )(P) \right| \leq \dfrac{C}{r_0}|x-P|,
	\end{aligned}
\end{equation}
where $C>0$ only depends on $M_0$, $M_1$ and $\theta_0$.

Let us choose a Cartesian coordinate system with origin at $P$ and axis ${X}_3$ oriented as $n$.  By \eqref{eq:LBD71.3} and \eqref{eq:stima-gradiente-Rongved}
\begin{equation}
	\label{eq:LBD72.1}
	\begin{aligned}{}
		&
		|J_2| \leq \dfrac{C}{r_0} \int_{  \sqrt{  x_1^2+x_2^2} < r^*}
		\dfrac{ \sqrt{  x_1^2+x_2^2}}{( x_1^2+x_2^2+r^2)^2}
		\leq
		\dfrac{C}{r r_0},
	\end{aligned}
\end{equation}
where $C>0$  depends on the a priori data.

By \eqref{ass4:U}, \eqref{eq:LBD3.2} and  \eqref{eq:PG4B.4}, we have
\begin{equation}
	\label{eq:LBD74.2}
	\begin{aligned}{}
		&
		|J_3| \leq \dfrac{C}{r_0^2},
	\end{aligned}
\end{equation}
where $C>0$ only depends on the a priori data.

Using the standard inequality $ab \leq \frac{a^2\epsilon^2}{2}+ \frac{b^2}{2\epsilon^2}$ for every $\epsilon >0$, we have
\begin{equation}
	\label{eq:LBD75.1}
	\begin{aligned}{}
		&
		|J_4| \leq \\
		&\leq C \left |(\widetilde {\mathcal U} \cdot n )(P) \right |
		\left (
		\epsilon^2 \int_{\partial D_0 \cap B_{r^*}(P)   }
		|\widehat{\nabla} (\Gamma^R(x,y_r)n)|^2 +
		\frac{1}{\epsilon^2} \int_{\partial D_0 \cap B_{r^*}(P)   } | \widehat{\nabla}(( G_0^\sharp (x,y_r)- \Gamma^R(x,y_r))n) |^2 
		\right )
		\leq
		\\
		&
		\leq
		C_0\left |(\widetilde {\mathcal U} \cdot n )(P) \right |
		\left (
		\epsilon^2 \int_{\partial D_0 \cap B_{r^*}(P)   }
		|\widehat{\nabla} (\Gamma^R(x,y_r)n)|^2 +
		\frac{1}{\epsilon^2 r_0^2}  
		\right ),
	\end{aligned}
\end{equation}
where $C>0$ depends on the a priori data.

By H\"{o}lder's inequality and \eqref{eq:LBD71.3}, we have 
\begin{equation}
	\label{eq:LBD75A.1}
	\begin{aligned}{}
		&
		|J_5| \leq  \dfrac{C}{r_0} J_5' J_5'',
	\end{aligned}
\end{equation}
where
\begin{equation}
	\label{eq:LBD75A.3}
	\begin{aligned}{}
		&
		J_5' = 
		\left (
		\int_{\partial D_0 \cap B_{r^*}(P)   }
		|x-P|^2
		| {\nabla} \Gamma^R(x,y_r)|^2
		\right )^{1/2}
		\leq C \left (  \left |  \log \dfrac{r}{r_0}   \right |  \right )^{1/2},
	\end{aligned}
\end{equation}
\begin{equation}
	\label{eq:LBD75A.2}
	\begin{aligned}{}
		&
		J_5'' = 
		\left (
		\int_{\partial D_0 \cap B_{r^*}(P)   }	
		\left | {\nabla} (G_0^\sharp(x,y_r) - \Gamma^R(x,y_r)) \right |^2
		\right )^{1/2}
		\leq \dfrac{C}{r_0},
	\end{aligned}
\end{equation}
so that 
\begin{equation}
	\label{eq:LBD75B.1}
	\begin{aligned}{}
		&
		|J_5| \leq  \dfrac{C}{r_0^2} 
		\left (  \left |  \log \dfrac{r}{r_0}   \right |  \right )^{1/2},
	\end{aligned}
\end{equation}
where $C>0$ depends on the a priori data.

By \eqref{eq:LBD-M-convex+}, \eqref{eq:LBD-M-convex-} and choosing as above a Cartesian coordinate system with origin at $P$ and axis ${X}_3$ oriented as $n$,  we have
\begin{equation}
	\label{eq:LBD76.1}
	\begin{aligned}{}
		&
		|J_1| \geq  \sigma
		\left |( \widetilde {\mathcal U} \cdot n  )(P)\right | 
		\int_{\partial D_0 \cap B_{r^*}(P)   }
		\left | \widehat{\nabla} (\Gamma^R(x,y_r)n) \right |^2.
	\end{aligned}
\end{equation}
By \eqref{eq:LBD71.2}, \eqref{eq:LBD72.1}-\eqref{eq:LBD75.1}, \eqref{eq:LBD75B.1}, \eqref{eq:LBD76.1}, we have
\begin{equation}
	\label{eq:LBD86.0}
	\begin{aligned}{}
		&
		|I_{CL}| \geq (\sigma-C_0\epsilon^2)|(\widetilde {\mathcal U} \cdot n )(P)|
		\int_{\partial D_0 \cap B_{r^*}(P)   }
		|\widehat{\nabla} (\Gamma^R(x,y_r)n)|^2
		-
		\\
		&
		-\frac{C}{rr_0} -\frac{C}{r_0^2} -\frac{C}{r_0^2}\left|\log\left(\frac{r}{r_0}\right)\right|^{1/2}
		-\frac{C_0}{\epsilon^2r_0^2}|(\widetilde {\mathcal U} \cdot n )(P)|,
	\end{aligned}
\end{equation}
where $C,\ C_0 >0$  depend on the a priori data.

Since $|\log x|^{1/2}\leq \dfrac{1}{x\sqrt{2e}}$, for $0<x \leq 1$, noticing that, by \eqref{ass4:U} and \eqref{eq:LBD3.2},
$\|\widetilde U\|_{\infty}\leq C$, with $C>0$ only depending on $M_0$, $M_1$ and $\theta_0$,
choosing $\epsilon=\frac{1}{\sqrt 2}\sqrt{\frac{\sigma}{C_0}}$ and using $r<r_0$, we have
\begin{equation}
	\label{eq:LBD86.0bis}
	\begin{aligned}{}
		&
		|I_{CL}| \geq \frac{\sigma}{2}|(\widetilde {\mathcal U} \cdot n )(P)|
		\int_{\partial D_0 \cap B_{r^*}(P)   }
		|\widehat{\nabla} (\Gamma^R(x,y_r)n)|^2
		-\frac{C}{rr_0}
		\geq
		\\
		&
		\geq
		\frac{\sigma}{2}
		\left |( \widetilde {\mathcal U} \cdot n  )(P)\right | 
		\int_{\partial D_0 \cap B_{r^*}(P)   }
		\left (
		\dfrac{\partial \Gamma^R_{33}(x;y_r)}{\partial x_3}
		\right )^2-\frac{C}{rr_0},
	\end{aligned}
\end{equation}
with $C>0$ depending on the a priori data.

An estimate {}from below for $\frac{\partial \Gamma^R_{33}(x;y_r)}{\partial x_3}$ is provided in the following lemma.

\begin{lemma}
	\label{lem:LBD77.1}
	For every $r>0$ and $y_r=(0,0,r)$, we have
\begin{equation}
	\label{eq:LBD77.2}
	\begin{aligned}{}
		&
		\dfrac{\partial \Gamma^R_{33}(x;y_r)}{\partial x_3}
		\geq \dfrac{C}{r^2}, \qquad \hbox{for every } x=(x_1,x_2,0) \ \hbox{such that } \ \sqrt{x_1^2+x_2^2} \leq \dfrac{r}{\sqrt{2}},
	\end{aligned}
\end{equation}
where the constant $C>0$ only depends on $\mu^e$, $\mu^i$, $\lambda^e$.
\end{lemma}
\begin{proof}
For simplicity of notation, let us denote $\mu=\mu^e$, $\lambda=\lambda^e$, $\nu=\nu^e$, $\mu'=\mu^i$, $\lambda'=\lambda^i$, $\nu'=\nu^i$. By using the explicit expression derived in \cite{rongved1955force}, we have 
\begin{equation}
	\label{eq:LBD78.3}
	\begin{aligned}{}
		&
		16(1-\nu) \pi \mu \dfrac{\partial \Gamma^R_{33}(x;y_r)}{\partial x_3} =
		\\
		&
		=
		(3-4\nu)
		\left \{
		- \dfrac{x_3-r}{R_1^3}
		+
		\dfrac{\mu-\mu'}{\mu+ \mu'(3-4\nu)}
		\left [
		(3-4\nu) \left ( - \dfrac{x_3+r}{R_2^3}  \right )
		+
		2r(x_3+r) \left ( - \dfrac{3(x_3+r)}{R_2^5}  \right )
		+
		\dfrac{2r}{R_2^3}
		\right ]
		\right \}
		+
		\\
		&
		+
		\dfrac{2(x_3-r)}{R_1^3} + (x_3-r)^2 \left ( - \dfrac{3(x_3-r)}{R_1^5}  \right )
		+
		\\
		&
		+
		\dfrac{\mu-\mu'}{\mu+ \mu'(3-4\nu)}
		\left \{
		(3-4\nu) \dfrac{2x_3}{R_2^3}
		+
		(3-4\nu)(x_3^2 -r^2) \left ( - \dfrac{3(x_3+r)}{R_2^5}  \right ) -
		\right.
		\\
		&
		-
		\dfrac{4(1-\nu)\mu}{\mu-\mu'}	
		\left [
		\dfrac{\mu(1-2\nu)(3-4\nu')- \mu'(1-2\nu')(3-4\nu) }{\mu' + \mu (3-4\nu')}
		\right ]	
		\left ( - \dfrac{x_3+r}{R_2^3}  \right )-
		\\
		&
		- \dfrac{2r}{R_2^5} \left(  x_1^2 +x_2^2 -2(x_3 +r)^2 \right)-
		\\
		&
		\left.
		-2rx_3 \left ( x_1^2 +x_2^2 -2(x_3+r)^2   \right )
		\left (
		- \dfrac{5(x_3+r)}{R_2^7}
		\right )
		-
		\dfrac{2rx_3}{R_2^5}
		(-4(x_3+r))
		\right \},
	\end{aligned}
\end{equation}
where
\begin{equation*}
	\begin{aligned}{}
		&
		R_1= \left (x_1^2 +x_2^2 +(x_3-r)^2  \right )^{1/2},
		\quad
		R_2= \left (x_1^2 +x_2^2 +(x_3+r)^2  \right )^{1/2}.
	\end{aligned}
\end{equation*}
Let us evaluate the above expression on the circle $x_1^2+x_2^2 = \alpha^2 r^2$ contained in the plane $x_3=0$. Therefore, we have
\begin{equation*}
	\begin{aligned}{}
		&
		R_1= R_2 = r\sqrt{\alpha^2 +1}, \quad \dfrac{r}{R_1}=\dfrac{r}{R_2}= \gamma := \dfrac{1}{\sqrt{\alpha^2 +1}}.
	\end{aligned}
\end{equation*}
We found it convenient to rewrite the right hand side of \eqref{eq:LBD78.3} as a quadratic form in the \textit{positive} Lam\'e parameters $\mu$, $\mu'$. After straightforward computation, we obtain, for $x_3=0$, $x_1^2+x_2^2=\alpha^2r^2$,
\begin{equation}
	\label{eq:LBD80.1}
	\begin{aligned}{}
		&
		16r^2(1-\nu) \pi \mu \dfrac{\partial \Gamma^R_{33}(x;y_r)}{\partial x_3} 
		=
		A \mu^2 + B (\mu')^2 + C \mu \mu',
	\end{aligned}
\end{equation}
where
\begin{equation}
	\label{eq:LBD81.1}
	\begin{aligned}{}
		&
		A = 4  \nu \gamma^3  (3-4\nu')  (3-2\nu),
	\end{aligned}
\end{equation}
\begin{equation}
	\label{eq:LBD82.1}
	\begin{aligned}{}
		&
		B = 4\gamma^3 (  2- 8\nu +3\gamma^2 -6\nu \gamma^2 +8\nu^2  ),
	\end{aligned}
\end{equation}
\begin{equation}
	\label{eq:LBD83.1}
	\begin{aligned}{}
		&
		C = (3\gamma^5 +(1-4\nu)\gamma^3) ( 1+(3-4\nu)(3-4\nu')   )
		+
		\\
		&
		+
		(4\nu'-2) 
		\left \{
		(3-4\nu) 
		\left [ 
		(4\nu-1)\gamma^3 -6\gamma^5
		\right ]
		+
		15 \gamma^5
		-12 \nu \gamma^5 -2\gamma^3		
		\right \}-
		\\
		&
		-4\gamma^3 (1-\nu)(   1 -2\nu') (3-4\nu).
	\end{aligned}
\end{equation}
Under our assumptions on the Lam\'e moduli, in particular \eqref{Notaz-Poisson-bounds},  we have
\begin{equation}
	\label{eq:pos-A}
	\begin{aligned}{}
		&
		A \geq 0.
	\end{aligned}
\end{equation}
Let us denote by $g_\gamma(\nu,\nu')$ the expression on the right hand side of \eqref{eq:LBD83.1}. 

In order to estimate $g_\gamma$ for $\nu, \nu'\in \left[0,\frac{1}{2}\right)$, let us notice that it is linear w.r.t. the variable $\nu'$, so that it attains its minimum value either on the side $l_1=\left[0,\frac{1}{2} \right]\times\{0\}$ or on the side $l_2=\left[0,\frac{1}{2} \right]\times\{\frac{1}{2}\}$ of the square $Q=\left[0,\frac{1}{2} \right]\times\left[0,\frac{1}{2} \right]$.

On $l_1$, $g_\gamma(\nu,0)=2\gamma^3[18\gamma^2-30\nu\gamma^2-28\nu+32\nu^2+4]$. The function 
$G_\gamma(\nu) = 18\gamma^2-30\nu\gamma^2-28\nu+32\nu^2+4$ satisfies $G_\gamma (0)=18\gamma^2+4>0$, $G_\gamma(\frac{1}{2})=3\gamma^2-2$ and it has no critical points in the interval $\left[0,\frac{1}{2} \right]$. Therefore, for $\gamma^2\geq \frac{2}{3}$, that is $\alpha\leq \dfrac{1}{\sqrt 2}$, 
$g_\gamma\geq 0$ on $l_1$.

On $l_2$, $g_\gamma\left(\nu,\frac{1}{2}\right)=4\gamma^3(1-\nu)(3\gamma^2+1-4\nu)\geq
4\gamma^3(1-\nu)(3\gamma^2-1)> 0 $ for $\gamma^2\geq \frac{2}{3}$. 

Hence 

\begin{equation}
	\label{eq:pos-C}
	\begin{aligned}{}
		&
		C \geq 0.
	\end{aligned}
\end{equation}

By \eqref{eq:LBD80.1}, \eqref{eq:pos-A}, \eqref{eq:pos-C}

\begin{equation*}
	\begin{aligned}{}
		16r^2(1-\nu) \pi \mu \dfrac{\partial \Gamma^R_{33}(x;y_r)}{\partial x_3} 
	\geq
	B (\mu')^2,
	\end{aligned}
\end{equation*}
with $B$ given by \eqref{eq:LBD82.1}.

Since $\nu<1/2$, the function $f_\nu (\gamma)=  2- 8\nu +3\gamma^2 -6\nu \gamma^2 +8\nu^2 $ is strictly increasing. Therefore, for
 $\gamma^2\geq \frac{2}{3}$, that is for $\alpha\leq \dfrac{1}{\sqrt 2}$, 
$f_\nu (\gamma)\geq f_\nu \left(\sqrt{2/3} \right)=
4(1-\nu)(1-2\nu)$, so that 

\begin{equation*}
	\begin{aligned}{}
		&
		B\geq \dfrac{32\sqrt 2}{3\sqrt 3}(1-\nu)(1-2\nu)>0,
	\end{aligned}
\end{equation*}

\begin{equation*}
	\begin{aligned}{}
		&
		\dfrac{\partial \Gamma^R_{33}(x;y_r)}{\partial x_3}
		\geq
		\dfrac{C}{r^2},
	\end{aligned}
\end{equation*}
where $C>0$ only depends on $\mu'$, $\lambda$, $\mu$.
 
\end{proof}

By \eqref{eq:LBD86.0bis} and \eqref{eq:LBD77.2}, and using $r<r_0$, we have
\begin{equation}
	\label{eq:LBD86.1}
	\begin{aligned}{}
		&
		|I_{CL}| \geq 
		\frac{Cr_0^2}{r^4}|(\widetilde {\mathcal U} \cdot n )(P)|
		-\frac{C}{rr_0}\geq
		\frac{C}{r^2}|(\widetilde {\mathcal U} \cdot n )(P)|
		-\frac{C}{rr_0},
	\end{aligned}
\end{equation}
and,
by \eqref{eq:LBD68.3}, \eqref{eq:LBD69.1} and \eqref{eq:LBD86.1}, we finally get
\begin{equation}
	\label{eq:LBD87.1}
	\begin{aligned}{}
		&
		|\Theta(y_r,y_r;n,n)| \geq \frac{C}{r^2}|(\widetilde {\mathcal U} \cdot n )(P)|
		-\frac{C}{rr_0},
	\end{aligned}
\end{equation}
with $C>0$ depending on a priori data.

By comparing the estimates \eqref{eq:LBD87.1} and \eqref{eq:END3-2.5} of $\Theta$, we obtain
\begin{equation*}
	\begin{aligned}{}
		&
		|(\widetilde {\mathcal U} \cdot n )(P)|\leq 
		C\left(
		 \left (  \dfrac{r}{r_0} \right )^{2\tau^{2L}} m_1 ^{\tau^{2L}}+ \dfrac{r}{r_0}\right)
		 ,
		\quad
		\hbox{for every } r < \dfrac{\sin(\theta_0/2)}{2(1+\sin(\theta_0/2)}r^*,
	\end{aligned}
\end{equation*}
and, by using $r<r_0$ and by recalling \eqref{eq:numero-palle},
\begin{equation}
	\label{eq:END3-1bis}
	\begin{aligned}{}
		&
		|(\widetilde {\mathcal U} \cdot n )(P)|\leq 
		C\left(
		m_1 ^{\tau^{2C_L\left(\frac{r}{r_0}\right)^{-3}}}+ \dfrac{r}{r_0}\right),
		\quad
		\hbox{for every } r < \dfrac{\sin(\theta_0/2)}{2(1+\sin(\theta_0/2)}r^*,
	\end{aligned}
\end{equation}
where $C_L>0$ only depends on $M_1$ and $C>0$ depends on the a priori data.

Let 
\begin{equation*}
	\begin{aligned}{}
		&
		\epsilon_0 = \min \left \{e^{-1}, \exp \left [ -\exp \left ( 
		2C_L \left (  \dfrac{2(1+\sin(\theta_0/2))r_0}{\sin(\theta_0/2)r^*}   \right )^3 | \log \tau|
		\right )   \right ]  \right  \}.
	\end{aligned}
\end{equation*}
If $m_1 \geq \epsilon_0$, then the thesis of Proposition \ref{prop:LBD1.1} follows trivially. 

Let us assume, therefore, that $m_1 < \epsilon_0$. 
Let us choose in \eqref{eq:END3-1bis}
\begin{equation*}
	\begin{aligned}{}
		&
		 r=r(m_1)= r_0 
		 \left (  2 C_L |\log \tau|  \right )^{1/3} 
		 \left (  \log | \log m_1|      \right )^{-1/3}.
	\end{aligned}
\end{equation*}
Noticing that, by the definition of $\epsilon_0$,  $r(m_1) < \dfrac{\sin(\theta_0/2)}{2(1+\sin(\theta_0/2))}r^*$ and $e^{-|\log m_1|^{1/2}} < \left(\log|\log m_1|\right)^{-1/3}$, we obtain
\begin{equation}
	\label{eq:Utilde-per-n}
	\begin{aligned}{}
		&
		|(\widetilde {\mathcal U} \cdot n )(P)|\leq 
		C
		\left (
		e^{-|\log m_1|^{1/2}}
		+
		\left (  2 C_L |\log \tau|  \right )^{1/3} 
		\left (  \log | \log m_1|      \right )^{-1/3}
		\right )
		\leq 
		C \omega_0(m_1),
	\end{aligned}
\end{equation}
where the constant $C>0$ depends on the a priori data and  
\begin{equation}
	\label{eq:omega-0}
	\begin{aligned}{}
		&
		\omega_0 (m_1)=
		\left (  \log | \log m_1|      \right )^{-1/3}.
	\end{aligned}
\end{equation}

\medskip

\textit{Step 5.}

By \eqref{distPbordoT}, the estimate \eqref{eq:Utilde-per-n}-\eqref{eq:omega-0} holds for every point 
of the disc $B'_{\frac{c_0r_0}{2}}(P)$ and, since $\widetilde{\mathcal{U}}$ is affine on $T$, the estimate extends to the whole triangle. This argument applies to the all the triangles $T\in \mathcal{T}_0$ and, therefore, to every vertex of $D_0$. 

By \eqref{eq:LBD3.3} and \eqref{ass1:U}, for every vertex $V^{D_0}$ of $D_0$ and for every face $F_j$ containing $V^{D_0}$, we have
\begin{equation}
	\label{eq:spost-vert-e-omega0}
	\begin{aligned}{}
		& 
		\dfrac{  \left | ( V^{D_1}  - V^{D_0}  ) \cdot n_j  \right |   }{|W|}
		\leq \omega_0(m_1), 
	\end{aligned}
\end{equation}
where $n_j$ is the outer unit normal to $F_j$.

Let $\overline{V}^{D_0}$ the vertex in which the maximum of $\left |  V^{D_1}  - V^{D_0}  \right | $  is attained. It follows trivially that
\begin{equation}
	\label{eq:radice-di-N}
	\begin{aligned}{}
		& 
		\dfrac{  | \overline{V}^{D_1}  - \overline{V}^{D_0}  |   }{|W|}
		\geq \dfrac{1}{\sqrt{N}},
	\end{aligned}
\end{equation}
where the total number of vertices $N$ of both $D_0$ and $D_1$ is bounded by the constant $N_0$, which only depends on $M_0$, $M_1$, $\theta_0$, see Remark \ref{rem1}.

In order to extend \eqref{eq:spost-vert-e-omega0} to the direction $\overline{n} = \dfrac{ \overline{V}^{D_1}  - \overline{V}^{D_0}}{\left |  \overline{V}^{D_1}  - \overline{V}^{D_0} \right |}$, let us premise the following proposition, whose proof is postponed to the end of this section.
\begin{proposition}\label{geometrica}
\label{prop:geometrica}
Given a vertex $V$ of $D_0$ and given adjacent faces $F_1$, $F_2$, $F_3$ containing $V$, for every unit vector $\overline{n} \in \RR^3$, we have
\begin{equation}
	\label{eq:geom-1.1}
	\begin{aligned}{}
		& 
		 \overline{n}= \sum_{i=1}^3 \alpha_i n_i,
	\end{aligned}
\end{equation}
with
\begin{equation}
	\label{eq:geom-1.2}
	\begin{aligned}{}
		& 
		|\alpha_i| \leq \dfrac{2(\sin^2 \theta_0+1)}{\sin^4 \theta_0},
	\end{aligned}
\end{equation}
where $n_i$ is the outer unit normal to $F_i$, $i=1,2,3$.
\end{proposition}
By applying the above lemma to \eqref{eq:spost-vert-e-omega0} with  $\overline{n} = \dfrac{ \overline{V}^{D_1}  - \overline{V}^{D_0}}{\left |  \overline{V}^{D_1}  - \overline{V}^{D_0} \right |}$ and by using \eqref{eq:radice-di-N}, we have
\begin{equation*}
	\begin{aligned}{}
		& 
		\dfrac{1}{\sqrt{N}}
		\leq
		\dfrac{   | \overline{V}^{D_1}  - \overline{V}^{D_0}  |   }{|W|}
		\leq
		\omega_0(m_1),
	\end{aligned}
\end{equation*}
so that 
\begin{equation*}
	\begin{aligned}{}
		& 
		m_1 \geq 
		\omega_0^{-1} \left ( \dfrac{1}{\sqrt{N}}   \right ).
	\end{aligned}
\end{equation*}
By \eqref{eq:LBD3.2}, \eqref{eq:LBD3.6} and \eqref{eq:LBD3.8}, the thesis \eqref{eq:LBD1.1} follows.
\end{proof}

\begin{proof}[Proof of Proposition \ref{geometrica}]

Let us premise the following Lemma

\begin{lemma}	\label{lem:geometrico}
	Let $r$ and $s$ be two lines intersecting in a point $V$ and forming an angle $\theta_1$. Let $\alpha$ be the plane containing $r$ and $s$ and let $\beta$ be a plane through $r$ forming an angle $\theta_2$ with $\alpha$.
	Then the angle between $s$ and $\beta$ is
	\begin{equation}
		\label{eq:GEO1.3}
		\alpha_0=\arcsin(\sin \theta_1\cdot \sin \theta_2).
	\end{equation}

\end{lemma}

\begin{proof}
If $\theta_2$ is a right angle, then $\alpha_0=\theta_1$ and \eqref{eq:GEO1.3} follows.
Otherwise,
let $P\in s$ such that $|PV|=1$ and let $H$ the orthogonal projection of $P$ in the plane $\beta$.
Let $t\subset\beta$ be the line through $H$ orthogonal to $r$, and let $K=r\cap s$. The plane through $P$, $H$ and $K$ is orthogonal to $r$ since it contains two lines ($r$ and $s$) orthogonal to $r$. Therefore the segment $PK$ is orthogonal to $r$ and the angle between $PK$ and $t$ is $\theta_2$.
Since the triangle $PHV$ has a right angle at $H$, we have $|PH|=|PV|\sin\alpha_0=\sin \alpha_0$.
Since the triangle $PHK$ has a right angle at $H$, we have $|PH|=|PK|\sin \theta_2$.
Since the triangle $PKV$ has a right angle at $K$, we have $|PK|=|PV|\sin \theta_1= \sin \theta_1$. It follows that  $\sin\alpha_0=\sin \theta_1 \cdot\sin \theta_2$.

\end{proof}

Let $\pi_i$ the plane containing the face $F_i$, $i=1,2,3$. Let $r=\pi_2\cap\pi_3$, $s=\pi_1\cap\pi_2$. 
By \eqref{angolifacce}, the angle $\theta_2$ between $\pi_2$ and $\pi_3$ satisfies $\theta_2\geq\theta_0$.
By \eqref{angolinterni},  the angle $\theta_1$ between $r$ and $s$ satisfies $\theta_1\geq\theta_0$.
By applying Lemma \ref{lem:geometrico} with $\alpha=\pi_2$, $\beta=\pi_3$, we have that the angle $\alpha_0$ between $s$ and $\pi_3$ satisfies
\begin{equation}
	\label{eq:GEO3.1}
	\begin{aligned}{}
		& 
		\sin \alpha_0\geq \sin^2\theta_0.
	\end{aligned}
\end{equation}

Next, the line $s$ is orthogonal both to $n_1$ and to $n_2$ and therefore to the plane spanned by
$n_1$ and $n_2$. 

On the other hand, the angle between the line $s=\pi_1\cap\pi_2$ and $\pi_3$ coincides  with the angle between any plane orthogonal to $s$ and the vector $n_3$, orthogonal to $\pi_3$. That is, the angle between $n_3$ and the plane  spanned by $n_1$ and $n_2$ is eual to $\alpha_0$.

Choosing a suitable cartesian orthonormal system, $n_1=(x_1,x_2,0)$, $n_2=(y_1,y_2,0)$, $n_3=(z_1,z_2,z_3)$, so that
\begin{equation}
	\label{eq:GEO4.1}
	\begin{aligned}{}
		& 
		|z_3|=\sin\alpha_0.
	\end{aligned}
\end{equation}
Since, by our assumptions, $n_1$, $n_2$ and $n_3$ are linearly independent, \eqref{eq:geom-1.1} holds, that is
\begin{equation}
	\label{eq:GEO4.3}
	\left\{ \begin{array}{l}
		\overline{n}_1 = \alpha_1 x_1 + \alpha_2 y_1 + \alpha_3 z_1,
		\\
		\overline{n}_2 = \alpha_1 x_2 + \alpha_2 y_2 + \alpha_3 z_2,
		\\
		\overline{n}_3= \alpha_3 z_3.
	\end{array}\right.
\end{equation}
We have that
\begin{equation}
	\label{eq:GEO4.4}
	\begin{aligned}{}
		& 
		|\alpha_3| \leq \dfrac{1}{\sin \alpha_0}.
	\end{aligned}
\end{equation}
Moreover, 
\begin{equation}
	\label{eq:GEO4.5-6-7}
	\begin{aligned}{}
		& 
		x_1^2 +x_2^2 = | n_1|^2 =1, \quad 
		y_1^2 +y_2^2 = | n_2|^2 =1, \quad
		x_1 y_1 + x_2 y_2 = n_1 \cdot n_2.
	\end{aligned}
\end{equation}
By adding the first equation in \eqref{eq:GEO4.3} multiplied by $x_1$ and the second equation multiplied by $x_2$, and taking into account \eqref{eq:GEO4.4} and \eqref{eq:GEO4.5-6-7}, we obtain
\begin{equation}
	\label{eq:GEO5.1}
	\begin{aligned}{}
		& 
		\alpha_1 + \alpha_2 (n_1 \cdot n_2) = \overline{n}_1 x_1 + \overline{n}_2 x_2 -
		\alpha_3 (x_1 z_1 +x_2 z_2) :=A,
	\end{aligned}
\end{equation}
and, similarly, 
\begin{equation}
	\label{eq:GEO5.2}
	\begin{aligned}{}
		& 
		\alpha_2 + \alpha_1 (n_1 \cdot n_2) = \overline{n}_1 y_1 + \overline{n}_2 y_2 -
		\alpha_3 (y_1 z_1 +y_2 z_2) :=B.
	\end{aligned}
\end{equation}
By subtracting by \eqref{eq:GEO5.2} the equation \eqref{eq:GEO5.1} multiplied by $n_1 \cdot n_2$, we have
\begin{equation*}
	\begin{aligned}{}
		& 
		\alpha_2 \left (  1- ( n_1 \cdot n_2)^2 \right )=B-A (n_1 \cdot n_2).
	\end{aligned}
\end{equation*}
Noticing that the angle between the normals $n_1$ and $n_2$ is equal to angle between the corresponding faces $F_1$ and $F_2$, and recalling \eqref{angolifacce}, we have 
\begin{equation*}
	\begin{aligned}{}
		& 
		 1- ( n_1 \cdot n_2)^2  \geq  \sin^2 \theta_0,
	\end{aligned}
\end{equation*}
so that
\begin{equation*}
	\begin{aligned}{}
		& 
		|\alpha_2| \leq \dfrac{|B-A(n_1 \cdot n_2)|}{\sin^2 \theta_0}.
	\end{aligned}
\end{equation*}
By the definition of $A$ and $B$, and by \eqref{eq:GEO4.4}, we have
\begin{equation*}
	\begin{aligned}{}
		& 
		|A| \leq |\overline{n} \cdot n_1| + |\alpha_3|  |n_1 \cdot n_3| \leq 1 + \dfrac{1}{\sin \alpha_0},
		\\
		&
		|B| \leq |\overline{n} \cdot n_2| + |\alpha_3|  |n_2 \cdot n_3| \leq 1 + \dfrac{1}{\sin \alpha_0}
	\end{aligned}
\end{equation*}
and, therefore, 
\begin{equation*}
	\begin{aligned}{}
		& 
		|\alpha_2| \leq \dfrac{ 2 \left (    1 +  \dfrac{1}{\sin \alpha_0} \right ) }{\sin^2 \theta_0}.
	\end{aligned}
\end{equation*}
Analogously, 
\begin{equation*}
	\begin{aligned}{}
		& 
		|\alpha_1| \leq \dfrac{ 2 \left (    1 +  \dfrac{1}{\sin \alpha_0} \right ) }{\sin^2 \theta_0}.
	\end{aligned}
\end{equation*}
By \eqref{eq:GEO3.1} and recalling \eqref{eq:GEO4.4}, the thesis follows.
\end{proof}

\section{Proof of the main Theorem}
\label{sec:proof_main_th}

We now assemble the results derived in Section 4, 5, and 6 to derive the final Lipschitz stability estimate.

\begin{proof}[Proof of Theorem \ref{th:main}]
Let $\delta_1=\min\left\{\delta_0, \dfrac{1}{3\widetilde{C}}\right\}$, where
$\delta_0$ and $\widetilde{C}$ have been introduced in Proposition \ref{distvert}, Proposition \ref{vectorfield}, respectively.  Let $\epsilon_0$, $0<\epsilon_0< e^{-c}$, be such that
\begin{equation}
	\label{eq:main.1}
	\begin{aligned}{}
		& 
		\widetilde\omega(\epsilon_0)\leq \delta_1,
	\end{aligned}
\end{equation}
where $c>0$ and the modulus of continuity $\widetilde\omega$ have been introduced in Theorem \ref{Teo:stima_log}.

Let us assume that  
\begin{equation}
	\label{eq:main.2}
	\begin{aligned}{}
		& 
		\epsilon:=r_0\|\Lambda^\Sigma_{D_0}- \Lambda^\Sigma_{D_1}\|_*\leq\epsilon_0,
	\end{aligned}
\end{equation}
so that, by Theorem \ref{Teo:stima_log}, 
\begin{equation}
	\label{eq:main.3}
	\begin{aligned}{}
		& 
		d_H:=d_H(\partial D_0, \partial D_1) \leq\delta_0r_0
	\end{aligned}
\end{equation}
and the results of Section \ref{sec:domain derivative} apply.

Recalling definition \eqref{eq:GAT29.3}, for every $f, g \in H^{1/2}_{co}(\Sigma)$, we have that
\begin{equation}
	\label{eq:main.4}
	\begin{aligned}{}
&
|	F(1,f,g) -F(0,f,g)|= \left|\left \langle \left(\Lambda_{D_1}^\Sigma-\Lambda_{D_0}^\Sigma\right) f, g \right\rangle\right|\leq r_0^2 \|\Lambda^\Sigma_{D_0}- \Lambda^\Sigma_{D_1}\|_*\|f\|_{H^{1/2}_{co}(\Sigma)} \|g\|_{H^{1/2}_{co}(\Sigma)}=\\
&
=r_0 \epsilon \|f\|_{H^{1/2}_{co}(\Sigma)}\|g\|_{H^{1/2}_{co}(\Sigma)}.
\end{aligned}
\end{equation}
On the other hand,
\begin{equation}
	\label{eq:main.5}
	\begin{aligned}{}
	&
	|F(1,f,g) -F(0,f,g)| 
		=\left|F'(0,f,g)+
		\int_0^1( F'(t,f,g)- F'(0,f,g))dt\right|\geq\\
		&
		\geq |F'(0,f,g)|-	\int_0^1|( F'(t,f,g)- F'(0,f,g))|dt.
		\end{aligned}
\end{equation}
By Proposition \ref{prop:Continuità-derivata-di-F},
\begin{equation}
	\label{eq:main.6}
	\begin{aligned}{}
		&
		\int_0^1|( F'(t,f,g)- F'(0,f,g))|dt\leq C
		\frac{d_H^2}{r_0}
		 \|f\|_{H^{1/2}_{co}(\Sigma)}
		\|g\|_{H^{1/2}_{co}(\Sigma)},
	\end{aligned}
\end{equation}
with $C$ depending on the a priori data.

Moreover, by Proposition \ref{prop:LBD1.1}, there exist $f_0,g_0\in H^{1/2}_{co}(\Sigma)$ such that
\begin{equation}
	\label{eq:main.7}
	\begin{aligned}{}
		&
		|F'(0,f,g)|\geq\frac{m_0}{2}d_H
		\|f_0\|_{H^{1/2}_{co}(\Sigma)}
		\|g_0\|_{H^{1/2}_{co}(\Sigma)}.
	\end{aligned}
\end{equation}
By \eqref{eq:main.4}--\eqref{eq:main.6} (with $f=f_0$, $g=g_0$) and by \eqref{eq:main.7}, we have
\begin{equation}
	\label{eq:main.8}
	\begin{aligned}{}
		&
		r_0\epsilon\geq\left(\frac{m_0}{2}- C  \frac{d_H}{r_0}\right)d_H.
	\end{aligned}
\end{equation}
By Theorem \ref{Teo:stima_log} there exists $\epsilon_1>0$, depending on the a apriori data, $\epsilon_1\leq\epsilon_0$, such that, if $\epsilon\leq\epsilon_1$, then 
\begin{equation}
	\label{eq:main.9}
	\begin{aligned}{}
		&
		\frac{m_0}{2}- C  \frac{d_H}{r_0}\geq \frac{m_0}{4},
	\end{aligned}
\end{equation}
so that, by \eqref{eq:main.8},
\begin{equation}
	\label{eq:main.10}
	\begin{aligned}{}
		&
		d_H\leq \frac{4}{m_0}r_0\epsilon=
		\frac{4}{m_0}r_0^2\|\Lambda^\Sigma_{D_0}- \Lambda^\Sigma_{D_1}\|_*.
	\end{aligned}
\end{equation}
If, otherwise, $\epsilon> \epsilon_1$, 
we trivially have
\begin{equation}
	\label{eq:main.11}
	\begin{aligned}{}
		&
		d_H\leq \hbox{diam}(\Omega)\leq M_1 r_0< \dfrac{M_1 r_0}{\epsilon_1}\epsilon
		=
		\dfrac{M_1 }{\epsilon_1}r_0^2
		\|\Lambda^\Sigma_{D_0}- \Lambda^\Sigma_{D_1}\|_*.
	\end{aligned}
\end{equation}

\end{proof}

\section{Appendix}
\label{Appendix}

Throughout this Section, we shall denote by $C$ a positive constant, only depending on $M_0 ,M_1, \theta_0$, which may vary {}from line to line.

\begin{proof}[Proof of Lemma \ref{triangoli_isosceli}]
Let $F$ a face of $D_0$, $\alpha$ the plane containing $F$ and $\sigma$ an edge of $F$ with endpoints $A$, $B$. Let $M$ the midpoint of $\sigma$ and $E= M-h_0n$, where $n$ denotes the unit outer normal to $F$ at $M$ in the plane $\alpha$, and let $T$ the isoscele triangle of vertices $A$, $B$, $E$. Let us show that $T\subset F$. It is sufficient to prove that any  point $y\in T\setminus\sigma$ belongs to $F$. Let $y^\perp$ be the ortogonal projection of $y$ on $\sigma$. By the Lipschitz character of the face $F$ (see \ref{ass:Lipschitz_regularity}), there exists a coordinate system $Ox_1x_2$ in the plane $\alpha$ such that $O=y^\perp$, 
$F \cap R_{r_0,2M_0}=\hbox{epi} (g)$,
where $g(0)=0$, $\|g\|_{{C}^{0,1}([-r_0,r_0])} \leq M_0r_0$, where we recall that $R_{r_0,2M_0}= [-r_0,r_0]\times [-2M_0r_0,2M_0r_0]$.
Since $|y|=|y-y^\perp|\leq r_0$ and recalling that $M_0\geq 1$, we have that $y\in R_{r_0,2M_0}$. Let us see that $y\in \hbox{epi}( g)$. We have that $\sigma\cap R_{r_0,2M_0}$ is represented by the equation $x_2=mx_1$. with $m=\tan\alpha$, $|\alpha|\leq \alpha_0:=\arctan (M_0)$, so that $y=(y_1,y_2)=\frac{|y|}{\sqrt{m^2+1}}(-m,1)$.

If $y_1=0$, then $g(y_1)=0<y_2$ so that $y\in \hbox{epi}( g)$. If $y_1>0$, then either $g(x_1)=mx_1$ in $ [0,r_0]$ or the endpoint $B=(b_1,b_2)$ of $\sigma$ satisfies $0<b_1<r_0$. In the former case $g(y_1)=my_1=- \frac{m^2|y|}{\sqrt{m^2+1}}\leq 0<y_2$, and  $y\in \hbox{epi}( g)$. In the latter case the angle $\beta$ between the line $x_1=b_1$ and $\sigma$ satisfies $\beta=\frac{\pi}{2}-|\alpha|\geq \frac{\pi}{2}-\arctan(M_0)\geq \gamma$, whereas,  by the definition of $h_0$ and by \ref{ass:edge}, the angle between the segment $By$ and $\sigma$ is less or equal to $\gamma$. Therefore $y_1\leq b_1$ and $y\in \hbox{epi}( g)$.
Similarly, we have that  $y\in \hbox{epi}( g)$, assuming $y_1<0$.

Let us consider the collection of all the isoscele triangles $T$ constructed as above on the edges of the face $F$.

Since the internal angles of $F$ are greater or equal to $\theta_0$ and the basis angles of the triangles $T$ so constructed are less or equal to $\frac{\theta_0}{3}$, triangles constructed on adjacent edges of $F$ do not intersect except for the common endpoint of their bases. Let us see that triangles constructed on not adjacent edges are disjoint.

To this aim, let us consider a point  $y\in T\setminus\sigma$, its orthogonal projection $y^\perp$ on $\sigma$ and the coordinate system of origin $O=y^\perp$ as above.

Let $\widetilde T$ be any isoscele triangle constructed on an edge $\widetilde \sigma$ of $F$ with $\sigma$,  $\widetilde \sigma$ not adjacent and let us estimate its distance {}from $y$. If a portion of  $\widetilde \sigma$ is contained in $\hbox{epi} (g)$, then an endpoint, say $\widetilde A$ of  $\widetilde \sigma$ coincides with the endpoint $B^*$ of an edge $\sigma^*=A^*B^*$ fully contained in $\hbox{epi} (g)$. By the Lipschitz character of $g$ and by \ref{ass:edge}, it follows that $b^*_1-a^*_1\geq r_0 \cos(\arctan (M_0))$. By the arguments seen above, $y_1\leq a^*_1$.
Since the angle between $\widetilde \sigma$ and $\sigma^*$ is greater or equal to $\theta_0$ and the basis angle of $\widetilde T$ is less or equal to $ \frac{\theta_0}{3}$, $\widetilde T$ is contained in the half-plane $\{x_1\geq b^*_1\}$ and therefore $\hbox{dist}(y, \widetilde T)\geq  r_0 \cos(\arctan (M_0))$.

If, otherwise, $\widetilde \sigma$ does not intersect $R_{r_0,2M_0}$ then, for every $z=(z_1,z_2)\in \widetilde \sigma$, we have that $|z_2-y_2|\geq 2M_0r_0-h_0$, so that
for any point $w\in\widetilde T$, denoting by $w^\perp$ its orthogonal projection on $\widetilde \sigma$, $|w-y|\geq |w^\perp-y|-h_0\geq 2M_0r_0-2h_0\geq r_0(2M_0-\tan\gamma)\geq r_0(2M_0-\tan(\frac{\pi}{2}-\arctan(M_0)))=r_0\left(2M_0-\frac{1}{M_0}\right)$. Let us notice that, since $M_0\geq 1$, we have that $2M_0-\frac{1}{M_0}\geq  \cos(\arctan (M_0))$. 

We have therefore obtained \eqref{dist_triangoli}.

\end{proof}

\begin{proof}[Proof of Proposition \ref{vectorfield}]

To simplify the notation, given a face $F$ of $D_0$ and an edge $\sigma$ of $F$ with endpoints $A, B$, we shall denote by $F',\sigma', A', B'$ be the corresponding elements of $D_1$. 
 
Let $T'$ the isoscele triangle contained in $F'$ with basis $\sigma'$ and height $h_0$ and let $E'$ be its third vertex.

\textit{Claim. For every face $F$ of $D_0$, there exists a map $\overline{\mathcal{U}}^F:F\rightarrow \mathbb{R}^3$ affine on the triangles $T\in \mathcal{T}_0$ contained in $F$, such that $\overline{\mathcal{U}}^{F}(V)=V'-V$ for any vertex $V$ of $F$, satisfying
\begin{equation}\label{W1_infinitoU_F}
	\|\overline{\mathcal{U}}^F\|_{L^{\infty}(F)} + r_0 \|D\overline{\mathcal{U}}^F\|_{L^{\infty}(F)}\le Cd_H 
\end{equation}
and such that $Id+\overline{\mathcal{U}}^F:F\rightarrow F'$ is a biLipschitz continuous map. }

\begin{proof}[Proof of the Claim]

Let us distinguish three cases, according to the mutual position of the planes containing the two faces $F$, $F'$. 

\medskip
\emph{First case: The faces $F, F'$ are contained in the same plane $\alpha$.} 

\medskip
By \eqref{eq:distvert}, we know that $|A-A'|, |B-B'|\le C_4 d_H$, where $C_4$ only depends on $M_0,M_1,\theta_0$.  Let us first prove that the same estimate holds for the third pair of vertices
\begin{equation}\label{distterzovertice}
	|E-E'|\le C d_H. 
\end{equation}
Let us choose a coordinate system in $\alpha$ such that 
$A=(a_1,0)$, $B=(b_1,0)$, $E=(e_1,e_2),$ with $b_1>a_1$. By elementary computation, 
\[E=\left( \dfrac{a_1+b_1}{2},h_0\right),\quad E'=\left( \dfrac{a'_1+b'_1}{2}- \dfrac{h_0(b_2'-a_2')}{|B'-A'|},\dfrac{a'_2+b'_2}{2}+ \dfrac{h_0(b_1'-a_1')}{|B'-A'|}\right).\]

We have that 
\begin{eqnarray*}
&&(E'-E)_1=\dfrac{a_1'-a_1}{2} + \dfrac{b'_1-b_1}{2}- \dfrac{h_0(b_2'-a_2')}{|B'-A'|}, \\ 
&&(E'-E)_2=\dfrac{a'_2+b'_2}{2}- h_0
\left(1-\dfrac{b_1'-a_1'}{|B'-A'|}\right).
\end{eqnarray*}
Since $|B'-A'|\ge r_0$, by \eqref{h0} and by
\begin{equation}
	\label{b2'-a2'}
|b_2'-a_2'|\leq |b_2'| + |a_2'|=|b_2'-b_2| + |a_2'-a_2|\le Cd_H,
\end{equation}
it follows that
\begin{equation}
| (E'-E)_1|\le Cd_H\left(1+ \tan\left(\frac{\theta_0}{3}\right) \right).
\end{equation}
 It is not restrictive to assume that $Cd_H\le \frac{r_0}{2}$ so that $b_1'-a_1'\geq 0$.  By \eqref{b2'-a2'} and using the inequality $\sqrt{a+b}\leq \sqrt a +\sqrt b$, for $a,b\geq 0$,
 \begin{eqnarray*}
 	&&|(E'-E)_2|\le \frac{|a_2'|}{2} + \frac{|b_2'|}{2} + h_0\frac{|B'-A'|-(b_1'-a_1')}{|B'-A'|}\le\nonumber \\
 	&&\le  Cd_H + \frac{h_0}{r_0}\left( \sqrt{(b_1'-a_1')^2 +(b_2'-a_2')^2 } - \sqrt{(b_1'-a_1')^2}  \right)\le Cd_H.
 \end{eqnarray*}
 
Hence \eqref{distterzovertice} follows.

Let us define $\bar{A}=A'-A , \bar{B}=B'-B , \bar{E}=E'-E$ and let us consider the affine map $\mathcal{U}^T: \mathbb{R}^2 \rightarrow \mathbb{R}^2$ such that $\mathcal{U}^T(A)=\bar{A},\  \mathcal{U}^T(B)=\bar{B},\ \mathcal{U}^T(E)=\bar{E}$ .

Since $\mathcal{U}^T$ is an affine map, namely $\mathcal{U}^T(x)=Mx + v$, with $M\in \mathbb{M}^2$ and $v\in \mathbb{R}^2$,  we have 
\begin{equation}
	\mathcal{U}^T\left(\sum_{i=1}^3 \lambda_i P_i\right)=\sum_{i=1}^3 \lambda_i\,\mathcal{U}^T(P_i), \ \ 0\le \lambda_i\le 1 \ , \ \sum_{i=1}^3 \lambda_i=1 \ ,\ \ P_i\in \mathbb{R}^2 , 
	\end{equation}
and therefore, since $T$ is the convex hull of its vertices, 
\begin{equation}\label{LinftyU}
\|	\mathcal{U}^T \|_{L^{\infty}(T)}\le Cd_H.
\end{equation}
 By further choosing $a_1=0$, that is $A=(0,0)$, and solving the linear system $MA + v =\bar{A}$, $MB + v =\bar{B}$, $ME + v =\bar{E}$, we find 
\[
\begin{array}{ll}
       m_{11}=\frac{\bar{b}_1 -\bar{a}_1}{b_1},   &  m_{12}=\frac{1}{e_2}\left(\bar{e}_1 - \bar{a}_1 + \frac{(\bar{a}_1 -\bar{b}_1)e_1}{b_1} \right),\\
m_{21}=\frac{\bar{b}_2 -\bar{a}_2}{b_1},   &  m_{22}=\frac{1}{e_2}\left(\bar{e}_2 - \bar{a}_2 + \frac{(\bar{a}_2 -\bar{b}_2)e_1}{b_1} \right),\\
v_1= \bar{a}_1, &  v_2= \bar{a}_2.
\end{array}
\]
 We compute 
 \begin{eqnarray}
 &&	\|M\|_2^2 =\sum_{i,j=1}^2 m_{ij}^2=\frac{1}{b_1^2}\left(|\bar{B}|^2 +  |\bar{A}|^2 - 2 \bar B\cdot \bar A  \right) + \nonumber \\
 &&	+ \frac{1}{{e_2}^2}\left(|\bar{E}|^2 +  |\bar{A}|^2 - 2 \bar E\cdot \bar A + \frac{e_1^2}{b_1^2}(|\bar{B}|^2 +  |\bar{A}|^2 - 2 \bar B\cdot \bar A  ) -2\frac{e_1}{b_1}(\bar E -\bar A)\cdot (\bar B -\bar A) \right)\le \\ 
 && \le 4 C^2 d_H^2\left(  \frac{1}{b_1^2} +\frac{1}{e_2^2} +\frac{e_1^2}{e_2^2b_1^2} + \frac{2|e_1|}{b_1e_2^2}  \right)\nonumber.
 	\end{eqnarray}
 Since $b_1=|A-B|\geq r_0$, $|e_1|\leq |A-E|\leq M_1 r_0$, $|e_2|=h_0$, we have 
 \begin{eqnarray}
 	\label{norma di M}
 	\|M\|_2\le C\frac{d_H}{r_0}, 
  \end{eqnarray}
with $C$ only depending on $M_0$, $M_1$, $\theta_0$.  
By  \eqref{LinftyU} and by the above inequality we have 
\begin{eqnarray}
	\|\mathcal{U}^T \|_{L^{\infty}(T)} + r_0	\|D\mathcal{U}^T \|_{L^{\infty}(T)} \le Cd_H \ . 
		\end{eqnarray}
 
Let us apply the above construction to every triangle $T\in \mathcal{T}_0$ which is contained in $F$.  Let $\mathcal{T}_F$ be the union of such triangles and let us define 
 \begin{eqnarray*}
 	\mathcal{U}^{\mathcal{T}_F}\! \! \! \! \! \!&&: \mathcal{T}_F \rightarrow \mathbb{R}^2\\
 	&&\ \ P \mapsto \mathcal{U}^T(P) \quad  \ \hbox{if}\  \ P\in T.
 \end{eqnarray*}
 Let us now prove that $	\mathcal{U}^{\mathcal{T}_F}$ is Lipschitz continuous in $\mathcal{T}_F$. 
 
 Let $x,y\in \mathcal{T}_F$.
 
If $x$ and $y$ belong to the same triangle, then the Lipschitz estimate is trivial. If $x$ and $y$ belong to adjacent triangles having a vertex $V$ in common, then let us notice that 
\begin{equation}
	\begin{aligned}{}
|\mathcal{U}^{\mathcal{T}_F}(x) - \mathcal{U}^{\mathcal{T}_F}(y)|\le |\mathcal{U}^{\mathcal{T}_F}(x) - \mathcal{U}^{\mathcal{T}_F}(V)|+|\mathcal{U}^{\mathcal{T}_F}(V) - \mathcal{U}^{\mathcal{T}_F}(y)|\leq\\
\leq C\frac{d_H}{r_0} \left(|x-V|+ |y-V|\right).
\end{aligned}
\end{equation}
Therefore, the thesis follows if we prove that 
\begin{equation}\label{Elisa}
	|x-V| + |y-V|\le C|x-y|
\end{equation}
 for some positive constant $C$ only depending on $\theta_0$.
 
Let us consider the triangle $xVy$ and let $\alpha$ be the angle in $V$. By our a priori assumption \eqref{angolinterni} and since, by the choice of $h_0$, the basis angles of the isoscele triangles $T \in \mathcal{T}_0$ are smaller than $\frac{\theta_0}{3}$ we have that $\alpha\geq \frac{\theta_0}{3}$.  
Now, if $\alpha\geq \frac{\pi}{2}$ we have $	|x-V| + |y-V|\le 2|x-y|$. 
Otherwise, $\frac{\theta_0}{3}\leq \alpha< \frac{\pi}{2}$, so that  $0<\cos(\alpha)\le 1- \delta_0$, with $\delta_0= 1- \cos\left(\frac{\theta_0}{3}\right)$. 

We compute 
\begin{eqnarray*}
&&	|x-y|^2= |x-V|^2 + |y-V|^2 - 2|x-V|\cdot|y-V|\cos\alpha\ge\\
&&	\ge |x-V|^2 + |y-V|^2 - 2(1-\delta_0)|x-V|\cdot|y-V|=\\
&&	= |x-V|^2 + |y-V|^2 - 2|x-V|\cdot|y-V| + 2\delta_0|x-V|\cdot|y-V|=\\
&&	=(|x-V| - |y-V|)^2 + 2\delta_0 |x-V|\cdot|y-V|,
	\end{eqnarray*}
{}from which it follows that 
\begin{eqnarray*}
(|x-V| - |y-V|)^2\le |x-y|^2,\\
|x-V|\cdot|y-V|\le \frac{1}{2\delta_0}|x-y|^2 \ .
\end{eqnarray*}
Now, by the above inequalities we have 
\begin{eqnarray*}
		&(|x-V| + |y-V|)^2 =(|x-V| - |y-V|)^2 + 4 |x-V|\cdot|y-V|\le\\ 
		&\leq \left(1+\frac{2}{\delta_0}\right)|x-y|^2
	\end{eqnarray*}
and \eqref{Elisa} follows.

By \eqref{dist_triangoli}, we know that $|x-y|\ge r_0\cos(\arctan (M_0)$ for every $x$ and $y$ belonging to not adjacent triangles $T_x$ and $T_y$, respectively, belonging to $\mathcal T_0$.  In this case, we can estimate
\begin{equation}
	\label{xylontani}
	|	\mathcal{U}^{\mathcal{T}_F}(x)-	\mathcal{U}^{\mathcal{T}_F}(y)|\le 	|	\mathcal{U}^{\mathcal{T}_F}(x)|+|\mathcal{U}^{\mathcal{T}_F}(y)|\le Cd_H \le C \frac{d_H}{r_0}|x-y|.
\end{equation}
By the Kirszbraun Theorem, there exists a Lipschitz continuous extension $\mathcal{U}^F$ of $\mathcal{U}^{\mathcal{T}_F}$, $\mathcal{U}^F: \mathbb{R}^2 \rightarrow \mathbb{R}^2$, having the same Lipschitz constant. 
It is straightforward to prove that 
\begin{equation}\label{normainfinito}
\|\mathcal{U}^F\|_{L^{\infty}(F^{r_0})}\le Cd_H,
\end{equation}
where 
\begin{equation}
F^{s}=\{x\in \mathbb{R}^2 \ |\ \mbox{dist}(x, F)< s\} \ .
\end{equation}
We now consider the cut-off function

\begin{equation}
		\label{cut-off}
\eta(t)=\left \{
\begin{array}{ll}

       1   & \ \ \ \  0\le t\le \frac{r_0}{8},\\
1- \frac{8}{r_0}\left(t- \frac{r_0}{8} \right)   &  \ \ \ \  \frac{r_0}{8} <t\le \frac{r_0}{4}, \\
0 &  \ \ \ \  t>\frac{r_0}{4}.
\end{array}
\right. 
\end{equation}

The function $\psi(x)= \eta (\mbox{dist}(x,F))$ satisfies the following properties: 
\begin{description}
\item [a)] $0\leq \psi\leq 1$ in $\mathbb{R}^2$,
\item [b)] $\psi \equiv 1$ in $F^{\frac{r_0}{8}} $,
\item [c)] $\psi \equiv 0$ in $\mathbb{R}^2\setminus F^{\frac{r_0}{4}}$,
\item [d)] $|\psi(x)-\psi(y)|\le \frac{8}{r_0}| x-y|$, for every $x,y\in\mathbb{R}^2$. 
\end{description}
Let us define $\overline{\mathcal{U}}^F(x) =\psi(x)\,\mathcal{U}^F (x)$ and let us prove that 
\begin{equation}\label{LipUbar}
|\overline{\mathcal{U}}^F(x)-\overline{\mathcal{U}}^F(y)|\le C \frac{d_H}{r_0}
 |x-y| , \quad \forall  x,y\in\mathbb{R}^2 \ . 
\end{equation}
If $x,y\in \mathbb{R}^2\setminus F^{\frac{r_0}{4}}$ then
\eqref{LipUbar} trivially follows. 
Otherwise, let for instance $x\in F^{\frac{r_0}{4}}$.  We have 
\begin{eqnarray*}
|\overline{\mathcal{U}}^F(x) -\overline{\mathcal{U}}^F(y) |\le |\psi(x)-\psi(y)|\cdot |\mathcal{U}^F (x)| + |\psi(y)|\cdot|\mathcal{U}^F (x)-\mathcal{U}^F (y) | .
\end{eqnarray*}
By combining the above inequality with \eqref{normainfinito}, the Lipschitz regularity of $\psi$ and the global Lipschitz regularity of $\mathcal{U}^F$, \eqref{LipUbar} follows. 

Let $\Phi_1^F= Id +\overline{\mathcal{U}}^F $. 
We observe that $\Phi_1^F$ is a perturbation of the identity map and hence it is injective; in fact, for any $x,y\in \mathbb{R}^2$ we have 
\begin{eqnarray}
		\label{iniettiva}
|\Phi_1^F(x)-\Phi_1^F(y)|\ge |x-y|-|\overline{\mathcal{U}}^F(x)-\overline{\mathcal{U}}^F(y)| \geq \left(1-\frac{Cd_H}{r_0}\right)|x-y|
\geq \frac{|x-y|}{2},
\end{eqnarray}
for $ \frac{d_H}{r_0}\leq \frac{1}{2C}$. 
Let us notice that 
\begin{equation}
		\label{limitbarU}
\hbox{supp}(\overline{\mathcal{U}}^F)\subset F^{\frac{r_0}{4}}, \quad \| \overline{\mathcal{U}}^F\|_{L^{\infty}(\mathbb{R}^2)}\le Cd_H \ .
\end{equation} 

Let us now prove that $\Phi_1^F$ is surjective. Given $y\in\mathbb{R}^2$, we consider the continuous map 
\begin{eqnarray*}
H_y :&& \!\!\!\!\mathbb{R}^2\rightarrow  \mathbb{R}^2\\
&&  \!\!\!\!x\mapsto y- \overline{\mathcal{U}}^F(x)
\end{eqnarray*}
For $z\in B_{Cd_H}(y)$, by\eqref{limitbarU} we have 
\begin{equation*}
|H_{y}(z)-y|=| \overline{\mathcal{U}}^F(z)|\le Cd_H \ .
\end{equation*}
{}From this it follows  that $H_{y}(B_{Cd_H}(y))\subset B_{Cd_H}(y)$ . By the Brouwer fixed point theorem there exists $x\in B_{Cd_H}(y)$ such that $H_{y}(x)=x$, that is 
$y=x+\overline{\mathcal{U}}^F(x)=\Phi_1^F(x)$.  Therefore $\Phi_1^F$ is bijective and, by \eqref{iniettiva}, it is biLipschitz continuous.

Now, $\Phi_1^F$ is an affine map in each triangle $T\in \mathcal{T}_0$ contained in $F$, and it maps the vertices of $F$ into the corresponding vertices of $F'$. Therefore, we have that 
\begin{equation}
		\label{bordo nel bordo}
\Phi_1^F(\partial F)=\partial F'.
\end{equation}
Let us see that  $\Phi_1^F(F)=F'$. 

Given $x\in Int(F)$, let $E$ be the vertex internal to $F$ of a triangle $T\in \mathcal{T}_0$ contained in $F$. Let $\gamma(t)$, $t\in[0,1]$, be an arc contained in $Int(F)$ such that $\gamma(0)=x$, $\gamma(1)=E$. The set $\Phi_1^F(\gamma([0,1])$ is a connected set containing  $\Phi_1^F(x)$ and  $\Phi_1^F(E)=E'\in Int(F')$.
 By \eqref{bordo nel bordo} and by the injectivity of the map $\Phi_1^F$,  $\Phi_1^F(x)\not \in \partial F'$. If $\Phi_1^F(x)\not \in F'$, then there exists $t\in(0,1)$ such that 
 $\Phi_1^F(\gamma(t))\in \partial F'$, obtaining a contradiction.
 We have therefore proved that $\Phi_1^F(Int(F))\subset Int(F')$.
 By repeating these arguments for the map  $(\Phi_1^F)^{-1}$, we obtain that
 $(\Phi_1^F)^{-1}(Int(F'))\subset Int(F)$. It follows that 
 $\Phi_1^F(Int(F))=Int(F')$ and, by recalling \eqref{bordo nel bordo}, 
 $\Phi_1^F(F)=F'$. 

\medskip
\emph{Second case: the faces $F, F'$ are contained in parallel planes $\alpha$, $\alpha'$.}

Let us choose a coordinate system in $\mathbb{R}^3$ in which the plane $\alpha$ is represented by $z=0$ and the plane $\alpha'$ is represented by $z=h$. By \eqref{eq:distvert} we have that $h\leq Cd_H$, with $C>0$ only depending on $M_0$, $M_1$, $\theta_0$.
Let $w=(0,0,-h)$ and let $T_w$ be the translation of vector $w$. We have that $F'':=T_w(F')\subset \{z=0\}$. Denoting $P'':=T_w(P')$ for any point $P'\in \alpha'$, we have that
$|V_i''-V_i|\leq |V_i''-V_i'|+|V_i'-V_i|\leq 2Cd_H$,
for any vertex $V_i'$ of $D_1$ corresponding to the vertex $V_i$ of $D_0$.
We can therefore apply the results established in the first case to the faces $F$, $F''$, obtaining that there exists a map
$\overline{\mathcal{U}}^{F''}: \{z=0\}\rightarrow \{z=0\}$, affine on each triangle $T\in \mathcal{T}_0$ contained in $F$, mapping its vertices $A$, $B$, $E$ in 
$A''-A$, $B''-B$, $E''-E$, respectively, and satisfiying
\begin{eqnarray}
		\label{suppF''}
	\hbox{supp}\,(\overline{\mathcal{U}}^{F''})\subset F^{r_0},
\end{eqnarray}
\begin{eqnarray}
	\label{normF''}
	\|\overline{\mathcal{U}}^{F''} \|_{L^{\infty}(\mathbb{R}^2)} + r_0	\|D\overline{\mathcal{U}}^{F''} \|_{L^{\infty}(\mathbb{R}^2)} \le Cd_H \ . 
\end{eqnarray}
Moreover the map $\Phi_1^{F''}:=Id+\overline{\mathcal{U}}^{F''}$ is an invertible, biLipschitz continuous map satisfying $\Phi_1^{F''}(F)=F''$, $\Phi_1^{F''}(\partial F)=\partial F''$.

Then, the maps $\overline{\mathcal{U}}^{F}=T_{-w}\circ\overline{\mathcal{U}}^{F''}$ and $\Phi_1^{F}=T_{-w}\circ \Phi_1^{F''}$ satisfy the properties of  $\mathcal{U}^{F''}$ and $ \Phi_1^{F''}$, respectively.

\medskip
\emph{Third case: the faces $F, F'$ are contained in planes $\alpha$, $\beta$ intersecting along the line r.}

Let $\varphi$ the angle between the planes $\alpha$ and $\beta$. By the a priori assumptions \eqref{lunghlati} and \eqref{angolinterni}, there exists a vertex $V'$ of $F'$ in the plane $\beta$ such that $\hbox{dist}(V',r)\geq \frac{r_0}{2}\sin \theta_0$. By \eqref{eq:distvert}, we have that

$$
C_4d_H\geq |V-V'|\geq \hbox{dist}(V',\alpha)= \hbox{dist}(V',r)\sin \varphi\geq 
\frac{r_0}{2}\sin \theta_0\sin \varphi.
$$
It follows that $\sin\varphi<1$ if $\frac{d_H}{r_0}<\frac{\sin \theta_0}{2C_4}$.
Therefore $0<\varphi<\frac{\pi}{2}$ and $1-\cos\varphi\leq \sin \varphi\leq C\frac{d_H}{r_0}$.

Let us choose a cartesian coordinate system in $\mathbb{R}^3$ in which $Oxy$ coincides with the plane $\alpha$, the $y$-axis coincides with the line $r$ and the line $\beta\cap Oxz$ is contained in $\{x>0,z>0\}\cup \{x<0,z<0\}$.

Let $R_{-\varphi}$ the rotation around the line $r$ represented by the matrix
\begin{equation}
	\label{R-varphi}
	R_{-\varphi}=\left(
	\begin{array}{ccc}
		
		\cos\varphi  & 0    & \sin\varphi\\
		0   &  1  & 0 \\
		-\sin\varphi &  0  & \cos\varphi
	\end{array}
	\right) 
\end{equation}
and mapping $\beta$ into $\alpha$. Let $F'':=R_{-\varphi}(F')$, $T''=R_{-\varphi}(T')$ and, in general, for any $P'\in\beta$ let us denote $P'':=R_{-\varphi}(P')$.

Given $P'=(x_{P'},y_{P'},z_{P'})\in \beta$, let $O'$ be the intersection of the plane orthogonal to the line $r$ through $P'$ with the line $r$. Let $P\in\alpha$ be such that $|P'-P|<Cd_H$ and let us estimate $|P''-P|$. 
To fix ideas, let $x_{P'}>0$. 

We have that $P=(x_P,y_P,0)$, $P''=(x_{P''},y_{P''},0)$, with $y_{P''}=y_{P'}$ and, trivially,  $|y_{P''}-y_P|=|y_{P'}-y_P|$, $|z_{P''}-z_P|=0$.

Since $x_{P''}=|O'-P''|=|O'-P'|$, $x_{P'}=|O'-P'|\cos\varphi=x_{P''}\cos\varphi$, we compute $x_{P''}\sin\varphi=z_{P'}=|z_{P'}-z_P|$ so that $|x_{P''}-x_{P'}|=x_{P''}(1-\cos\varphi)\leq x_{P''}\sin\varphi = z_{P'}=|z_{P'} -z_P| \leq Cd_H$.
Therefore $|x_{P''}-x_P|\leq |x_{P''}-x_{P'}|+|x_{P'}-x_P|\leq 2 C d_H$, so that $|P''-P|\leq \sqrt 5 Cd_H$. Therefore we obtain that $|A''-A|\leq Cd_H$,  $|B''-B|\leq Cd_H$ and, by the first step, also  $|E''-E|\leq Cd_H$.

We can therefore apply the results established in the first case to the faces $F$, $F''$, obtaining that there exists a map
$\overline{\mathcal{U}}^{F''}: \{z=0\}\rightarrow \{z=0\}$, affine on each triangle $T\in \mathcal{T}_0$ contained in $F$, mapping its vertices $A$, $B$, $E$ in 
$A''-A$, $B''-B$, $E''-E$ respectively and satisfiying \eqref{suppF''}, \eqref{normF''}.
Moreover the map $\Phi_1^{F''}:=Id+\overline{\mathcal{U}}^{F''}$ is an invertible, bilipschitz continuous map satisfying $\Phi_1^{F''}(F)=F''$, $\Phi_1^{F''}(\partial F)=\partial F''$.

Then the map $\Phi_1^{F}=R_{\varphi}\circ \Phi_1^{F''}$ satisfy the properties of $ \Phi_1^{F''}$. Let us define $\overline{\mathcal{U}}^F=\Phi_1^{F}-Id$. Let us notice that $\|R_{\varphi}\|\leq \sqrt 3$,  $\|R_{\varphi}-I_3\|\leq 2\sin\varphi\leq C\frac{d_H}{r_0}$.
The map $\overline{\mathcal{U}}^F$ is affine in the triangles $T\in \mathcal{T}_0$ contained in the face $F$, satisfies $\overline{\mathcal{U}}^F(V)=V'-V$, for any vertex $V$ of $F$. Moreover
\begin{equation}
	\label{LipUruotata}
	\begin{aligned}{}
		&
		|\overline{\mathcal{U}}^F(x)-\overline{\mathcal{U}}^F(y)|=|(R_{\varphi}\circ \Phi_1^{F''}-Id)(x)- (R_{\varphi}\circ \Phi_1^{F''}-Id)(y)|=\\
		&=|(R_{\varphi}\circ (Id+\overline{\mathcal{U}}^{F''})-Id)(x)- (R_{\varphi}\circ (Id+\overline{\mathcal{U}}^{F''})-Id)(y)|=\\
		&=|(R_{\varphi}-Id)(x-y)+R_{\varphi} (\overline{\mathcal{U}}^{F''}(x)- \overline{\mathcal{U}}^{F''}(y))|\leq\\
		&\leq \|R_{\varphi}-I_3\||x-y|+\|R_{\varphi}\||\overline{\mathcal{U}}^{F''}(x)- \overline{\mathcal{U}}^{F''}(y)|\leq C\frac{d_H}{r_0}|x-y|.
	\end{aligned}
\end{equation}
For any $x\in F$, let $V$ be a vertex of $F$. By the above inequality and since $|x-V|\leq M_1r_0$, we have
\begin{equation}
	\label{boundUruotata}
	\begin{aligned}{}
		&
		|\overline{\mathcal{U}}^{F}(x)|\leq |\overline{\mathcal{U}}^{F}(x)-\overline{\mathcal{U}}^{F}(y)|+ |\overline{\mathcal{U}}^{F}(V)|\leq C\frac{d_H}{r_0}|x-V|+|V'-V|\leq Cd_H.
	\end{aligned}
\end{equation}

\end{proof}
\medskip

Let us define
 \begin{eqnarray*}
	\mathcal{U}^{\partial D_0}\! \! \! \! \! \!&&: \partial D_0 \rightarrow \mathbb{R}^3\\
	&&\ \ P \mapsto \mathcal{U}^F(P) \ ,  \ \hbox{if}\  \ P\in F.
\end{eqnarray*}
By \eqref{W1_infinitoU_F}, $\|\mathcal{U}^{\partial D_0}\|_{L^{\infty}(\partial D_0)}\le Cd_H$.
Let us now prove that $	\mathcal{U}^{\partial D_0}$ is Lipschitz continuous in $\partial D_0$. 

Let $x,y\in \partial D_0$. 

If $x$ and $y$ belong to faces $F_x$, $F_y$ respectively, having a common edge $\sigma$, let $P\in \sigma$ be the point in $\sigma$ realizing the minimum distance {}from $x$. Let $\alpha$ be the angle between the segments $Px$ and $Py$,  
$\beta$ the angle between $Px$ and $\sigma$ and $\gamma$ the angle between $Px$ and the plane containing $F_y$.
If $P$ is not an endpoint of $\sigma$, then $Px$ is orthogonal to $\sigma$, so that $\beta=\frac{\pi}{2}$. If, otherwise, $P$ is an endpoint of $\sigma$, then, by the a priori information, $\beta\geq\theta_0$. 

By Lemma \ref{lem:geometrico} we have $\gamma \geq \arcsin(\sin^2\theta_0)$. Since the angle between a line $s$ and a plane minimizes the angle between $s$ and any line of that plane, it follows that $\alpha\geq \gamma\geq \arcsin(\sin^2\theta_0)$ and the wished estimate follows arguing as in Step 1.

Let us consider now the case in which the intersection of the faces containing $x$ and $y$ is a vertex $V$. Let us consider the triangle $xVy$ and let $\alpha$ its angle in $V$. If both $x$ and $y$ belong to $B_{r_0}(V)$, then by the Lipschitz character of $D_0$, $\alpha\geq 2(\frac{\pi}{2}-\arctan \theta_0)$ and again the estimate follows arguing as in Step 1.

If, otherwise, at least one among $x$ and $y$ does not belong to $B_{r_0}(V)$, then, again by the Lipschitz character of $D_0$, $|x-y|\geq 2r_0\cos(\arctan M_0)$ and in this case, we argue as in \eqref{xylontani}.

Next, let  $x$ and $y$ belonging to disjoint faces $F_x$ and $F_y$, respectively.
By the Lipschitz character of $D_0$, there exists a coordinate system $Ox_1x_2x_3$ such that $O=x$, 
$\partial D_0 \cap R_{r_0,2M_0}=\hbox{epi} (g)$,
where $g(0,0)=0$, $\|g\|_{{C}^{0,1}([-r_0,r_0]^2 )} \leq M_0r_0$, where we recall that $R_{r_0,2M_0}= [-r_0,r_0]^2 \times[-2M_0r_0,2M_0r_0]$.
Let us consider the plane containing the $x_3$-axis and the point $y$. By following the approach used to derive \eqref{dist_triangoli} in the proof of Lemma \ref{triangoli_isosceli}, it follows that 
\begin{equation}
	\label{dist-facce-non-adia}
|x-y|\ge r_0\cos(\arctan(M_0)),
\end{equation}
and again we argue as in \eqref{xylontani}.
Therefore we have proved that
\begin{eqnarray}
	\label{norm_bordo}
	\|\mathcal{U}^{\partial D_0}\|_{L^{\infty}(\partial D_0)} + r_0	\|D\mathcal{U}^{\partial D_0}\|_{L^{\infty}(\partial D_0)}\le Cd_H \ . 
\end{eqnarray}

Next, the thesis follows by arguing as in Step 1 above: the Kirszbraun Theorem provides a Lipschitz extension $\mathcal{U}^{\mathbb{R}^3}$ of $\mathcal{U}^{\partial D_0}$ to $\mathbb{R}^3$ with the same Lipschitz constant. The cut-off argument produces the map $\mathcal{U}$ having support in the neighborhhod $W$ satisfying the properties stated in Proposition \ref{vectorfield}. Moreover, let us notice that, arguing as above, we obtain that $\Phi_1=Id+\mathcal{U}$ is bijective, biLipschitz continuous and that $\Phi_1(D_0)=D_1$. 

\end{proof}

\begin{proof}[Proof of Proposition \ref{prop:property_Phi}] 
We just know, by the above proof, that $\Phi_1(D_0)=D_1$ and, by the same arguments seen for $\Phi_1$, it follows that $\Phi_t$ is invertible for every $t\in[0,1]$. Since $\Phi_t\equiv Id$ outside $W$, we have that $\Phi_t(\Omega)=\Omega$. 
The estimate \eqref{eq:DU-infty-bound} is an obvious consequence of \eqref{ass4:U} and all the remaining properties follow by straightforward computations.

\end{proof}
\vskip 2.5 cm
\noindent
{\bf{Acknowledgements}}
\vskip 0.3 cm
\noindent
The work of AA, EF,  ES and SV have been supported by Gruppo Nazionale per l’Analisi Matematica, la Probabilit\`a e le loro applicazioni (GNAMPA) by the grant ”Problemi inversi
per equazioni alle derivate parziali”. EB has been partially supported by NYUAD Science Program Project Fund AD364.
 EF, ES and SV are supported by the Italian MUR through the PRIN 2022 project ``Inverse problems in PDE: theoretical and numerical analysis'', project code 2022B32J5C, under the National Recovery and Resilience Plan (PNRR), Italy, funded by the European Union  - Next Generation EU, Mission 4 Component 1 CUP F53D23002710006. The work of AM has been supported by PRIN 2022 n. 2022JMSP2J ”Stability, multiaxial fatigue and fatigue life prediction in statics and
dynamics of innovative structural and material coupled systems” funded by MUR, Italy, and by the European Union – Next Generation EU. AM is a member of the INdAM Gruppo Nazionale di Fisica Matematica.

\bibliography{references}{}

\begin{thebibliography}{10}

\bibitem{alberti2022inverse}
Giovanni~S Alberti, {\'A}ngel Arroyo, and Matteo Santacesaria.
\newblock Inverse problems on low-dimensional manifolds.
\newblock {\em Nonlinearity}, 36(1):734, 2022.

\bibitem{AS2021}
Giovanni~S. Alberti and Matteo Santacesaria.
\newblock Infinite-dimensional inverse problems with finite measurements.
\newblock {\em Arch. Ration. Mech. Anal.}, 243(1):1--31, 2022.

\bibitem{g.alessandrinim.dicristo2005}
G.~Alessandrini and M.~Di~Cristo.
\newblock Stable determination of an inclusion by boundary measurements.
\newblock {\em SIAM J. Math. Anal.}, 37:200--217, 2005.

\bibitem{ADCMR}
G.~Alessandrini, M.~Di~Cristo, A.~Morassi, and E.~Rosset.
\newblock Stable determination of an inclusion in an elastic body by boundary measurements.
\newblock {\em SIAM J. Math. Anal.}, 46(4):2692--2729, 2014.

\bibitem{AM01}
G.~Alessandrini and A.~Morassi.
\newblock Strong unique continuation for the {L}amé system of elasticity.
\newblock {\em Commun. Partial Differ. Equ.}, 26(9):1787--1810, 2001.

\bibitem{ARRV}
G.~Alessandrini, L.~Rondi, E.~Rosset, and S.~Vessella.
\newblock The stability for the {C}auchy problem for elliptic equations.
\newblock {\em Inverse Probl.}, 25(12):123004, 2009.

\bibitem{AlVes}
G.~Alessandrini and S.~Vessella.
\newblock Lipschitz stability for the inverse conductivity problem.
\newblock {\em Adv. in Appl. Math.}, 35(2):207--241, 2005.

\bibitem{ammari2010reconstruction}
H.~Ammari, E.~Beretta, E.~Francini, H.~Kang, and M.~Lim.
\newblock Reconstruction of small interface changes of an inclusion from modal measurements ii: The elastic case.
\newblock {\em J. Math. Pures Appl.}, 94(3):322--339, 2010.

\bibitem{ammari2015mathematical}
Habib Ammari, Elie Bretin, Josselin Garnier, Hyeonbae Kang, Hyundae Lee, and Abdul Wahab.
\newblock {\em Mathematical methods in elasticity imaging}.
\newblock Princeton University Press, 2015.

\bibitem{AspBerFraVes22}
A.~Aspri, E.~Beretta, E.~Francini, and S.~Vessella.
\newblock Lipschitz stable determination of polyhedral conductivity inclusions from local boundary measurements.
\newblock {\em SIAM J. Math. Anal.}, 54(5):5182--5222, 2022.

\bibitem{BaV2006}
V.~Bacchelli and S.~Vessella.
\newblock Lipschitz stability for a stationary 2d inverse problem with unknown polygonal boundary.
\newblock {\em Inv. Prob.}, 22(5), 2006.

\bibitem{barbone2004elastic}
Paul~E Barbone and Nachiket~H Gokhale.
\newblock Elastic modulus imaging: on the uniqueness and nonuniqueness of the elastography inverseproblem in two dimensions.
\newblock {\em Inverse problems}, 20(1):283, 2004.

\bibitem{beretta2012small}
E.~Beretta, E.~Bonnetier, E.~Francini, and A.L. Mazzucato.
\newblock Small volume asymptotics for anisotropic elastic inclusions.
\newblock {\em Inverse Probl. Imaging}, 6(1):1--23, 2012.

\bibitem{beretta2006asymptotic}
E.~Beretta and E.~Francini.
\newblock An asymptotic formula for the displacement field in the presence of thin elastic inhomogeneities.
\newblock {\em SIAM J. Math. Anal.}, 38(4):1249--1261, 2006.

\bibitem{BF}
E.~Beretta and E.~Francini.
\newblock Lipschitz stability for the electrical impedance tomography problem: the complex case.
\newblock {\em Commun. Partial Differ. Equ.}, 36(10):1723--1749, 2011.

\bibitem{BerFraMorRosVes14}
E.~Beretta, E.~Francini, A.~Morassi, E.~Rosset, and S.~Vessella.
\newblock Lipschitz continuous dependence of piecewise constant {L}am\'e{} coefficients from boundary data: the case of non-flat interfaces.
\newblock {\em Inverse Probl.}, 30(12):125005, 2014.

\bibitem{BFV-IPI-2014}
E.~Beretta, E.~Francini, and S.~Vessella.
\newblock Uniqueness and {L}ipschitz stability for the identification of {L}amé parameters from boundary mesurements.
\newblock {\em Inverse Probl. Imaging}, 8(3):611--644, 2014.

\bibitem{BerMicPerSan18}
E.~Beretta, S.~Micheletti, S.~Perotto, and M.~Santacesaria.
\newblock Reconstruction of a piecewise constant conductivity on a polygonal partition via shape optimization in {EIT}.
\newblock {\em J. Comput. Phys.}, 353:264--280, 2018.

\bibitem{bonnet2005inverse}
Marc Bonnet and Andrei Constantinescu.
\newblock Inverse problems in elasticity.
\newblock {\em Inverse problems}, 21(2):R1, 2005.

\bibitem{brown2005variational}
Brian~Malcolm Brown, Mathias Jais, and IW~Knowles.
\newblock A variational approach to an elastic inverse problem.
\newblock {\em Inverse Problems}, 21(6):1953, 2005.

\bibitem{DiCristoRondi}
M.~Di Cristo and L.~Rondi.
\newblock Examples of exponential instability for inverse inclusion and scattering problems.
\newblock {\em Inverse Probl.}, 19(3):685--701, 2003.

\bibitem{dambrine2024robust}
Marc Dambrine and Viacheslav Karnaev.
\newblock Robust obstacle reconstruction in an elastic medium.
\newblock {\em Discrete and Continuous Dynamical Systems-Series B}, 29(1):151--173, 2024.

\bibitem{deHoopetal}
M.~V. de~Hoop, L.~Qiu, and O.~Scherzer.
\newblock An analysis of a multi-level projected steepest descent iteration for nonlinear inverse problems in {B}anach spaces subject to stability constraints.
\newblock {\em Numer. Math.}, 129(1):127--148, 2015.

\bibitem{deHQS}
Maarten~V. de~Hoop, Lingyun Qiu, and Otmar Scherzer.
\newblock Local analysis of inverse problems: {H}\"{o}lder stability and iterative reconstruction.
\newblock {\em Inverse Problems}, 28(4):045001, 16, 2012.

\bibitem{druskin1982uniqueness}
VL~Druskin.
\newblock The uniqueness of a solution to an inverse problem of resistivity prospecting and resistivity logging for piecewise-constant conductances.
\newblock {\em Izv. Akad. Nauk SSSR, Fiz. Zemli;(USSR)}, 1, 1982.

\bibitem{eberle2021lipschitz}
Sarah Eberle, Bastian Harrach, Houcine Meftahi, and Taher Rezgui.
\newblock Lipschitz stability estimate and reconstruction of lam{\'e} parameters in linear elasticity.
\newblock {\em Inverse Problems in Science and Engineering}, 29(3):396--417, 2021.

\bibitem{Evans}
L.C. Evans.
\newblock {\em Partial Differential Equations}, volume~19 of {\em Graduate Studies in Mathematics}.
\newblock American Mathematical Society, 2010.

\bibitem{FS2023}
S.~Foschiatti and E.~Sincich.
\newblock Stable determination of an anisotropic inclusion in the {S}chrödinger equation from local {C}auchy data.
\newblock {\em Inv. Prob. Imag.}, 17(5):584--613, 2023.

\bibitem{FraMur86}
G.A. Francfort and F.~Murat.
\newblock Homogenization and optimal bounds in linear elasticity.
\newblock {\em Arch. Ration. Mech. Anal.}, 94(4):307--334, 1986.

\bibitem{henrot2018shape}
A.~Henrot and M.~Pierre.
\newblock {\em Shape variation and optimization. A geometrical analysis.}
\newblock EMS Tracts in Mathematics. European Mathematical Society, Zürich, 2018.

\bibitem{hofmann2007green}
S.~Hofmann and S.~Kim.
\newblock The {G}reen function estimates for strongly elliptic systems of second order.
\newblock {\em Manuscripta Math.}, 124(2):139--172, 2007.

\bibitem{v.isakov1988}
V.~Isakov.
\newblock On uniqueness of recovery of a discontinuous conductivity coefficient.
\newblock {\em Commun. Pure Appl. Math.}, 41(12):865--877, 1998.

\bibitem{koch2025instability}
Herbert Koch, Angkana R{\"u}land, and Mikko Salo.
\newblock On instability mechanisms for inverse problems: Ars inveniendi analytica (2025), revised paper no. 1, 94 pp.
\newblock {\em Ars Inveniendi Analytica}, pages 94--pp, 2025.

\bibitem{laurain2020distributed}
A.~Laurain.
\newblock Distributed and boundary expressions of first and second order shape derivatives in nonsmooth domains.
\newblock {\em J. Math. Pures Appl.}, 134:328--368, 2020.

\bibitem{LiNir03}
Y.~Li and L.~Nirenberg.
\newblock Estimates for elliptic systems from composite material.
\newblock {\em Comm. Pure Appl. Math.}, 56(7):892--925, 2003.

\bibitem{MR2016}
A.~Morassi and E.~Rosset.
\newblock Stable determination of an inclusion in an inhomogeneous elastic body by boundary measurements.
\newblock {\em Rend. Istit. Mat. Univ. Trieste}, 48:1--20, 2016.

\bibitem{R08}
L.~Rondi.
\newblock Stable determination of sound-soft polyhedral scatterers by a single measurement.
\newblock {\em Indiana Univ. Math. J.}, 57(3):1377--1408, 2008.

\bibitem{rongved1955force}
L.~Rongved.
\newblock Force interior to one of two joined semi-infinite solid.
\newblock In {\em Proc. 2nd Midwest Conf. Solid Mech.}, pages 1--13, 1955.

\bibitem{segall2010earthquake}
Paul Segall.
\newblock Earthquake and volcano deformation.
\newblock In {\em Earthquake and volcano deformation}. Princeton University Press, 2010.

\bibitem{shi2020numerical}
Jia Shi, Elena Beretta, Maarten~V de~Hoop, Elisa Francini, and Sergio Vessella.
\newblock A numerical study of multi-parameter full waveform inversion with iterative regularization using multi-frequency vibroseis data.
\newblock {\em Computational Geosciences}, 24(1):89--107, 2020.

\bibitem{symes2009seismic}
William~W Symes.
\newblock The seismic reflection inverse problem.
\newblock {\em Inverse problems}, 25(12):123008, 2009.

\end{thebibliography}
\bibliographystyle{plain}

\end{document}